\documentclass[11pt, reqno]{amsart}

\usepackage{graphicx,amsmath,amssymb,epsfig}
\graphicspath{{figures/}{figgroup/}}
\usepackage[usenames,dvipsnames]{color}
\usepackage{multirow}
\usepackage{hyperref}
\usepackage{harvard}
\usepackage{fancybox}

\newtheorem{theorem}{Theorem}

\newtheorem{lemma}{Lemma}
\newtheorem*{lemma*}{Lemma}

\theoremstyle{definition}
\newtheorem{assumption}{}

\newtheorem{remark}{Comment}[section]

\newtheorem*{example*}{Example}
\providecommand{\norm}[1]{\left\Vert#1\right\Vert}
\providecommand{\abs}[1]{\left\vert#1\right\vert}

\providecommand{\be}{\begin{eqnarray}}
\providecommand{\ee}{\end{eqnarray}}

\providecommand{\GG}{\mathcal{G}}

\providecommand{\ba}{\begin{array}}
\providecommand{\ea}{\end{array}}
\providecommand{\bs}{\begin{align}\begin{split}\nonumber}
\providecommand{\bsnumber}{\begin{align}\begin{split}}
\providecommand{\es}{\end{split}\end{align}}

\providecommand{\LL}{ \mathcal{L}}
\providecommand{\M}{ \scriptscriptstyle M}

\renewcommand{\hat}{\widehat}

\providecommand{\N}{\mathcal{N}}
\providecommand{\F}{\mathcal{F}}

\providecommand{\M}{\mathcal{M}}
\providecommand{\G}{\mathbb{G}}

\renewcommand{\N}{\mathcal{N}}
\renewcommand{\F}{\mathcal{F}}

\renewcommand{\M}{\mathcal{M}}
\renewcommand{\G}{\mathbb{G}}

\providecommand{\Gn}{\mathbb{G}_n}

\providecommand{\Pn}{\mathbb{P}_n}

\providecommand{\En}{\mathbb{E}_n}
\providecommand{\Er}{{\mathrm{E}}}
\providecommand{\Ep}{{\mathrm{E}}}
\renewcommand{\Pr}{{\mathrm{P}}}
\providecommand{\Pp}{{\mathrm{P}}}

\providecommand{\argmin}{\operatornamewithlimits{arg\,min}}
\providecommand{\Var}{\mathrm{var}} 
\providecommand{\mineig}{\mathrm{min}\mathrm{eig}}
\providecommand{\indx}{\alpha} 
\providecommand{\indxSet}{\mathcal{A}} 
\providecommand{\critv}{\tau}

\def\tint{\mathop{\textstyle \int}}%
\def\tsum{\mathop{\textstyle \sum }}%
\def\tbigcap{\mathop{\textstyle \bigcap }}%

\begin{document}

\title{Inference for best linear approximations to set identified functions}
\author[Chandrasekhar]{Arun Chandrasekhar}
\address{Department of Economics\\
M. I. T.}
\author[Chernozhukov]{Victor Chernozhukov}
\address{Department of Economics\\
M. I. T.}
\author[Molinari]{Francesca Molinari}
\address{Department of Economics\\
Cornell University}
\author[Schrimpf]{Paul Schrimpf}
\address{Department of Economics\\
University of British Columbia}

\begin{abstract}
This paper provides inference methods for best linear approximations to
functions which are known to lie within a band. It extends the partial
identification literature by allowing the upper and lower functions defining
the band to be any functions, including ones carrying an index, which can be
estimated parametrically or non-parametrically.\ The identification region
of the parameters of the best linear approximation is characterized via its
support function, and limit theory is developed for the latter. We prove
that the support function approximately converges to a Gaussian process and establish
validity of the Bayesian bootstrap. The paper nests as
special cases the canonical examples in the literature: mean regression with
interval valued outcome data and interval valued regressor data. Because the
bounds may carry an index, the paper covers problems beyond mean regression;
the framework is extremely versatile. Applications include quantile and
distribution regression with interval valued data, sample selection
problems, as well as mean, quantile, and distribution treatment effects.
Moreover, the framework can account for the availability of instruments. An
application is carried out, studying female labor force participation along
the lines of \citeasnoun{mulligan_selection_2008}.\newline
\newline
JEL Classification Codes: C13, C31.\newline
AMS Classification Codes: 62G08, 62G09, 62P20.\newline
First draft presented on Dec. 4, 2008. \newline
This version: \today 
\end{abstract}

\keywords{Set identified function; best linear approximation; partial
identification; support function; bayesian bootstrap; convex set.}
\maketitle





\pagebreak{}

\section{Introduction\label{sec:introduction}}

This paper contributes to the literature on estimation and inference for
best linear approximations to set identified functions. Specifically, we
work with a family of functions $f\left( x,\indx\right) $ indexed by some
parameter $\indx\in \indxSet,$ that is known to satisfy $\theta _{0}(x,\indx%
)\leq f(x,\indx)\leq \theta _{1}(x,\indx)\;x-a.s.,$ with $x\in \mathbb{R}%
^{d} $ a vector of regressors. Econometric frameworks yielding such
restriction are ubiquitous in economics and in the social sciences, as
illustrated by Manski (2003, 2007). Cases explicitly analyzed in this paper
include: (1) mean regression; (2) quantile regression; and (3) distribution
and duration regression, in the presence of interval valued data, including
hazard models with interval-valued failure times; (4) sample selection
problems; (5) mean treatment effects; (6) quantile treatment effects; and
(7) distribution treatment effects, see Section \ref{sec:examples} for
details.\footnote{%
For example, one may be interested in the $\indx$-conditional quantile of a
random variable $y$ given $x,$ denoted $Q_{y}\left( \indx|x\right) ,$ but
only observe interval data $\left[ y_{0},y_{1}\right] $ which contain $y$
with probability one. In this case, $f\left( x,\indx\right) \equiv
Q_{y}\left( \indx|x\right) $ and $\theta _{\ell }(x,\indx)\equiv Q_{\ell
}\left( \indx|x\right) ,$ $\ell =0,1,$ the conditional quantiles of properly
specified random variables$.$} Yet, the methodology that we propose can be
applied to virtually any of the frameworks discussed in %
\citeasnoun[2007]{Manski03}. In fact, our results below also allow for
exclusion restrictions that yield intersection bounds of the form $%
\sup_{v\in \mathcal{V}}\theta _{0}(x,v,\indx)\equiv $ $\theta _{0}(x,\indx%
)\leq $ $f(x,\indx)\leq $ $\theta _{1}(x,\indx)\equiv $ $\inf_{v\in \mathcal{%
V}}\theta _{1}(x,v,\indx)\;x-a.s.,$ with $v$ an instrumental variable taking
values in a set $\mathcal{V}.$ The bounding functions $\theta _{0}(x,\indx)$
and $\theta _{1}(x,\indx)$ may be indexed by a parameter $\indx\in \indxSet$
and may be \emph{any} estimable function of $x$.

Our method appears to be the first and currently only method available
in the literature for performing inference on best linear
approximations to set identified functions when the bounding functions
$\theta _{0}(x,\indx)$ and $%
\theta _{1}(x,\indx)$ need to be estimated. Moreover, we allow for the
functions to be estimated both parametrically as well as
non-parametrically via series estimators. Previous closely related
contributions by \citeasnoun{molinari_beresteanu_2008} and
\citeasnoun{bmm} provided inference methods for best linear
approximations to conditional expectations in the presence of interval
outcome data. In that environment, the bounding functions do not need
to be estimated, as the set of best linear approximations can be
characterized directly through functions of moments of the observable
variables. Hence, our paper builds upon and significantly generalizes
their results. These generalizations are our main contribution and
are imperative for many empirically relevant applications.

Our interest in best linear approximations is motivated by the fact that
when the restriction $\theta _{0}(x,\indx)\leq f(x,\indx)\leq \theta _{1}(x,%
\indx)\;x-a.s.$ summarizes all the information available to the researcher,
the sharp identification region for $f(\cdot ,\indx)$ is given by the set of
functions%
\begin{equation*}
\mathfrak{F}\left( \indx\right) =\left\{ \phi (\cdot ,\indx):\theta _{0}(x,%
\indx)\leq \phi (x,\indx)\leq \theta _{1}(x,\indx)\;x-a.s.\right\}
\end{equation*}%
The set $\mathfrak{F}\left( \indx\right) ,$ while sharp, can be difficult to
interpret and report, especially when $x$ is multi-dimensional. Similar
considerations apply to related sets, such as for example the set of
marginal effects of components of $x$ on $f\left( x,\indx\right) .$\
Consequently, in this paper we focus on the sharp set of parameters
characterizing best linear approximations to the functions comprising $%
\mathfrak{F}\left( \indx\right) $. This set is of great interest in
empirical work because of its tractability. Moreover, when the set
identified function is a conditional expectation, the corresponding set of
best linear approximations is robust to model misspecification (%
\citeasnoun{PonomarevaTamer09}).

In practice, we propose to estimate the sharp set of parameter vectors,
denoted $B\left( \indx\right) $, corresponding to the set of best linear
approximations. Simple linear transformations applied to $B\left( \indx%
\right) $ yield the set of best linear approximations to $f\left( x,\indx%
\right) ,$ the set of linear combinations of components of $b\in B(\indx),$
bounds on each single coefficient, etc. The set $B(\indx)$ is especially
tractable because it is a transformation, through linear operators, of the
random interval $\left[ \theta _{0}(x,\indx),\theta _{1}(x,\indx)\right] $,
and therefore is convex. Hence, inference on $B(\indx)$ can be carried out
using its \emph{support function} $\sigma \left( q,\indx\right) \equiv
\sup_{b\in B\left( \indx\right) }q^{\prime }b,$ where $q\in \mathcal{S}^{d-1}
$ is a direction in the unit sphere in $d$ dimensions.\footnote{%
\textquotedblleft The support function (of a nonempty closed convex set $B$
in direction $q$) $\sigma ^{B}\left( q\right) $\ is the signed distance of
the support plane to $B$\ with exterior normal vector $q$\ from the origin;
the distance is negative if and only if $q$ points into the open half space
containing the origin,\textquotedblright\ \citeasnoun[page 37]{Schneider93}.
See \citeasnoun[Chapter 13]{Rockafellar70} or 
\citeasnoun[Section
1.7]{Schneider93} for a thorough discussion of the support function of a
closed convex set and its properties.} \citeasnoun{molinari_beresteanu_2008}
and \citeasnoun{bmm} previously proposed the use of the support function as
a key tool to conduct inference in best linear approximations to conditional
expectation functions. An application of their results gives that the
support function of $B(\indx)$ is equal to the expectation of a function of $%
\left( \theta _{0}(x,\indx),\theta _{1}(x,\indx),x,\Er\left( xx^{\prime
}\right) \right) .$ Hence, an application of the analogy principle suggests
to estimate $\sigma (q,\indx)$ through a sample average of the same
function, where $\theta _{0}(x,\indx)$ and $\theta _{1}(x,\indx)$ are
replaced by parametric or non-parametric estimators, and $E\left( xx^{\prime
}\right) $ is replaced by its sample analog. The resulting estimator,
denoted $\hat{\sigma}(q,\indx),$ yields an estimator for $B(\indx)$ through
the characterization in equation (\ref{eq:convex_set_support_f}) below. We
show that $\hat{\sigma}(q,\indx)$ is a consistent estimator for $\sigma (q,%
\indx)$, uniformly over $q,\indx\in \mathcal{S}^{d-1}\times \indxSet.$ We
then establish the approximate asymptotic Gaussianity of our set estimator.
Specifically, we prove that when properly recentered and normalized, $\hat{%
\sigma}(q,\indx)$ approximately converges to a Gaussian process on $\mathcal{%
S}^{d-1}\times \indxSet$ (we explain below what we mean by ``approximately'').
The covariance function of this process is quite complicated, so we propose
the use of a Bayesian bootstrap procedure for practical inference, and we
prove consistency of this bootstrap procedure.

Because the support function process converges on $\mathcal{S}^{d-1}\times %
\indxSet$, our asymptotic results also allow us to perform inference on
statistics that involve a continuum of values for $q$ and/or $\indx.$ For
example, for best linear approximations to conditional quantile functions in
the presence of interval outcome data, we are able to test whether a given
regressor $x_{j}$ has a positive effect on the conditional $\indx$-quantile
for any $\indx\in \indxSet$.

In providing a methodology for inference, our paper overcomes significant
technical complications, thereby making contributions of independent
interest. First, we allow for the possibility that some of the regressors in 
$x$ have a discrete distribution. In order to conduct test of hypothesis and
make confidence statements, both \citeasnoun{molinari_beresteanu_2008} and %
\citeasnoun{bmm} had explicitly ruled out discrete regressors, as their
presence greatly complicates the derivation of the limiting distribution of
the support function process. By using a simple data-jittering technique, we
show that these complications completely disappear, albeit at the cost of
basing statistical inference on a slightly conservative confidence set.

Second, when $\theta _{0}(x,\indx)$ and $\theta _{1}(x,\indx)$ are
non-parametrically estimated through series estimators, we show that the
support function process is approximated by a Gaussian process that may not
necessarily converge as the number of series functions increases to
infinity. To solve this difficulty, we use a strong approximation argument
and show that each subsequence has a further subsequence converging to a
tight Gaussian process with a uniformly equicontinuous and non-degenerate
covariance function. We can then conduct inference using the properties of
the covariance function.


 To illustrate the use of our estimator, we revisit the analysis of %
 \citeasnoun{mulligan_selection_2008}. The literature studying female
 labor force participation has argued that the gender wage gap has
 shrunk between 1975 and 2001. \citeasnoun{mulligan_selection_2008}
 suggest that women's wages may have grown less than men's wages
 between 1975 and 2001, had their behavior been held constant, but a
 selection effect induces the data to show the gender wage gap
 contracting. They point out that a growing wage inequality within
 gender induces women to invest more in their market productivity. In
 turn, this would differentially pull high skilled women into the
 workplace and the selection effect may make it appear as if
 cross-gender wage inequality had declined.

To test this conjecture they employ a Heckman selection model to correct
married women's conditional mean wages for selectivity and investment
biases. Using CPS repeated cross-sections from 1975-2001 they argue that the
selection of women into the labor market has changed sign, from negative to
positive, or at least that positive selectivity bias has come to overwhelm
investment bias. Specifically, they find that the gender wage gap measured
by OLS decreased from -0.419 in 1975-1979 to -0.256 in 1995-1999. After
correcting for selection using the classic Heckman selection model, they
find that the wage gap was -0.379 in 1975-1979 and -0.358 in 1995-1999,
thereby concluding that correcting for selection, the gender wage gap may
have not shrunk at all. Because it is well known that without a strong
exclusion restriction results of the normal selection model can be
unreliable, Mulligan and Rubinstein conduct a sensitivity analysis which
corroborates their findings.

We provide an alternative approach. We use our method to estimate bounds on
the quantile gender wage gap without assuming a parametric form of selection
or a strong exclusion restriction. We are unable to reject that the gender
wage gap declined over the period in question. This suggests that the
instruments may not be sufficiently strong to yield tight bounds and that
there may not be enough information in the data to conclude that the gender
gap has or has not declined from 1975 to 1999 without strong functional form
assumptions.\medskip

\noindent \textbf{Related Literature.} This paper contributes to a growing
literature on inference on set-identified parameters. Important examples in
the literature include \citeasnoun{andrews_inference_2008}, %
\citeasnoun{AndrewsShi09}, \citeasnoun{Andrews:Soares07},  %
\citeasnoun{molinari_beresteanu_2008}, \citeasnoun{bmm}, \citeasnoun{Bugni10}%
, \citeasnoun{Canay10}, \citeasnoun{chernozhukov_estimation_2007}, %
\citeasnoun{ChernozhukovLeeRosen09}, \citeasnoun{galichon_test_2009}, %
\citeasnoun{Kaido10}, \citeasnoun{romano_inference_2008}, %
\citeasnoun{RomanoShaikh10}, and \citeasnoun{rosen_confidence_2008}, among
others. \citeasnoun{molinari_beresteanu_2008} propose an approach for
estimation and inference for a class of partially identified models with
convex identification region based on results from random set theory.
Specifically, they consider models where the identification region is equal
to the Aumann expectation of a properly defined random set that can be
constructed from observable random variables. Extending the analogy
principle, Beresteanu and Molinari suggest to estimate the Aumann
expectation using a Minkowski average of random sets. Building on the
fundamental insight in random set theory that convex compact sets can be
represented via their support functions (see, e.g., \citeasnoun{Artstein75}%
), Beresteanu and Molinari accordingly derive asymptotic properties of set
estimators using limit theory for stochastic processes. \citeasnoun{bmm}
extend the results of \citeasnoun{molinari_beresteanu_2008} in important
directions, by allowing for incomplete linear moment restrictions where the
number of restrictions exceeds the number of parameters to be estimated, and
extend the familiar Sargan test for overidentifying restrictions to
partially identified models. \citeasnoun{Kaido10} establishes a duality
between the criterion function approach proposed by %
\citeasnoun{chernozhukov_estimation_2007}, and the support function approach
proposed by \citeasnoun{molinari_beresteanu_2008}.

Concurrently and independently of our work, \citeasnoun{kline_interval_2010}
study the sensitivity of empirical conclusions about conditional quantile
functions to the presence of missing outcome data, when the
Kolmogorov-Smirnov distance between the conditional distribution of observed
outcomes and the conditional distribution of missing outcomes is bounded by
some constant $k$ across all values of the covariates. Under these
assumptions, Kline and Santos show that the conditional quantile function is
sandwiched between a lower and an upper band, indexed by the level of the
quantile and the constant $k.$ To conduct inference, they assume that the
support of the covariates is finite, so that the lower and upper bands can
be estimated at parametric rates. Kline and Santos' framework is a special
case of our sample selection example listed above. Hence, our results
significantly extend the scope of Kline and Santos' analysis, by allowing
for continuous regressors. Moreover, the method proposed in this paper
allows for the upper and lower bounds to be non-parametrically estimated by
series estimators, and allows the researcher to utilize instruments. While
technically challenging, allowing for non-parametric estimates of the
bounding functions and for intersection bounds considerably expands the
domain of applicability of our results.\medskip

\noindent \textbf{Structure of the Paper.} This paper is organized as
follows. In Section \ref{sec:framework} we develop our framework, and in
Section \ref{sec:examples} we demonstrate its versatility by applying it to
quantile regression, distribution regression, sample selection problems, and
treatment effects. Section \ref{sec:Inference} provides an overview of our
theoretical results and describes the estimation and inference procedures.
Section \ref{Sec:Empirical} gives the empirical example. Section \ref%
{sec:conclusion} concludes. All proofs are in the Appendix.

\section{The General Framework\label{sec:framework}}

We aim at conducting inference for best linear approximations to the set of
functions%
\begin{equation*}
\mathfrak{F}\left( \indx\right) =\left\{ \phi (\cdot ,\indx):\theta _{0}(x,%
\indx)\leq \phi (x,\indx)\leq \theta _{1}(x,\indx)\;x-a.s.\right\} 
\end{equation*}%
Here, $\indx\in \indxSet$ is some index with $\indxSet$ a compact set, and $x
$ is a column vector in $\mathbb{R}^{d}$. For example, in quantile
regression $\indx$ denotes a quantile; in duration regression $\indx$
denotes a failure time. For each $x,$ $\theta _{0}(x,\indx)$ and $\theta
_{1}(x,\indx)$ are point-identified lower and upper bounds on a true but
non-point-identified function of interest $f(x,\indx).$ If $f(x,\indx)$ were
point identified, we could approximate it with a linear function by choosing
coefficients $\beta (\indx)$ to minimize the expected squared prediction
error \textrm{E}$[(f(x,\indx)-x^{\prime }\beta \left( \indx\right) )^{2}].$
Because $f(x,\indx)$ is only known to lie in $\mathfrak{F}\left( \indx%
\right) ,$ performing this operation for each admissible function $\phi
(\cdot ,\indx)\in \mathfrak{F}\left( \indx\right) $ yields a set of
observationally equivalent parameter vectors, denoted $B\left( \indx\right) $%
:%
\begin{align}
B(\indx)=& \{\beta \in \argmin_{b}\text{\textrm{E}}[(\phi (x,\indx%
)-x^{\prime }b )^{2}]:\mathrm{P}\left( \theta _{0}(x,%
\indx)\leq \phi (x,\indx)\leq \theta _{1}(x,\indx)\right) =1\}  \notag \\
=& \{\beta =\text{\textrm{E}}[xx^{\prime }]^{-1}\text{\textrm{E}}[x\phi (x,%
\indx)]:\mathrm{P}\left( \theta _{0}(x,\indx)\leq \phi (x,\indx)\leq \theta
_{1}(x,\indx)\right) =1\}.  \label{eq: Balpha}
\end{align}%
It is easy to see that the set $B\left( \indx\right) $ is almost surely
non-empty, compact, and convex valued, because it is obtained by applying linear
operators to the (random) almost surely non-empty interval $\left[ \theta
_{0}(x,\indx),\theta _{1}(x,\indx)\right] ,$ see 
\citeasnoun[Section
4]{molinari_beresteanu_2008} for a discussion. Hence, $B\left( \indx\right) $
can be characterized quite easily through its support function%
\begin{equation*}
\sigma \left( q,\indx\right) :=\sup_{\beta \left( \indx\right) \in B\left( %
\indx\right) }q^{\prime }\beta \left( \indx\right) ,
\end{equation*}%
which takes on almost surely finite values $\forall q\in \mathcal{S}%
^{d-1}:=\left\{ q\in \mathbb{R}^{d}:\left\Vert q\right\Vert =1\right\} $, $%
d=\dim \beta $. In fact, 
\begin{equation}
B(\indx)=\tbigcap_{q\in \mathcal{S}^{d-1}}\left\{ b:q^{\prime }b\leq \sigma
\left( q,\indx\right) \right\} ,  \label{eq:convex_set_support_f}
\end{equation}%
see \citeasnoun[Chapter 13]{Rockafellar70}. Note also that $\left[ -\sigma
\left( -q,\indx\right) ,\sigma \left( q,\indx\right) \right] $ gives sharp
bounds on the linear combination of $\beta \left( \indx\right) $'s
components obtained using weighting vector $q$.

More generally, if the criterion for{}\textquotedblleft
best\textquotedblright\ linear approximation is to minimize \textrm{E}$[(f(x,%
\indx)-x^{\prime }\beta (\indx))\tilde{z}^{\prime }W\tilde{z}(f(x,\indx%
)-x^{\prime }\beta (\indx))],$ where $W$ is a $j\times j$ weight matrix and $%
\tilde{z}$ a $j\times 1$ vector of instruments, then we have%
\begin{equation*}
B(\indx)=\{\beta =\mathrm{E}[x\tilde{z}^{\prime }W\tilde{z}x^{\prime }]^{-1}%
\mathrm{E}[x\tilde{z}^{\prime }W\tilde{z}\phi (x,\indx)]:\mathrm{P}\left(
\theta _{0}(x,\indx)\leq \phi (x,\indx)\leq \theta _{1}(x,\indx)\right) =1\}.
\end{equation*}

As in \citeasnoun{bmm}, \citeasnoun{magnac_partial_2008}, and 
\citeasnoun[p.
807]{molinari_beresteanu_2008} the support function of $B(\indx)$ can be
shown to be 
\begin{equation*}
\sigma (q,\indx)=\Er[z_{q}w_{q}]
\end{equation*}%
where 
\begin{eqnarray*}
z &=&x\tilde{z}^{\prime }W\tilde{z},\ z_{q}=q^{\prime }\mathrm{E}[xz^{\prime
}]^{-1}z, \\
w_{q} &=&\theta _{1}(x,\indx)1(z_{q}>0)+\theta _{0}(x,\indx)1(z_{q}\leq 0).
\end{eqnarray*}%
We estimate the support function by plugging in estimates of $\theta _{\ell
},$ $\ell =0,1,$ and taking empirical expectations: 
\begin{equation*}
\hat{\sigma}(q,\indx)=\En\left[ q^{\prime }\left( \En\left[
x_{i}z_{i}^{\prime }\right] \right) ^{-1}z_{i}\left( \hat{\theta}_{1}(x_{i},%
\indx)1(\hat{z}_{iq}>0)+\hat{\theta}_{0}(x_{i},\indx)1(\hat{z}_{iq}\leq
0)\right) \right] ,
\end{equation*}%
where $\En$ denotes the empirical expectation, $\hat{z}_{iq}=q^{\prime
}\left( \En\left[ x_{i}z_{i}^{\prime }\right] \right) ^{-1}z,$ and $\hat{%
\theta}_{\ell }\left( x,\indx\right) $ are the estimators of $\theta _{\ell
}\left( x,\indx\right) $, $\ell =0,1.$

\section{Motivating Examples\label{sec:examples}}

\subsection{Interval valued data}

Analysis of regression with interval valued data has become a canonical
example in the partial identification literature. Interest in this example
stems from the fact that dealing with interval data is a commonplace problem
in empirical work. Due to concerns for privacy, survey data often come in
bracketed form. For example, public use tax data are recorded as the number
of tax payers which belong to each of a finite number of cells, as seen in %
\citeasnoun{Picketty05}. The Health and Retirement Study provides a finite
number of income brackets to each of its respondents, with degenerate
brackets for individuals who opt to fully reveal their income level; see %
\citeasnoun{JusterSuzman95} for a description. The Occupational Employment
Statistics (OES) program at the Bureau of Labor Statistics collects wage
data from employers as intervals, and uses these data to construct estimates
for wage and salary workers in 22 major occupational groups and 801 detailed
occupations.\footnote{%
See \citeasnoun{manski_inferenceregressions_2002} and \citeasnoun{bmm} for
more examples.}

\subsubsection{Interval valued $y$}

\citeasnoun{molinari_beresteanu_2008} and \citeasnoun{bmm}, among others,
have focused on estimation of best linear approximations to conditional
expectation functions with interval valued outcome data. Our framework
covers the conditional expectation case, as well as an extension to quantile
regression wherein we set identify $\beta (\indx)$ across all quantiles $%
\indx\in \indxSet$. To avoid redundancy with the related literature, here we
describe the setup for quantile regression. Let the $\indx$-th conditional
quantile of $y|x$ be denoted $Q_{y}(\indx|x)$. We are interested in a linear
approximation $x^{\prime }\beta (\indx)$ to this function. However, we do
not observe $y$. Instead we observe $y_{0}$ and $y_{1},$ with $\mathrm{P}%
\left( y_{0}\leq y\leq y_{1}\right) =1.$ It is immediate that 
\begin{equation*}
Q_{y_{0}}(\indx|x)\leq Q_{y}(\indx|x)\leq Q_{y_{1}}(\indx|x)\text{ \ \ }%
x-a.s.,
\end{equation*}%
where $Q_{y_{\ell }}(\indx|x)$ is the $\indx$-th conditional quantile of $%
y_{\ell }|x,$ $\ell =0,1.$ Hence, the identification region $B(\indx)$ is as
in equation (\ref{eq: Balpha}), with $\theta _{\ell }(\indx,x)=Q_{y_{\ell }}(%
\indx|x).$

\subsubsection{Interval valued $x_{i}$}

Suppose now that one is interested in the conditional expectation $\mathrm{E}%
\left( y|x\right) ,$ but only observes $y$ and variables $x_{0},x_{1}$ such
that $\mathrm{P}\left( x_{0}\leq x\leq x_{1}\right) =1$. This may occur, for
example, when education data is binned into categories such as primary
school, secondary school, college, and graduate education, while the
researcher may be interested in the return to each year of schooling. It
also happens when a researcher is interested in a model in which wealth is a
covariate, but the available survey data report it by intervals.

Our approach applies to the framework of %
\citeasnoun{manski_inferenceregressions_2002} for conditional expectation
with interval regressors, and extends it to the case of quantile regression.%
\footnote{%
Our methods also apply to the framework of \citeasnoun{magnac_partial_2008},
who study identification in semi-parametric binary regression models with
regressors that are either discrete or measured by intervals. Compared to %
\citeasnoun{manski_inferenceregressions_2002}, Magnac and Maurin's analysis
requires an uncorrelated error assumption, a conditional independence
assumption between error and interval/discrete valued regressor, and a
finite support assumption.} Following Manski and Tamer, we assume that the
conditional expectation of $y|x$ is (weakly) monotonic in $x$, say
nondecreasing, and mean independent of $x_{0},x_{1}$ conditional on $x$.
Manski and Tamer show that 
\begin{equation*}
\sup_{x_{1}\leq x}\mathrm{E}\left( y|x_{0},x_{1}\right) \leq \mathrm{E}%
\left( y|x\right) \leq \inf_{x_{0}\geq x}\mathrm{E}\left(
y|x_{0},x_{1}\right) .
\end{equation*}%
Hence, the identification region $B(\indx)$ is as in equation (\ref{eq:
Balpha}), with $\theta _{0}(\indx,x)=\sup_{x_{1}\leq x}\mathrm{E}\left(
y|x_{0},x_{1}\right) $ and $\theta _{1}(\indx,x)=\inf_{x_{0}\geq x}\mathrm{E}%
\left( y|x_{0},x_{1}\right) .$

Next, suppose that the $\indx $-th conditional quantile of $y|x$ is
monotonic in $x$, say nondecreasing, and that $Q_{y}(\indx
|x,x_{0},x_{1})=Q_{y}(\indx |x).$\ By the same reasoning as above, 
\begin{equation*}
\sup_{x_{1}\leq x}Q_{y}\left( \indx |x_{0},x_{1}\right) \leq Q_{y}(\indx %
|x)\leq \inf_{x_{0}\geq x}Q_{y}\left( \indx |x_{0},x_{1}\right) .
\end{equation*}%
Hence, the identification region $B(\indx )$ is as in equation (\ref{eq:
Balpha}), with $\theta _{0}(\indx ,x)=\sup_{x_{1}\leq x}Q_{y}\left( \indx
|x_{0},x_{1}\right) $ and $\theta _{1}(\indx ,x)=\inf_{x_{0}\geq
x}Q_{y}\left( \indx |x_{0},x_{1}\right) .$

\subsection{Distribution and duration regression with interval outcome data}

Researchers may also be interested in distribution regression with interval
valued data. For instance, a proportional hazard model with time varying
coefficients where the probability of failure conditional on survival may be
dependent on covariates and coefficients that are indexed by time. More
generally, we can consider models in which the conditional distribution of $%
y|x$ is given by%
\begin{equation*}
\mathrm{P}\left( y\leq \indx|x\right) \equiv F_{y|x}(\indx|x)=\Phi \left( f(%
\indx,x)\right)
\end{equation*}%
where $\Phi \left( .\right) $ is a known one-to-one link function. A special
case of this class of models is the duration model, wherein we have $f(\indx%
,x)=g(\indx)+\gamma \left( x\right) $, where $g\left( .\right) $ is a
monotonic function.

As in the quantile regression example, assume that we observe $%
(y_{0},y_{1},x)$ with $\mathrm{P}\left( y_{0}\leq y\leq y_{1}\right) =1$.
Then 
\begin{equation*}
\Phi ^{-1}\left( F_{y_{1}|x}(\indx|x)\right) \leq f(\indx,x)\leq \Phi
^{-1}\left( F_{y_{0}|x}(\indx|x)\right) .
\end{equation*}%
Hence, the identification region $B(\indx)$ is as in equation (\ref{eq:
Balpha}), with $\theta _{\ell }\left( \indx,x\right) =\Phi ^{-1}\left(
F_{y_{1-\ell }|x}(\indx|x)\right) ,$ $\ell =0,1$. A leading example,
following \citeasnoun{HanHausman90} and \citeasnoun{ForesiPeracchi95}, would
include $\Phi $ as a probit or logit link function.

\subsection{Sample Selection}

Sample selection is a well known first-order concern in the empirical
analysis of important economic phenomena. Examples include labor force
participation (see, e.g., \citeasnoun{mulligan_selection_2008}), skill composition of immigrants (see, e.g., %
\citeasnoun{JassoRosenzweig08}), returns to
education (e.g., \citeasnoun{Card99}), program
evaluation (e.g., \citeasnoun{Imbens09}),
productivity estimation (e.g., \citeasnoun{olley_dynamics_1996}),
insurance (e.g., \citeasnoun{Einav11}), models with 
occupational choice and financial intermediation (e.g., 
\citeasnoun{townsend_measuringimpact_2009}). In
Section \ref{Sec:Empirical} we revisit the analysis of 
\citeasnoun{mulligan_selection_2008} who confront selection in the context
of female labor supply.

Consider a standard sample selection model. We are interested in the
behavior of $y$ conditional on $x$; however, we only observe $y$ when $u=1$. %
\citeasnoun{manski_anatomy_1989} observes that the following bound on the
conditional distribution of $y$ given $x$ can be constructed: 
\begin{align*}
F(y|x,u=1)\mathrm{P}(u=1|x)& \leq F(y|x) \leq F(y|x,u=1)\mathrm{P}(u=1|x)+\mathrm{P}(u=0|x).
\end{align*}%
The inverse image of these distribution bounds gives bounds on the
conditional quantile function of $y$%
\begin{eqnarray*}
Q_{0}(\indx|x) &=&%
\begin{cases}
Q_{y}\left( \left. \frac{\indx-\mathrm{P}(u=0|x)}{\mathrm{P}(u=1|x)}%
\right\vert x,u=1\right) & \text{ if }\indx\geq \mathrm{P}(u=0|x) \\ 
y_{0} & \text{ otherwise}%
\end{cases}
\\
Q_{1}(\indx|x) &=&%
\begin{cases}
Q_{y}\left( \left. \frac{\indx}{\mathrm{P}(u=1|x)}\right\vert x,u=1\right) & 
\text{ if }\indx\leq \mathrm{P}(u=1|x) \\ 
y_{1} & \text{ otherwise}%
\end{cases}%
\end{eqnarray*}%
where $y_{0}$ is the smallest possible value that $y$ can take (possibly $%
-\infty $) and $y_{1}$ is the largest possible value that $y$ can take
(possibly +$\infty $). Thus, we obtain 
\begin{equation*}
Q_{0}\left( \indx|x\right) \leq Q_{y}\left( \indx|x\right) \leq Q_{1}\left( %
\indx|x\right) .
\end{equation*}%
and the corresponding set of coefficients of linear approximations to $%
Q_{y}\left( \indx|x\right) $ is as in equation (\ref{eq: Balpha}), with $%
\theta _{\ell }\left( \indx,x\right) =Q_{\ell }\left( \indx|x\right) ,$ $%
\ell =0,1.$

\subsubsection{Alternative Form for the Bounds}

As written above, the expressions for $Q_{0}(\indx|x)$ and $Q_{1}(\indx|x)$
involve the propensity score, $\mathrm{P}\left( u|x\right) $ and several
different conditional quantiles of $y|u=1$. Estimating these objects might
be computationally intensive. However, we show that $Q_{0}$ and $Q_{1}$ can
also be written in terms of the $\indx$-th conditional quantile of a
different random variable, thereby providing computational simplifications.
Define%
\begin{equation}
\tilde{y}_{0}=y1\left\{ u=1\right\} +y_{0}1\left\{ u=0\right\} ,\ \ \ \ \ 
\tilde{y}_{1}=y1\left\{ u=1\right\} +y_{1}1\left\{ u=0\right\} .
\label{eq:change}
\end{equation}%
Then one can easily verify that $Q_{\tilde{y}_{0}}(\indx|x)=Q_{0}(\indx|x),\ 
$and$\ Q_{\tilde{y}_{1}}(\indx|x)=Q_{1}(\indx|x),$ and therefore the bounds
on the conditional quantile function can be obtained without calculating the
propensity score.

\subsubsection{Sample Selection With an Exclusion Restriction}

Notice that the above bounds let $F(y|x)$ and $\mathrm{P}(u=1|x)$ depend on $%
x$ arbitrarily. However, often when facing selection problems researchers
impose exclusion restrictions. That is, the researcher assumes that there
are some components of $x$ that affect $\mathrm{P}(u=1|x)$, but not $F(y|x)$%
. Availability of such an instrument, denoted $v$, can help shrink the
bounds on $Q_{y}(\indx|x)$. For concreteness, we replace $x$ with $(x,v)$
and suppose that $F(y|x,v)=F(y|x)$ $\forall v\in supp(v|x)$. By the same
reasoning as above,\ for each $v\in supp(v|x)$ we have the following bounds
on the conditional distribution function: 
\begin{equation*}
F(y|x,v,u=1)\mathrm{P}(u=1|x,v)\leq F(y|x)\leq F(y|x,v,u=1)\mathrm{P}%
(u=1|x,v)+\mathrm{P}(u=0|x,v).
\end{equation*}%
This implies that the conditional quantile function satisfies: 
\begin{equation*}
Q_{0}(\indx|x,v)\leq Q_{y}(\indx|x)\leq Q_{1}(\indx|x,v)\ \forall v\in
supp(v|x),
\end{equation*}%
and therefore%
\begin{equation*}
\sup_{v\in supp(v|x)}Q_{0}(\indx|x,v)\leq Q_{y}(\indx|x)\leq \inf_{v\in
supp(v|x)}Q_{1}(\indx|x,v)
\end{equation*}%
where $Q_{\ell }(\indx|x,v),$ $\ell =0,1,$\ are defined similarly to the
previous section with $x$ replaced by $(x,v).$ Once again, we can avoid
computing the propensity score by constructing the variables $\tilde{y}%
_{\ell },$ $\ell =0,1$ as in equation (\ref{eq:change}). Then $Q_{\tilde{y}%
_{\ell }}(\indx|x,v)=Q_{\ell }(\indx|x,v).$ Inspecting the formulae for $%
Q_{\ell }(\indx|x,v),$ $\ell =0,1,$ reveals that $Q_{0}(\indx|x,v)$ can only
be greater than $y_{0}$ when $1-\mathrm{P}(u=1|x,v)<\indx$, and $Q_{1}(\indx%
|x,v)$ can only be smaller than $y_{1}$ when $\indx<\mathrm{P}(u=1|x,v)$.
Thus, both bounds are informative only when 
\begin{equation*}
1-\mathrm{P}(u=1|x,v)<\indx<\mathrm{P}(u=1|x,v).
\end{equation*}%
From this we see that the bounds are more informative for central quantiles
than extreme ones. Also, the greater the probability of being selected
conditional on $x$, the more informative the bounds are. If $P\left(
u=1|x,v\right) =1$ for some $v\in supp(v|x)$ then $Q_{y}(\indx|x)$ is
point-identified. This is the large-support condition required to show
non-parametric identification in selection models. If $\mathrm{P}%
(u=1|x,v)<1/2 $ the upper and lower bounds cannot both be informative.

It is important to note that $Q_{\ell }(\indx|x,v),$ $\ell =0,1$ depend on
the quantiles of $y$ conditional on both $x$ and $v$. Moreover, $Q_{y}(\indx%
|x,v,u=1)$ is generally not linear in $x$ and $v$, even in the special cases
where $Q_{y}(\indx|x)$ is linear in $x$. Therefore, in practice, it is
important to estimate $Q_{y}(\indx|x,v,u=1)$ flexibly. Accordingly, our
asymptotic results allow for series estimation of the bounding functions.

\subsection{Average, Quantile, and Distribution Treatment Effects}

Researchers are often interested in mean, quantile, and distributional
treatment effects. Our framework easily accommodates these examples. Let $%
y_{i}^{C}$ denote the outcome for person $i$ if she does not receive
treatment, and $y_{i}^{T}$ denote the outcome for person $i$ if she receives
treatment. The methods discussed in the preceding section yield bounds on
the conditional quantiles of these outcomes. In turn, these bounds can be
used to obtain bounds on the quantile treatment effect as follows:%
\begin{eqnarray*}
\sup_{v\in supp(v|x)}Q_{0}^{T}(\indx|x,v) &-&\inf_{v\in supp(v|x)}Q_{1}^{C}(%
\indx|x,v) \\
&\leq &Q_{Y^{T}}(\indx|x)-Q_{Y^{C}}(\indx|x) \\
&\leq &\inf_{v\in supp(v|x)}Q_{1}^{T}(\indx|x,v)-\sup_{v\in
supp(v|x)}Q_{0}^{C}(\indx|x,v).
\end{eqnarray*}

As we discussed in the previous section, $Q_{0}^{T}(\indx|x,v)$ and $%
Q_{1}^{T}(\indx|x,v)$ are both informative only when 
\begin{equation}
1-P(u=1|x,v)<\indx<P(u=1|x,v).  \label{1i}
\end{equation}%
Similarly, $Q_{0}^{C}(\indx|x,v)$ and $Q_{1}^{C}(\indx|x,v)$ are both
informative only when 
\begin{equation}
P(u=1|x,v)<\indx<1-P(u=1|x,v).  \label{2i}
\end{equation}%
Note that inequalities (\ref{1i}) and (\ref{2i}) cannot both hold. Thus, we
cannot obtain informative bounds on the quantile treatment effect without an
exclusion restriction. Moreover, to have an informative upper and lower
bound on $Q_{Y^{T}}(\indx|x)-Q_{Y^{C}}(\indx|x)$, the excluded variables, $v$%
, must shift the probability of treatment, $P(u=1|x,v)$ sufficiently for
both (\ref{1i}) and (\ref{2i}) to hold at $x$ (for different values of $v$ ).

Analogous bounds apply for the distribution treatment effect and the mean
treatment effect.\footnote{%
Interval regressors can also be accommodated, by merging the results in this
Section with those in Section 3.1.}

\section{Estimation and Inference\label{sec:Inference}}

\subsection{Overview of the Results}

This section provides a less technically demanding overview of our results,
and explains how these can be applied in practice. Throughout, we use sample
selection as our primary illustrative example. As described in section \ref%
{sec:framework}, our goal is to estimate the support function, $\sigma
(q,\indx )$. The support function provides a convenient way to compute
projections of the identified set. These can be used to report upper and
lower bounds on individual coefficients and draw two-dimensional
identification regions for pairs of coefficients. For example, the bound for
the $k$th component of $\beta (\indx)$ is $[-\sigma (-e_{k},\indx),\sigma
(e_{k},\indx)]$, where $e_{k}$ is the $k$th standard basis vector (a vector
of all zeros, except for a one in the $k$th position). Similarly, the bound
for a linear combination of the coefficients, $q^{\prime }\beta (\indx)$, is 
$[-\sigma (-q,\indx),\sigma (q,\indx)]$. Figures \ref{fig:idRegion} provides
an illustration.
In this example, $\beta $ is three dimensional. The left panel shows the
entire identified set. The right panel shows the joint identification region
for $\beta _{1}$ and $\beta _{2}$. The identified intervals for $\beta _{1}$
and $\beta _{2}$ are also marked in red on the right panel.

\begin{figure}[tbp]
\caption{Identification region and its projection }
\label{fig:idRegion}%
\begin{minipage}{\linewidth}
    \begin{center}
      \begin{tabular}{cc}
        \includegraphics[width=0.55\textwidth]{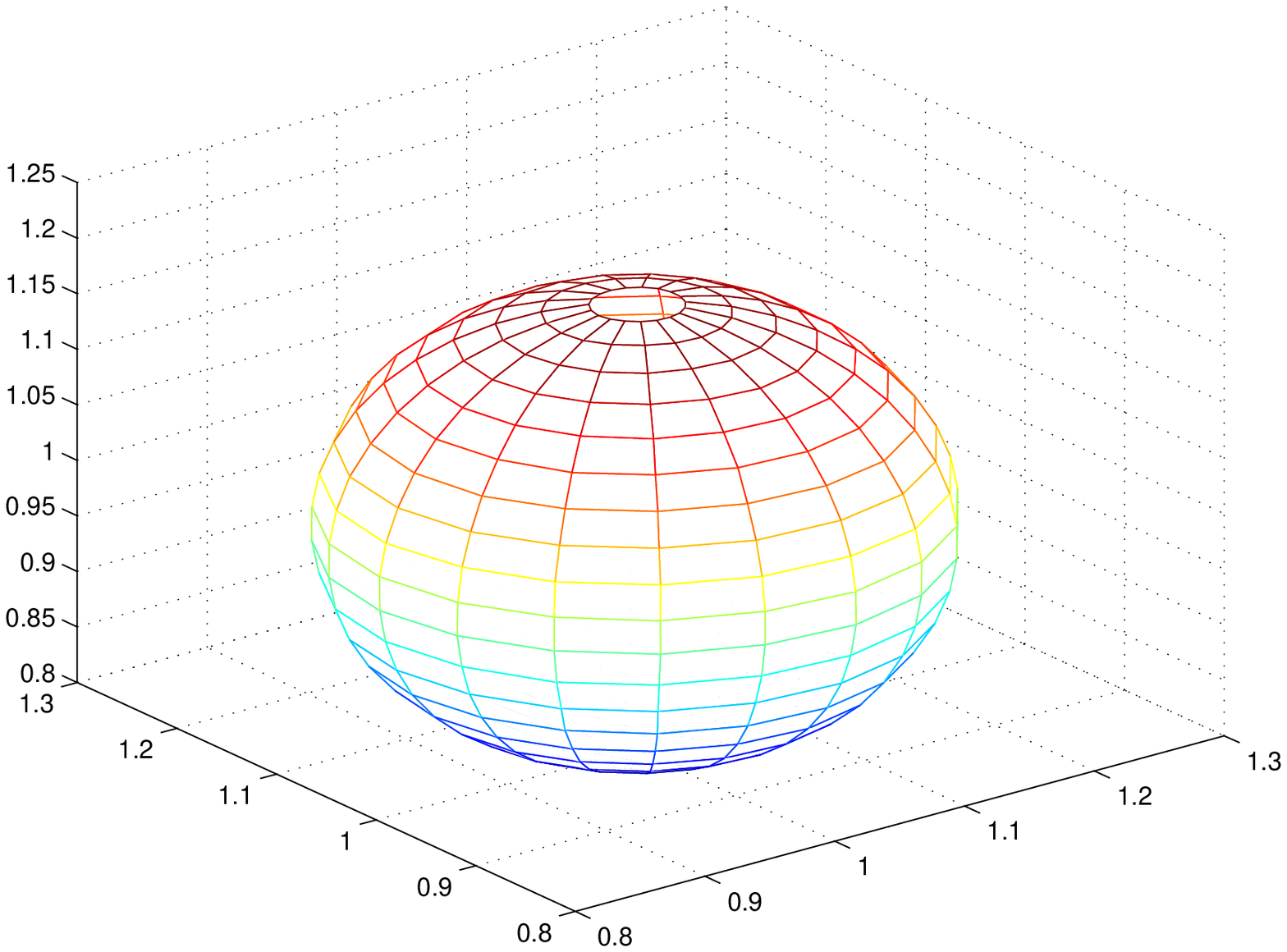} & %
        \includegraphics[width=0.4\textwidth]{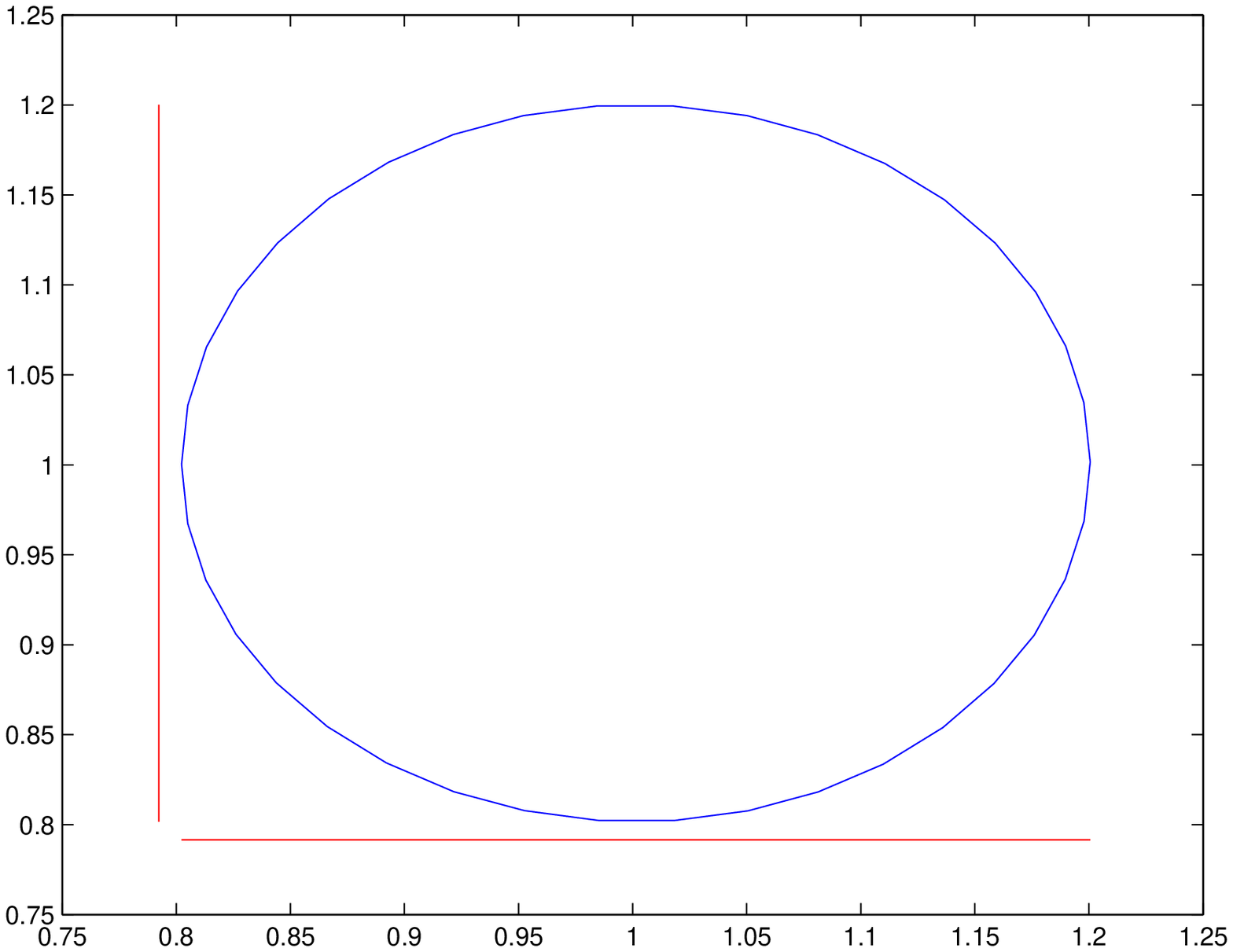}%
      \end{tabular}%
    \end{center}
  \end{minipage}
\end{figure}

Suppose that $x=[1;x_{1}],$ with $x_{1}$ a scalar random variable, so $\beta
(\indx)=%
\begin{bmatrix}
\beta _{0}(\indx) & \beta _{1}(\indx)%
\end{bmatrix}%
$ is two-dimensional. To simplify notation, let $z=x.$ In most applications, 
$\beta _{1}(\indx)$ is the primary object of interest. \citeasnoun{Stoye07}, %
\citeasnoun{molinari_beresteanu_2008} and \citeasnoun{bmm} give explicit
formulae for the upper and lower bound of $\beta _{1}(\indx)$. Recall that
the support function is given by: 
\begin{equation*}
\sigma (q)=q^{\prime }\Er[xx']^{-1}\Er\left[ x\left( \theta _{1}(x,\indx){1}%
\left\{ q^{\prime }\Er[xx']^{-1}x>0\right\} +\theta _{0}(x,\indx){1}\left\{
q^{\prime }\Er[xx']^{-1}x<0\right\} \right) \right]
\end{equation*}%
Setting $q=%
\begin{pmatrix}
0 & \pm 1%
\end{pmatrix}%
$ yields these bounds as follows: 
\begin{align*}
\underline{\beta }_{1}(\indx)=& \frac{\Er\left[ (x_{1i}-\Er[x_{1i}])\left(
\theta _{1i}{1}\left\{ x_{1i}<\Er[x_{1i}]\right\} +\theta _{0i}{1}\left\{
x_{1i}>\Er[x_{1i}]\right\} \right) \right] }{\Er[x_{1i}^{2}]-\Er[x_{1i}]^{2}}
\\
\overline{\beta }_{1}(\indx)=& \frac{\Er\left[ (x_{1i}-\Er[x_{1i}])\left(
\theta _{1i}{1}\left\{ x_{1i}>\Er[x_{1i}]\right\} +\theta _{0i}{1}\left\{
x_{1i}<\Er[x_{1i}]\right\} \right) \right] }{\Er[x_{1i}^{2}]-\Er[x_{1i}]^{2}}
\end{align*}%
where $\theta _{0i}=\theta _{0}(x_{i},\indx)$ and $\theta _{1i}=\theta
_{1}(x_{i},\indx)$.\footnote{%
As one would expect from the definition of the support function and the
properties of linear projection, 
\begin{align*}
\underline{\beta }_{1}(\indx)=& \inf_{f_{i}\in \lbrack \theta _{i0},\theta
_{i1}]}\frac{\text{cov}(x_{1i},f_{i})}{\text{var}(x_{i1})}, \\
\overline{\beta }_{1}(\indx)=& \sup_{f_{i}\in \lbrack \theta _{i0},\theta
_{i1}]}\frac{\text{cov}(x_{1i},f_{i})}{\text{var}(x_{i1})}.
\end{align*}%
}

\subsubsection{Use of asymptotic results}

We develop limit theory that allows us to derive the asymptotic distribution
of the support function process and provide inferential procedures, as well
as to establish validity of the Bayesian bootstrap. Bootstrapping is
especially important for practitioners, because of the potential complexity
of the covariance functions involved in the limiting distributions.

First, our limit theory shows that the support function process $S_{n}(t):=%
\sqrt{n}\left( \hat{\sigma}(t)-\sigma _{0}(t)\right) $ for $t\in \mathcal{S}%
^{d-1}\times \indxSet$ is approximately distributed as a Gaussian process on 
$\mathcal{S}^{d-1}\times \indxSet$. Specifically, we have that%
\begin{equation*}
S_{n}(t)=\mathbb{G}\left[ h_{k}\left( t\right) \right] +o_{\mathrm{P}}\left(
1\right) 
\end{equation*}%
in $\ell ^{\infty }\left( T\right) ,$ where $k$ denotes the number of series
terms in our non-parametric estimator of $\theta _{\ell }(x,\indx),$ $\ell
=0,1,$ $\ell ^{\infty }\left( T\right) $\ denotes the set of all uniformly
bounded real functions on $T,$ and $h_{k}\left( t\right) $ denotes a
stochastic process carefully defined in Section \ref{subsec:theoret_results}%
. Here, $\mathbb{G}\left[ h_{k}\left( t\right) \right] $ is a tight
P-Brownian bridge with covariance function $\Omega _{k}\left( t,t^{\prime
}\right) =\Ep\left[ h_{k}\left( t\right) h_{k}\left( t^{\prime }\right) %
\right] -\Ep\left[ h_{k}\left( t\right) \right] \Ep\left[ h_{k}\left(
t^{\prime }\right) \right] $. By ``approximately distributed'' we mean that
the sequence $\mathbb{G}[h_{k}(t)]$ does not necessarily converge weakly
when $k\rightarrow \infty $; however, each subsequence has a further
subsequence converging to a tight Gaussian process in $\ell ^{\infty }(T)$
with a non-degenerate covariance function.

Second, we show that inference is possible by using the quantiles of the
limiting distribution $\mathbb{G}\left[ h_{k}\left( t\right) \right] $.
Specifically, if we have a continuous function $f$ that satisfies certain
(non-restrictive) conditions detailed in Section 4.3,\footnote{%
For example, functions yielding test statistics based on the directed
Hausdorff distance and on the Hausdorff distance (see, e.g., %
\citeasnoun{molinari_beresteanu_2008}) belong to this class.} and $\hat{c}%
_{n}\left( 1-\tau \right) =c_{n}\left( 1-\tau \right) +o_{\mathrm{P}}\left(
1\right) $ is a consistent estimator of the $\left( 1-\tau \right) $%
-quantile of $f\left( \mathbb{G}\left[ h_{k}\left( t\right) \right] \right) $%
, given by $c_{n}\left( 1-\tau \right) $, then%
\begin{equation*}
\Pr \{f(S_{n})\leq \hat{c}_{n}(1-\tau )\}\rightarrow 1-\tau .
\end{equation*}

Finally, we consider the limiting distribution of the Bayesian bootstrap
version of the support function process, denoted $\tilde{S}_{n}(t):=\sqrt{n}%
\left( \tilde{\sigma}(t)-\hat{\sigma}(t)\right) ,$ and show that,
conditional on the data, it admits an approximation%
\begin{equation*}
\tilde{S}_{n}\left( t\right) =\widetilde{\mathbb{G}\left[ h_{k}\left(
t\right) \right] }+o_{\mathrm{P^{e}}}\left( 1\right)
\end{equation*}%
where $\widetilde{\mathbb{G}\left[ h_{k}\left( t\right) \right] }$ has the
same distribution as $\mathbb{G}\left[ h_{k}\left( t\right) \right] $ and is
independent of $\mathbb{G}\left[ h_{k}\left( t\right) \right] $. Since the
bootstrap distribution is asymptotically close to the true distribution of
interest, this allows us to perform many standard and some less standard
inferential tasks.

\paragraph{\textbf{Pointwise asymptotics}\label{subsec:pointwise}}

{Suppose} we want to form a confidence interval for $q^{\prime }\beta (\indx%
) $ for some fixed $q$ and $\indx$. Since our estimator converges to a
Gaussian process, we know that%
\begin{equation*}
\sqrt{n}%
\begin{pmatrix}
-\hat{\sigma}(-q,\indx)+\sigma _{0}(-q,\indx) \\ 
\hat{\sigma}(q,\indx)-\sigma _{0}(q,\indx)%
\end{pmatrix}%
\approx _{d}N\left( 0,\Omega (q,\indx)\right) .
\end{equation*}%
To form a confidence interval that covers the bound on $q^{\prime }\beta (%
\indx)$ with probability $1-\critv$ we can take\footnote{%
Instead, if one believes there is some true value, $q^{\prime }\beta _{0}(%
\indx)$, in the identified set, and one wants to cover this true value
(uniformly) with asymptotic probability $1-\critv$, then one can adopt the
procedures of \citeasnoun{ImbensManski04} and \citeasnoun{Stoye09}, see also %
\citeasnoun{bmm} for related applications.}%
\begin{equation*}
-\hat{\sigma}(-q,\indx)+n^{1/2}\hat{C}_{\critv/2}(q,\indx)\leq q^{\prime
}\beta (\indx)\leq \hat{\sigma}(q,\indx)+n^{-1/2}\hat{C}_{1-\critv/2}(q,\indx%
)
\end{equation*}%
where the critical values, $\hat{C}_{\critv/2}(q,\indx)$ and $\hat{C}_{1-%
\critv/2}(q,\indx)$, are such that if $%
\begin{pmatrix}
x_{1} & x_{2}%
\end{pmatrix}%
^{\prime }\sim N(0,\Omega )$, then 
\begin{equation*}
\Pr \left( x_{1}\geq \hat{C}_{\critv/2}(q,\indx)\,,\,x_{2}\leq \hat{C}_{1-%
\critv/2}(q,\indx)\right) =1-\critv+o_{p}(1)
\end{equation*}%
If we had a consistent estimate of $\Omega (q,\indx)$, then we could take 
\begin{equation*}
\begin{pmatrix}
\hat{C}_{\critv/2}(q,\indx) \\ 
\hat{C}_{1-\critv/2}(q,\indx)%
\end{pmatrix}%
=\hat{\Omega}^{1/2}(q,\indx)%
\begin{pmatrix}
-\Phi ^{-1}(\sqrt{1-\critv}) \\ 
\Phi ^{-1}(\sqrt{1-\critv})%
\end{pmatrix}%
\end{equation*}%
where $\Phi ^{-1}(\cdot )$ is the inverse normal distribution function.
However, the formula for $\Omega (q,\indx)$ is complicated and it can be
difficult to estimate. Therefore, we recommend and provide theoretical
justification for using a Bayesian bootstrap procedure to estimate the
critical values. See section \ref{sec:bootstrap} for details.

\paragraph{\textbf{Functional asymptotics}\label{subsec:functional}}

Since our asymptotic results show the functional convergence of $S_{n}(q,%
\indx)$, we can also perform inference on statistics that involve a
continuum of values of $q$ and/or $\indx$. For example, in our application
to quantile regression with selectively observed data, we might be
interested in whether a particular variable has a positive affect on the
outcome distribution. That is, we may want to test%
\begin{equation*}
H_{0}:0\in \lbrack -\sigma _{0}(-q,\indx),\sigma _{0}(q,\indx)]\text{ }%
\forall \indx\in \indxSet,
\end{equation*}%
with $q=e_{j}$. A natural family of test statistics is 
\begin{equation*}
T_{n}=\sqrt{n}\sup_{\indx\in \indxSet}%
\begin{pmatrix}
1\{-\hat{\sigma}(-q,\indx)>0\}\vert \hat{\sigma}(-q,\indx)\vert \rho (-q,%
\indx)\vee  \\ 
\vee \{\hat{\sigma}(q,\indx)<0\}\vert \hat{\sigma}(q,\indx)\vert \rho (q,%
\indx)%
\end{pmatrix}%
\end{equation*}%
where $\rho (q,\indx)\geq 0$ is some weighting function which can be chosen
to maximize weighted power against some family of alternatives. There are
many values of $\sigma _{0}(q,\indx)$ consistent with the null hypothesis,
but the one for which it will be hardest to control size is $-\sigma
_{0}(-q,\cdot )=\sigma _{0}(q,\cdot )=0$. In this case, we know that $%
S_{n}(t)=\sqrt{n}\hat{\sigma}(t),$ $t=(q,\indx )\in \mathcal{S}^{d-1}\times 
\indxSet,$ is well approximated by a Gaussian process, $\G[h_{k}(t)]$.
Moreover, the quantiles of any functional of $S_{n}(t)$ converge to the
quantiles of the same functional applied to $\G[h_{k}(t)]$. Thus, we could
calculate a $\critv$ critical value for $T_{n}$ by repeatedly simulating a
realization of $\G[h_{k}(q,\cdot)]$, computing $T_{n}(\G[h_{k}(q,\cdot)])$,
and then taking the $\left( 1-\critv\right) $-quantile of the simulated
values of $T_{n}(\G[h_{k}(q,\cdot)])$.\footnote{%
This procedure yields a test with correct asymptotic size. However, it might
have poor power properties in some cases. In particular, when $\sigma
_{0}(-q,\indx)\neq \sigma _{0}(q,\indx)$, the critical values might be too
conservative. One can improve the power properties of the test by applying
the generalized moment selection procedure proposed by %
\citeasnoun{AndrewsShi09} to our framework.} Simulating $\G[h_{k}(t)]$
requires estimating the covariance function. As stated above, the formula
for this function is complicated and it can be difficult to estimate.
Therefore, we recommend using the Bayesian bootstrap to compute the critical
values. Theorem \ref{thm:bsInference} proves that this bootstrap procedure
yields consistent inference. Section \ref{sec:bootstrap} gives a more
detailed outline of implementing this bootstrap. Similar reasoning can be
used to test hypotheses involving a set of values of $q$ and construct
confidence sets that are uniform in $q$ and/or $\indx$.

\subsubsection{Estimation}

The first step in estimating the support function is to estimate $\theta
_{0}(x,\indx)$ and $\theta _{1}(x,\indx)$. Since economic theory often
provides even less guidance about the functional form of these bounding
functions than it might about the function of interest, our asymptotic
results are written to accommodate non-parametric estimates of $\theta _{0}(x,%
\indx)$ and $\theta _{1}(x,\indx)$. In particular, we allow for series
estimators of these functions. In this section we briefly review this
approach. Parametric estimation follows as a special case where the number
of series terms is fixed. Note that while the method of series estimation
described here satisfies the conditions of theorems \ref{thm:limitTheory}
and \ref{thm:inference} below, there might be other suitable methods of
estimation for the bounding functions.

In each of the examples in section \ref{sec:examples}, except for the case
of sample selection with an exclusion restriction, series estimates of the
bounding functions can be formed as follows. Suppose there is some $y_{i\ell
}$, a known function of the data for observation $i,$ and a known function $%
m(y,\theta (x,\indx),\indx)$ such that 
\begin{equation*}
\theta _{\ell }(\cdot ,\indx)=\argmin_{\theta \in \LL^{2}(X,\Pr )}\Er\left[
m\left( y_{i\ell },\theta (x_{i},\indx),\indx\right) \right] ,
\end{equation*}%
where $X$ denotes the support of $x$ and $\LL^{2}(X,\Pr )$ denotes the space
of real-valued functions $g$ such that $\int_{X}\left\vert g(x)\right\vert
^{2}d\Pr (x)<\infty $. Then we can form an estimate of the function $\theta
_{\ell }(\cdot ,\indx)$ by replacing it with its series expansion and taking
the empirical expectation in the equation above. That is, obtaining the
coefficients 
\begin{equation*}
\hat{\vartheta}_{\ell }(\indx)=\argmin_{\vartheta }\En\left[ m\left(
y_{i\ell },p_{k}(x_{i})^{\prime }\vartheta ,\indx\right) \right] ,
\end{equation*}%
and setting%
\begin{equation*}
\hat{\theta}_{l}(x_{i},\indx)=p_{k}(x_{i})^{\prime }\hat{\vartheta}_{\ell }(%
\indx).
\end{equation*}%
Here, $p_{k}(x_{i})$ is a $k\times 1$ vector of $k$ series functions
evaluated at $x_{i}$. These could be any set of functions that span the
space in which $\theta _{\ell }(x,\indx)$ is contained. Typical examples
include polynomials, splines, and trigonometric functions, see %
\citeasnoun{chen_large_2007}. Both the properties of $m(\cdot )$ and the
choice of approximating functions affect the rate at which $k$ can grow. We
discuss this issue in more detail after stating our regularity conditions in
section \ref{sec:conditions}.

In the case of sample selection with an exclusion, one can proceed as
follows. First, estimate $Q_{\tilde{y}_{\ell }}(\indx|x,v),$ $\ell =0,1,$
using the method described above. Next, set $\hat{\theta}_{l}(x_{i},\indx%
)=\min_{v\in supp(v|x)}Q_{\tilde{y}_{\ell }}(\indx|x,v).$ Below we establish
the validity of our results also for this case.

\subsubsection{Bayesian Bootstrap \label{sec:bootstrap}}

We suggest using the Bayesian Bootstrap to conduct inference. In particular,
we propose the following algorithm.

\paragraph{Procedure for Bayesian Bootstrap Estimation}

\begin{enumerate}
\item \renewcommand{\theenumi}{\arabic{enumi}}Simulate each bootstrap draw
of $\tilde{\sigma}(q,\indx)$ :

\begin{enumerate}
\item Draw $e_{i}\sim \exp (1)$, $i=1,...,n$, $\bar{e}=\En[e_{i}]$

\item Estimate:%
\begin{eqnarray*}
\tilde{\vartheta}_{\ell } &=&\argmin_{\vartheta }\En\left[ \frac{e_{i}}{\bar{%
e}}m\left( y_{i\ell },p_{k}(x_{i})^{\prime }\vartheta ,\indx\right) \right] ,
\\
\tilde{\theta}_{\ell }(x,\indx) &=&p_{k}(x)^{\prime }\tilde{\vartheta}_{\ell
}, \\
\tilde{\Sigma} &=&\En\left[ \frac{e_{i}}{\bar{e}}x_{i}z_{i}^{\prime }\right]
^{-1}, \\
\tilde{w}_{i,q^{\prime }\tilde{\Sigma}} &=&\tilde{\theta}_{1}(x,\indx%
)1(q^{\prime }\tilde{\Sigma}z>0)+\tilde{\theta}_{0}(x,\indx)1(q^{\prime }%
\tilde{\Sigma}z\leq 0), \\
\tilde{\sigma}(q,\indx) &=&\En\left[ \frac{e_{i}}{\bar{e}}q^{\prime }\tilde{%
\Sigma}z_{i}\tilde{w}_{i,q^{\prime }\tilde{\Sigma}}\right] .
\end{eqnarray*}
\end{enumerate}

\item Denote the bootstrap draws as $\tilde{\sigma}^{(b)}$, $b=1,...,B$, and
let $\tilde{S}_{n}^{(b)}=\sqrt{n}(\tilde{\sigma}^{(b)}-\hat{\sigma})$. To
estimate the $1-\critv$ quantile of $\mathcal{T}(S_{n})$ use the empirical $%
1-\critv$ quantile of the sample $\mathcal{T}(\tilde{S}_{n}^{(b)})$, $%
b=1,...,B$

\item Confidence intervals for linear combinations of coefficients can be
obtained as outlined in Section \ref{subsec:pointwise}. Inference on
statistics that involve a continuum of values of $q$ and/or $\indx$ can be
obtained as outlined in Section \ref{subsec:functional}.
\end{enumerate}

\subsection{Regularity Conditions\label{sec:conditions}}

In what follows, we state the assumptions that we maintain to obtain our
main results. We then discuss these conditions, and verify them for the
examples in Section \ref{sec:examples}.

\begin{assumption}[Smoothness of Covariate Distribution\label{c:smooth}]
The covariates $z_{i}$ have a sufficiently smooth distribution, namely for
some $0<m\leq 1$, we have that $\Pr\left( |q^{\prime }\Sigma z_{i}/\Vert
z_{i}\Vert |<\delta \right) /\delta ^{m}\lesssim 1$ as $\delta \searrow 0$
uniformly in $q\in \mathcal{S}^{d-1},$ with $d$ the dimension of $x$. The
matrix $\Sigma =(\Er[x_{i}z_{i}])^{-1}$ is finite and invertible.
\end{assumption}

\begin{assumption}[Linearization for the Estimator of Bounding Functions 
\label{c:linearization}]
Let $\bar{\theta}$ denote either the unweighted estimator $\hat{\theta}$ or
the weighted estimator $\tilde{\theta}$, and let $v_{i}=1$ for the case of
the unweighted estimator, and $v_{i}=e_{i}$ for the case of the weighted
estimator. We assume that for each $\ell =0,1$ the estimator $\bar{\theta}%
_{\ell }$ admits a linearization of the form: 
\begin{equation}
\sqrt{n}\left( \bar{\theta}_{\ell }(x,\indx)-\theta _{\ell }(x,\indx)\right)
=p_{k}(x)^{\prime }J_{\ell }^{-1}(\indx)\mathbb{G}_{n}[v_{i}p_{i}\varphi
_{i\ell }(\indx)]+\bar{R}_{\ell }(x,\indx)  \label{eq:linearTheta}
\end{equation}%
where $p_{i}=p_{k}(x_{i}),$ $\sup_{\indx\in \indxSet}\Vert \bar{R}_{\ell
}(x_{i},\indx)\Vert _{\Pn,2}\rightarrow _{P}0$, and $(x_{i},z_{i},\varphi
_{i\ell })$ are i.i.d. random elements.
\end{assumption}

\begin{assumption}[Design Conditions\label{c:design conditions}]
The score function $\varphi _{i\ell }(\indx)$ is mean zero conditional on $%
x_{i},z_{i}$ and has uniformly bounded fourth moment conditional on $%
x_{i},z_{i}$. The score function is smooth in mean-quartic sense: $\Ep\left[
\left( \varphi _{i\ell }(\indx)-\varphi _{i\ell }(\tilde{\indx})\right)
^{4}|x_{i},z_{i}\right] ^{1/2}\leq C\norm{\indx-\tilde{\indx}}^{\gamma
_{\varphi }}$ for some constants $C$ and $\gamma _{\varphi }>0.$ Matrices $%
J_{\ell }(\indx)$ exist and are uniformly Lipschitz over $\indx\in \indxSet$%
, a bounded and compact subset of $\mathbb{R}^{l}$, and $\sup_{\indx\in %
\indxSet}\Vert J^{-1}(\indx)\Vert $ as well as the operator norms of
matrices $\Er[z_{i}z_{i}^{\prime }]$, $\Er[z_{i}p_{i}^{\prime }]$, and $%
\Er[\norm{p_i
p_i'}^{2}]$ are uniformly bounded in $k$. $\Er[\norm{z_i}^6]$ and $%
\Er[\norm{x_i}^6]$ are finite. $\Ep\left[ \Vert \theta _{\ell }(x_{i},\indx%
)\Vert ^{6}\right] $ is uniformly bounded in $\indx$, and $%
\Ep[
  |\varphi_{i\ell}(\indx)^{4} | x_i,z_i]$ is uniformly bounded in $\indx$, $%
x $, and $z$. The functions $\theta _{\ell }(x,\indx)$ are smooth, namely $%
\abs{\theta_\ell(x,\indx)-\theta_\ell(x,\tilde{\indx})}\leq L(x)%
\norm{\indx-\tilde{\indx}}^{\gamma _{\theta }}$ for some constant $\gamma
_{\theta }>0$ and some function $L(x)$\ with $\Er\left[ L(x)^{4}\right] $
bounded.
\end{assumption}

\begin{assumption}[Growth Restrictions\label{c:growth}]
When $k\rightarrow \infty $, $\sup_{x\in X}\Vert p_{k}(x)\Vert \leq \xi _{k}$%
, and the following growth condition holds on the number of series terms: 
\begin{equation*}
\log ^{2}n\left( n^{-m/4}+\sqrt{(k/n)\cdot \log n}\cdot \max_{i}\Vert
z_{i}\Vert \wedge \xi _{k}\right) \left\vert \sqrt{\max_{i\leq n,}F_{1}^{4}}%
\right\vert \rightarrow _{\Pr }0,\ \ \xi _{k}^{2}\log ^{2}n/n\rightarrow 0,
\end{equation*}%
where $\F_{1}$ is defined in Condition \ref{c:complexity function classes} below.
\end{assumption}

\begin{assumption}[Complexity of Relevant Function Classes\label%
{c:complexity function classes}]
The function set $\F_{1}=\{\varphi _{i\ell }(\indx),\indx\in \indxSet,\ell
=0,1\}$ has a square $\Pr $-integrable envelope $F_{1}$ and has a uniform
covering $L_{2}$ entropy equivalent to that of a VC class. The function
class $\F_{2}\supseteq \{\theta _{i\ell }(\indx),\indx\in \indxSet,\ell
=0,1\}$ has a square $\Pr $-integrable envelope $F_{2}$ for the case of
fixed $k$ and bounded envelope $F_{2}$ for the case of increasing $k$, and
has a uniform covering $L_{2}$ entropy equivalent to that of a VC class.
\end{assumption}

\subsubsection{Discussion and verification of conditions}

Condition \ref{c:smooth} requires that the covariates $z_{i}$ be
continuously distributed, which in turn assures that the support function is
everywhere differentiable in $q\in \mathcal{S}^{d-1},$ see 
\citeasnoun[Lemma
A.8]{molinari_beresteanu_2008} and Lemma \ref{lemma: derivative} in the
Appendix. With discrete covariates, the identified set has exposed faces and
therefore its support function is not differentiable in directions $q$
orthogonal to these exposed faces, see e.g., \citeasnoun[Section 3.1]{bmm}.
In this case, Condition \ref{c:smooth} can be met by adding to the discrete
covariates a small amount of smoothly distributed noise. Adding noise gives
"curvature" to the exposed faces, thereby guaranteeing that the identified
set intersects its supporting hyperplane in a given direction at only one
point, and is therefore differentiable, see 
\citeasnoun[Corollary
1.7.3]{Schneider93}. Lemma \ref{prop:approx} in the Appendix shows that the
distance between the true identified set and the set resulting from jittered
covariates can be made arbitrarily small. Therefore, in the presence of
discrete covariates one can apply our method obtaining arbitrarily slightly
conservative inference.

Assumptions \ref{c:design conditions} and \ref{c:complexity function classes}
are common regularity conditions and they can be verified using standard
arguments. Condition \ref{c:linearization} requires the estimates of the
bounding functions to be asymptotically linear. In addition, it requires
that the number of series terms grows fast enough for the remainder term to
disappear. This requirement must be reconciled with Condition \ref{c:growth}%
, which limits the rate at which the number of series terms can increase. We
show below how to verify these two conditions in each of the examples of
Section \ref{sec:examples}.

\begin{example*}[Mean regression, continued]
We begin with the simplest case of mean regression with interval valued
outcome data. In this case, we have $\hat{\theta}_{\ell }(\cdot ,\indx%
)=p_{k}(\cdot )^{\prime }\hat{\vartheta}_{\ell }$ with $\hat{\vartheta}%
_{\ell }=(P^{\prime }P)^{-1}P^{\prime }y_{\ell }$ and $%
P=[p_{k}(x_{1}),...,p_{k}(x_{n})]^{\prime }$.\ Let $\vartheta _{k}$ be the
coefficients of a projection of $\Er[y_{\ell }|x_{i}]$ on $P$, or pseudo-true
values, so that $\vartheta _{k}=(P^{\prime }P)^{-1}P^{\prime }\Er[y_{\ell
}|x_{i}]$. We have the following linearization for $\hat{\theta}_{\ell
}(\cdot ,\indx)$ 
\begin{equation*}
\sqrt{n}\left( \hat{\theta}_{\ell }(x,\indx)-\theta _{\ell }(x,\indx)\right)
=\sqrt{n}p_{k}(x)(P^{\prime }P)^{-1}P^{\prime }(y_{\ell }-\Er[y_{\ell }|x])+%
\sqrt{n}\left( p_{k}(x)^{\prime }\vartheta _{k}-\theta _{\ell }(x,\indx%
)\right) .
\end{equation*}%
This is in the form of (\ref{eq:linearTheta}) with $J_{\ell }(\indx%
)=P^{\prime }P$, $\varphi _{i\ell }(\indx)=(y_{i\ell }-\Er[y_{\ell }|x_{i}])$,
and $R_{\ell }(x,\indx)=\sqrt{n}\left( p_{k}(x)^{\prime }\vartheta
_{k}-\theta _{\ell }(x,\indx)\right) $. The remainder term is simply
approximation error. Many results on the rate of approximation error are
available in the literature. This rate depends on the choice of
approximating functions, smoothness of $\theta _{\ell }(x,\indx)$, and
dimension of $x$. When using polynomials as approximating function, if $%
\theta _{\ell }(x,\indx)=\Er[y_{\ell }|x_{i}]$ is $s$ times differentiable
with respect to $x,$ and $x$ is $d$ dimensional, then (see e.g.\ %
\citeasnoun{newey_convergence_1997} or \citeasnoun{Lorentz86}) 
\begin{equation*}
\sup_{x}|p_{k}(x)^{\prime }\vartheta _{k}-\theta _{\ell }(x,\indx%
)|=O(k^{-s/d}).
\end{equation*}%
In this case \ref{c:linearization} requires that $n^{1/2}k^{-s/d}\rightarrow
0$, or that $k$ grows faster than $n^{\frac{d}{2s}}$. Assumption \ref%
{c:growth} limits the rate at which $k$ can grow. This assumption involves $%
\xi _{k}$ and $\sup_{i,\indx}|\varphi _{i}(\indx)|$. The behavior of these
terms depends on the choice of approximating functions and some auxiliary
assumptions. With polynomials as approximating functions and the support of $%
x$ compact with density bounded away from zero, $\xi _{k}=O(k)$. If $%
y_{i\ell }-\Er[y_{\ell }|x_{i}]$ has exponential tails, then $\sup_{i,\indx%
}|\varphi _{i}(\indx)|=O(2(\log n)^{1/2})$. In this case, a sufficient
condition to meet \ref{c:growth} is that $k=o(n^{1/3}\log ^{-6}n)$. Thus, we
can satisfy both \ref{c:linearization} and \ref{c:growth} by setting $%
k\propto n^{\gamma }$ for any $\gamma $ in the interval connecting $\frac{d}{%
2s}$ and $\frac{1}{3}$. Notice that as usual in semiparametric problems, we
require undersmoothing compared to the rate that minimizes mean-squared
error, which is $\gamma =\frac{d}{d+2s}$. Also, our assumption requires
increasing amounts of smoothness as the dimension of $x$ increases.
\end{example*}

We now discuss how to satisfy assumptions \ref{c:linearization} and \ref%
{c:growth} more generally. Recall that in our examples, the series estimates
of the bounding functions solve 
\begin{equation*}
\hat{\theta}_{\ell }(\cdot ,\indx)=\argmin_{\theta _{\ell }\in \LL^{2}(X,\Pr
)}\En\left[ m(y_{i\ell },\theta _{\ell }(x_{i},\indx),\indx)\right] 
\end{equation*}%
or $\hat{\theta}_{\ell }(\cdot ,\indx)=p_{k}(\cdot )^{\prime }\hat{\vartheta}%
_{\ell }$ with $\hat{\vartheta}_{\ell }=\argmin_{\vartheta }\En\left[
m(y_{i\ell },p_{k}(x_{i})^{\prime }\vartheta ,\indx)\right] .$ As above, let 
$\vartheta _{k}$ be the solution to ${\vartheta }_{k}=\argmin_{\vartheta }\Ep%
\left[ m(y_{i\ell },p_{k}(x_{i})^{\prime }\vartheta ,\indx)\right] .$ We
show that the linearization in \ref{c:linearization} holds by writing 
\begin{equation}
\sqrt{n}\left( \hat{\theta}(x,\indx)-\theta _{\ell }(x,\indx)\right) =\sqrt{n%
}p_{k}(x)\left( \hat{\vartheta}-\vartheta _{k}\right) +\sqrt{n}\left(
p_{k}(x)^{\prime}\vartheta _{k}-\theta _{\ell }(x,\indx)\right) \label{eq:thetaParts}.
\end{equation}%
The first term in (\ref{eq:thetaParts}) is estimation error. 
We can use the results of \citeasnoun{He2000} to show that 
\begin{equation*}
\left( \hat{\vartheta}-\vartheta _{k}\right) =\En\left[ J_{\ell
}^{-1}p_{i}\psi _{i}\right] +o_{p}(n^{-1/2}),
\end{equation*}%
where $\psi $ denotes the derivative of $m(y_{i\ell },p_{k}(x_{i})^{\prime
}\vartheta ,\indx)$ with respect to $\vartheta.$ 

The second term in (\ref{eq:thetaParts}) is approximation
error. Standard results from approximation theory as stated in e.g.\ %
\citeasnoun{chen_large_2007} or \citeasnoun{newey_convergence_1997}
give the rate at which the error from the best $L_2$-approximation to
$\theta_\ell$ disappears. When $m$ is a least squares objective
function, these results can be applied directly. In other cases, such
as quantile or distribution regression, further work must be done.

\begin{example*}[Quantile regression with interval valued data, continued]
The results of \citeasnoun{bcf2011} can be used to verify our
conditions for quantile regression.  In particular, Lemma 1 from
Appendix B of \citeasnoun{bcf2011} gives the rate at which the
approximation error vanishes, and Theorem 2 from \citeasnoun{bcf2011}
shows that the linearization condition (\ref{c:linearization})
holds. The conditions required for these results are as follows.
\renewcommand{\theenumi}{Q.\arabic{enumi}}
\begin{enumerate}
\item The data $\{(y_{i0}, y_{i1}, x_i), 1\leq i \leq n \}$ are an
  i.i.d.\ sequence of real $(2+d)$-vectors.
\item For $\ell = \{0,1\}$, the conditional density of $y_{\ell}$
  given $x$ is bounded above by $\overline{f}$, below by
  $\underline{f}$, and its derivative is bounded above by
  $\overline{f'}$ uniformly in $y_{\ell}$ and
  $x$. $f_{y_\ell|X}\left(Q_{y_\ell|x}(\indx|x)|x\right)$ is bounded
  away from zero uniformly in $\indx \in \indxSet$ and $x \in
  \mathcal{X}$.
\item For all $k$, the eigenvalues of $\Er[p_i p_i']$ are uniformly bounded
  above and away from zero.
\item\label{q.series} 
  $\xi_k = O(k^a)$. $Q_{y_\ell|x}$ is $s$ times continuously
  differentiable with $s > (a+1)d$.  The series functions are such that 
  \[ \inf_{\vartheta} \Er \left[ \norm{p_k \vartheta -
      \theta_\ell(x,\indx)}_2 \right] = O(k^{-s/d}) \] 
  and 
  \[ \inf_{\vartheta} \norm{p_k \vartheta -
    \theta_\ell(x,\indx)}_\infty = O(k^{-s/d}). \]  
\item \label{q.k} $k$ is chosen such that $k^{3+6a} (\log n)^7 =
  o(n)$. 
\end{enumerate}
Condition \ref{q.series} is satisfied by polynomials with $a=1$ and by
splines or trigonometric series with $a=1/2$. Under these assumptions,
Lemma 1 of appendix B from \citeasnoun{bcf2011} shows that the
approximation error satisfies
\[ \sup_{x \in \mathcal X, \indx \in \indxSet} \abs{ p_k(x)^{\prime}
  \vartheta_k(\indx) - \theta_\ell(x,\indx) } \lesssim k^{\frac{ad -
    s}{d}}. \]
Theorem 2 of \citeasnoun{bcf2011} then shows that
\ref{c:linearization} holds. Condition \ref{c:growth} also holds
because for quantile regression $\psi_i$ is bounded, so \ref{c:growth}
only requires $k^{1+2a} (\log n)^2 = o(n)$, which is implied by
\ref{q.k}. 
\end{example*}

\begin{example*}[Distribution regression, continued]
  As described above, the estimator solves
 $\hat{\vartheta}=$ $\argmin%
_{\vartheta }\En\left[ m(y_{i},p_{k}(x_{i})^{\prime }\vartheta ,\indx)\right]
$ with 
\begin{equation*}
m(y_{i},p_{k}(x_{i})^{\prime }\vartheta ,\indx)=-1\{y<\indx\}\log \Phi \left(
p_{k}(x_{i})^{\prime }\vartheta \right) -1\{y\geq \indx\}\log \left( 1-\Phi
\left( p_{k}(x_{i})^{\prime }\vartheta \right) \right) 
\end{equation*}%
for some known distribution function $\Phi $. We must show that
estimation error, $\hat{\vartheta} - \vartheta_k$, can be linearized
and that the bias, $p_k(x) \vartheta_k - \theta_\ell(x,\indx)$, is
$o(n^{-1/2})$. We first verify the conditions of %
\citeasnoun{He2000} to show that $(\hat{\vartheta}-\vartheta _{k})$ can be
linearized. Adopting their notation, in this example we have that the
derivative of $m(y_{i},p_{k}(x_{i})^{\prime }\vartheta ,\indx)$ with respect
to $\vartheta $ is 
\begin{equation*}
  \psi (y_{i},x_{i},\vartheta )=-\left( \frac{1\{y<\indx\}}{\Phi \left(
p_{k}(x_{i})^{\prime }\vartheta \right) }-\frac{1\{y\geq \indx\}}{1-\Phi
\left( p_{k}(x_{i})^{\prime }\vartheta \right) }\right) \phi \left(
p_{k}(x_{i})^{\prime }\vartheta \right) p_{k}(x_{i}),
\end{equation*}%
where $\phi $ is the pdf associated with $\Phi $. Because
 $m(y_{i},p_{k}(x_{i})^{\prime }\vartheta ,\indx)$ 
is a smooth function of $%
\vartheta ,$ $\En\psi (y_{i},x_{i},\hat{\vartheta})=0$, and conditions C.0
and C.2 in \citeasnoun{He2000} hold. If $\phi $ is differentiable with a
bounded derivative, then $\psi $ is Lipschitz in $\vartheta $, and we have
the bound 
\begin{equation*}
\norm{\eta_i(\vartheta,\tau)}^{2}\lesssim \norm{p_k(x_i)}^{2}%
\norm{\tau-\vartheta}^{2},
\end{equation*}%
where $\eta _{i}(\vartheta ,\tau )=\psi (y_{i},x_{i},\vartheta )-\psi
(y_{i},x_{i},\tau )-E\psi (y_{i},x_{i},\vartheta )+E\psi (y_{i},x_{i},\tau ).
$ If we assume 
\begin{equation*}
\max_{i\leq n}\norm{p_k(x_i)}=O(k^{a }),
\end{equation*}%
as would be true for polynomials with $a =1$ or splines with
$%
a =1/2$, and $k$ is of order less than or equal to $n^{1/a }$
then condition C.1 in \citeasnoun{He2000} holds. Differentiability of $\phi $
and \ref{c:design conditions} are sufficient for C.3 in \citeasnoun{He2000}.
Finally, conditions C.4 and C.5 hold with $A(n,k) = k$ because 
\begin{equation*}
\left\vert s^{\prime }\eta _{i}(\vartheta ,\tau )\right\vert ^{2}\lesssim
\left\vert s^{\prime }p_{k}(x_{i})\right\vert ^{2}\norm{\tau - \vartheta}%
^{2} ,
\end{equation*}%
and $\Er\left[ |s^{\prime }p_{k}(x_{i})|^{2}\right] $ is
uniformly bounded for $s\in S^{k}$ for all $k$ when the series
functions are orthonormal. Applying Theorem 2.2 of
\citeasnoun{He2000}, we obtain the desired linearization if $k=o\left(
  (n/\log n)^{1/2}\right) $.

The results of \citeasnoun{hiranoImbensRidder2003} can be used to show
that the approximation bias is sufficiently small. Lemma 1 from
\citeasnoun{hiranoImbensRidder2003} shows that for the logistic
distribution regression,
\begin{align*}
  \norm{\vartheta_k - \vartheta_k^\ast} = & O(k^{-s/(2d)}) \\
  \sup_x \abs{ \Phi\left( p_k(x) \vartheta_k \right) 
    - \Phi\left(\theta_\ell(x,\indx) \right)} = &
  O(k^{-s/(2d)} \xi_k), 
\end{align*}
which implies that
\begin{align*}
  \sup_{x \in \mathcal{X}, \indx \in \indxSet} \abs{  p_k(x) \vartheta_k 
    - \theta_\ell(x,\indx) } = &
  O(k^{-s/(2d)} \xi_k).
\end{align*}
This result is only for the logistic link function, but it can easily
be adapted for any link function with first derivative bounded from
above and second derivative bounded away from zero. We need the
approximation error to be $o(n^{-1/2})$. For this, it suffices to have
\[ k^{-s/(2d)} \xi_k n^{1/2} = o(1). \] 
Letting $\xi_k = O(k^a)$ as above, it suffices to have $k \propto
n^{\gamma}$ for $\gamma > \frac{d}{s-2ad}$. 

To summarize, condition \ref{c:linearization} can be met by having $k
\propto n^{\gamma}$ for any $\gamma \in \left(\frac{d}{s - 2ad},
  \frac{1}{2} \right)$. Finally, as in the mean and quantile
regression examples above, condition \ref{c:growth} will be
met if $\gamma < \frac{1}{1+2a}$. 
\end{example*}

\subsection{Theoretical Results\label{subsec:theoret_results}}

In order to state the result we define 
\begin{eqnarray}
h_{k}(t):= &&q^{\prime }\Sigma \Er[z_i p_i' 1\{q '\Sigma z_i >0\}]J_{1}^{-1}(%
\indx)p_{i}\varphi _{i1}(\indx)  \notag \\
&+&q^{\prime }\Sigma \Er[z_i p_i' 1\{q'\Sigma z_i <0\}]J_{0}^{-1}(\indx%
)p_{i}\varphi _{i0}(\indx)  \notag \\
&-&q^{\prime }\Sigma x_{i}z_{i}^{\prime }\Sigma \Ep\left[ z_{i}w_{i,q^{%
\prime }\Sigma }(\indx)\right]  \notag \\
&+&q^{\prime }\Sigma z_{i}w_{i,q^{\prime }\Sigma }(\indx).  \notag
\end{eqnarray}

\begin{theorem}[Limit Theory for Support Function Process\label%
{thm:limitTheory}]
The support function process $S_{n}(t)=\sqrt{n}\left( \hat{\sigma}_{\widehat{%
\theta },\widehat{\Sigma }}-\sigma _{\theta ,\Sigma }\right) (t)$, where $%
t=(q,\indx)\in \mathcal{S}^{d-1}\times \indxSet$, admits the approximation $%
S_{n}(t)=\Gn[h_k(t)]+o_{\Pr }(1)$ in $\ell ^{\infty }(T)$. Moreover, the
support function process admits an approximation 
\begin{equation*}
S_{n}(t)=\mathbb{G}[h_{k}(t)]+o_{\Pr }(1)\text{ in }\ell ^{\infty }(T),
\end{equation*}%
where the process $\mathbb{G}[h_{k}(t)]$ is a tight P-Brownian bridge in $%
\ell ^{\infty }(T)$ with covariance function $\Omega _{k}\left( t,t^{\prime
}\right) =\Ep[h_k(t) h_k(t')]$ $-$ $\Ep[h_k(t)]\Ep[h_k(t')]$ that is
uniformly Holder on $T\times T$ uniformly in $k,$ and is uniformly
non-degenerate in $k$. These bridges are stochastically equicontinuous with
respect to the $L_{2}$ pseudo-metric $\rho _{2}(t,t^{\prime })=[%
\Ep[ h(t) -
h(t')]^{2}]^{1/2}\lesssim \Vert t-t^{\prime }\Vert ^{c}$ for some $c>0$
uniformly in $k$. The sequence $\mathbb{G}[h_{k}(t)]$ does not necessarily
converge weakly under $k\rightarrow \infty $; however, each subsequence has
a further convergent subsequence converging to a tight Gaussian process in $%
\ell ^{\infty }(T)$ with a non-degenerate covariance function. Furthermore,
the canonical distance between the law of the support function process $%
S_{n}(t)$ and the law of $\G[h_k(t)]$ in $\ell ^{\infty }(T)$ approaches
zero, namely $\sup_{g\in BL_{1}(\ell ^{\infty }(T),[0,1])}|\Ep [g (S_n)]-%
\Ep[g(\G [h_k] ) ]|\rightarrow 0.$
\end{theorem}


Next we consider the behavior of various statistics based on the support
function process. Formally, we consider these statistics as mappings $f:\ell
^{\infty }(T)\rightarrow \mathbb{R}$ from the possible values $s$ of the
support function process $S_{n}$ to the real line. Examples include:

\begin{itemize}
\item a support function evaluated at $t\in T,$ $f(s)=s(t)$,

\item a Kolmogorov statistic, $f(s)=\sup_{t\in T_{0}}|s(t)|/\varpi (t)$,

\item a directed Kolmogorov statistic, $f(s)=\sup_{t\in T_{0}}\left\{
-s(t)\right\} _{+}/\varpi (t)$,

\item a Cramer-Von-Mises statistic, $f(s)=\int_{T}s^{2}(t)/\varpi (t)d\nu
(t) $,
\end{itemize}

\noindent where $T_{0}$ is a subset of $T$, $\varpi $ is a continuous and
uniformly positive weighting function, and $\nu $ is a probability measure
over $T$ whose support is $T$.\footnote{%
Observe that test statistics based on the (directed) Hausdorff distance
(see, e.g., \citeasnoun{molinari_beresteanu_2008}) are special cases of the
(directed) Kolmogorov statistics above.} More generally we can consider any
continuous function $f$ such that $f(Z)$ (a) has a continuous distribution
function when $Z$ is a tight Gaussian process with non-degenerate covariance
function and (b) $f(\xi _{n}+c)-f(\xi _{n})=o(1)$ for any $c=o(1)$ and any $%
\Vert \xi _{n}\Vert =O(1)$. Denote the class of such functions $\mathcal{F}%
_{c}$ and note that the examples mentioned above belong to this class by the
results of \citeasnoun{davydov}.

\begin{theorem}[Limit Inference on Support Function Process\label%
{thm:inference}]Furthermore,
the canonical distance between the law of the support function process $%
S_{n}(t)$ and the law of $\G[h_k(t)]$ in $\ell ^{\infty }(T)$ approaches
zero, namely $\sup_{g\in BL_{1}(\ell ^{\infty }(T),[0,1])}|\Ep [g (S_n)]-%
\Ep[g(\G [h_k] ) ]|\rightarrow 0.$
For any $\hat{c}_{n}=c_{n}+o_{\Pr }(1)$ and $c_{n}=O_{\Pr }(1)$ and $f\in 
\mathcal{F}_{c}$ we have 
\begin{equation*}
\Pr \{f(S_{n})\leq \widehat{c}_{n}\}-\Pr \{f(\G [h_k])\leq
c_{n}\}\rightarrow 0.
\end{equation*}%
If $c_{n}(1-\critv)$ is the $(1-\critv)$-quantile of $f(\G [h_k])$ and $\hat{%
c}_{n}(1-\critv)=c_{n}(1-\critv)+o_{\Pr }(1)$ is any consistent estimate of
this quantile, then 
\begin{equation*}
\Pr \{f(S_{n})\leq \hat{c}_{n}(1-\critv)\}\rightarrow 1-\critv.
\end{equation*}
\end{theorem}

Let $e_{i}^{o}=e_{i}-1,$ $h_{k}^{o}=h_{k}-\Ep[h_k],$ and let $\Pr^{e}$
denote the probability measure conditional on the data.

\begin{theorem}[Limit Theory for the Bootstrap Support Function Process\label%
{thm:bsLimit}]
The bootstrap support function process $\tilde{S}_{n}(t)=\sqrt{n}(\widetilde{%
\sigma }_{\widetilde{\theta },\widetilde{\Sigma }}-\hat{\sigma}_{\hat{\theta}%
,\hat{\Sigma}})(t)$, where $t=(q,\indx)\in \mathcal{S}^{d-1}\times \indxSet$%
, admits the following approximation conditional on the data: $\widetilde{S}%
_{n}(t)=\Gn[e_i^oh^o_k(t)]+o_{\Pr^{e}}(1)$ in $\ell ^{\infty }(T)$ in
probability $\Pr $. Moreover, the bootstrap support function process admits
an approximation conditional on the data: 
\begin{equation*}
\widetilde{S}_{n}(t)=\widetilde{\G[h_k(t)]}+o_{\Pr^{e}}(1)\text{ in }\ell
^{\infty }(T),\text{ in probability }\Pr ,
\end{equation*}%
where $\widetilde{\G[h_k]}$ is a sequence of tight $\Pr $-Brownian bridges
in $\ell ^{\infty }(T)$ with the same distributions as the processes $\G[h_k]
$ defined in Theorem 1, and independent of $\G[h_k]$. Furthermore, the
canonical distance between the law of the bootstrap support function process 
$\widetilde{S}_{n}(t)$ conditional on the data and the law of $\G[h_k(t)]$
in $\ell ^{\infty }(T)$ approaches zero, namely $\sup_{g\in BL_{1}(\ell
^{\infty }(T),[0,1])}|\Ep_{\Pr^{e}}[g(\widetilde{S}_{n})]-\Ep[g(\G [h_k] ) ]%
|\rightarrow _{P}0.$
\end{theorem}

\begin{theorem}[Bootstrap Inference on the Support Function Process\label%
{thm:bsInference}]
For any $c_{n}=O_{\Pr }(1)$ and $f\in \mathcal{F}_{c}$ we have 
\begin{equation*}
\Pr \{f(S_{n})\leq c_{n}\}-\Pr^{e}\{f(\widetilde{S}_{n})\leq
c_{n}\}\rightarrow _{P}0.
\end{equation*}%
In particular, if $\widetilde{c}_{n}(1-\critv)$ is the $(1-\critv)$-quantile
of $f(\widetilde{S}_{n})$ under $\Pr^{e}$, then 
\begin{equation*}
\Pr \{f(S_{n})\leq \widetilde{c}_{n}(1-\critv)\}\rightarrow _{P}1-\critv.
\end{equation*}%
%
%
%
%
%
%
%
%
%
%
%
\end{theorem}

\section{Application: the gender wage gap and selection\label{Sec:Empirical}}


An important question in labor economics is whether the gender wage
gap is shrinking over time. \citeasnoun{BlauKahn2007} and %
\citeasnoun{CardDiNardo2002}, among others, have noted the coincidence
between a rise in within-gender inequality and a fall in the gender
wage gap over the last 40 years. \citeasnoun{mulligan_selection_2008}
observe that the growing wage inequality within gender should induce
females to invest more in productivity. In turn, able females should
differentially be pulled into the workforce. Motivated by this
observation, they use Heckman's two-step estimator on repeated Current
Population Survey cross-sections in order to compute relative wages
for women since 1970, holding skill composition constant. They find
that in the 1970s selection into the female workforce was negative,
while in the 1990s it was positive. Moreover, they argue that the
majority of the reduction in the gender gap can be attributed to the
changes in the female workforce composition. In particular, the OLS
estimates of the log-wage gap has fallen from -0.419 in the 1970s to
-0.256 in the 1990s, though the Heckman two step estimates suggest
that once one controls for skill composition, the wage gap is -0.379
in the 1970s and -0.358 in the 1990s. Based on these results,
\citeasnoun{mulligan_selection_2008} conclude that the wage gap has
not shrunk over the last 40 years. Rather, the behavior of the OLS
estimates can be explained by a switch from negative to positive
selection into female labor force participation.

In what follows, we address the same question as
\citeasnoun{mulligan_selection_2008}, but use our method to estimate
bounds on the quantile gender wage gap without assuming a parametric
form of selection or a strong exclusion restriction.\footnote{We use
  the same data, the same variables and the same instruments as in
  their paper.} We follow their approach of comparing conditional
quantiles that ignore the selection effect, with the bounds on these
quantiles that one obtains when taking selection into account.

Our results show that we are unable to reject that the gender wage gap
declined over the period in question. This suggests that the
instruments may not be sufficiently strong to yield tight bounds and
that there may not be enough information in the data to conclude that
the gender gap has or has not declined from 1975 to 1999 without
strong functional form assumptions.

\subsection{Setup}

The \citeasnoun{mulligan_selection_2008} setup relates log-wage to
covariates in a linear model as follows: 
\begin{equation*}
\log w=x^{\prime }\beta +\varepsilon ,
\end{equation*}%
wherein $x$ includes marital status, years of education, potential
experience, potential experience squared, and region dummies, as well as
their interactions with an indicator for gender which takes the value 1 if
the individual is female, and zero otherwise. They model selection as in the
following equation:%
\begin{equation*}
u=1\left\{ z^{\prime }\gamma >0\right\} ,
\end{equation*}%
where $z=\left[ x\ \tilde{z}\right] $ and $\tilde{z}$ is marital
status interacted with indicators for having zero, one, two, or more
than two children.

For each quantile, we estimate bounds for the gender wage gap
utilizing our method. The bound equations we use are given by 
$ \theta
_{\ell }(x,v,\alpha )=p_{k}(x,v)^{\prime }\vartheta _{\ell
}^{xv}(\alpha )$, where $p_{k}(x,v)=%
\begin{bmatrix}
x & v & w
\end{bmatrix}
$, $v$ are indicators for the number of children, and $w$ consists of
years of education squared, potential experience cubed, and
education $\times$ potential experience, and $v$ interacted with
marital status. After taking the intersection of the bounds over the
excluded variables $v$, our estimated bounding functions are simply
the minimum or maximum over $v$ of $p_{k}(x,v)^{\prime }\vartheta
_{\ell}^{xv}(\alpha )$.

\providecommand{\myTable}[1]{\begin{table}\caption{#1}
    \begin{minipage}{\linewidth} \begin{center}} \providecommand{%
\myTableEnd}{\end{center}\end{minipage}\end{table}}

\providecommand{\myFigure}[1]{\begin{figure}\caption{#1}
    \begin{minipage}{\linewidth} \begin{center}} \providecommand{%
\myFigureEnd}{\end{center}\end{minipage}\end{figure}}

\subsection{Results}

Let $\bar{x}_{f}$ be a female with average (unconditional on gender)
characteristics and $\bar{x}_{m}$ be a male with average
(unconditional on gender or year) characteristics. In what follows, we report
the predicted gender wage gap for someone with average
characteristics, $\left( \bar{x}_{f}-\bar{x}_{m}\right) \beta(\indx)$.
The first two columns of table \ref{tab:boundsFull} reproduce the
results of \citeasnoun{mulligan_selection_2008}. The first column shows the
gender wage gap estimated by ordinary least squares. The second column
shows estimates from Heckman's two-step selection correction. The OLS
estimates show a decrease in the wage gap, while the Heckman selection
estimates show no change. The third column shows estimates of the
median gender wage gap from quantile regression. Like OLS, quantile
regression shows a decrease in the gender wage gap. The final two
columns show bounds on the median gender wage gap that account for
selection. The bounds are wide, especially in the 1970s. In both
periods, the bounds do not preclude a negative nor a positive gender
wage gap. The bounds let us say very little about the change in the
gender wage gap. 

\myTable{Gender wage gap estimates \label{tab:boundsFull}} 
\begin{tabular}{c|ccccc}
& OLS & 2-step & QR(0.5) & Low & High \\ \hline
1975-1979 & -0.408 & -0.360 & -0.522 & -1.242 &  0.588 \\
         & (0.003) & (0.013) & (0.003) & (0.016) & (0.061) \\
1995-1999& -0.268 & -0.379 & -0.355 & -0.623 &  0.014 \\
           & (0.003) & (0.013) & (0.003) & (0.012) & (0.010) \\
\end{tabular}
\footnotetext{ This table shows estimates of the gender wage gap
  (female $-$ male) conditional on having average characteristics. The
  first column shows OLS estimates of the average gender gap. The
  second column shows Heckman two-step estimate. The third column
  shows quantile regression estimates of the median gender wage
  gap. The fourth and fifth columns show estimates of bounds on the
  median wage gap that account for selection. Standard errors are
  shown in parentheses. The standard errors were calculated using the
  reweighting bootstrap described above.} \myTableEnd

Figure \ref{fig:all} shows the estimated quantile gender wage gaps in
the 1970s and 1990s. The solid black line shows the quantile gender
wage gap when selection is ignored. In both the 1970s and 1990s, the
gender wage gap is larger for lower quantiles. At all quantiles the
gap in the 1990s is about smaller 40\% smaller than in the
1970s. However, this result should be interpreted with caution because
it ignores selection into the labor force.

The blue line with downward pointing triangles and the red line with
upward pointing triangles show our estimated bounds on the gender wage
gap after accounting for selection. The dashed lines represent a
uniform 90\% confidence region. In both the 1970s and 1990s, the
upper bound lies below zero for low quantiles. This means that the low
quantiles of the distribution of wages conditional on having average
characteristics are lower for a woman than for a man. This difference
exists even if we allow for the most extreme form of selection
(subject to our exclusion restriction) into the labor force for
women. For quantiles at or above the median, our estimated upper bound
lies above zero and our lower bound lies below zero. Thus, high 
quantiles of the distribution of wages conditional on average
characteristics could be either higher or lower for women than for
men, depending on the true pattern of selection. For all quantiles,
there is considerable overlap between the bounded region in the 1970s
and in the 1990s. Therefore, we can essentially say nothing about the
change in the gender wage gap. It may have decreases, as
suggested by least squares or quantile regression estimates that
ignore selection. The gap may also have stayed the same, as suggested
by Heckman selection estimates. In fact, we cannot even rule out the
possibility that the gap increased. 

\myFigure{Bounds at Quantiles for full sample \label{fig:all}}
\begin{tabular}{cc}
\textbf{1975-1979} & \textbf{1995-1999} \\ 
\includegraphics[width=0.48\linewidth]{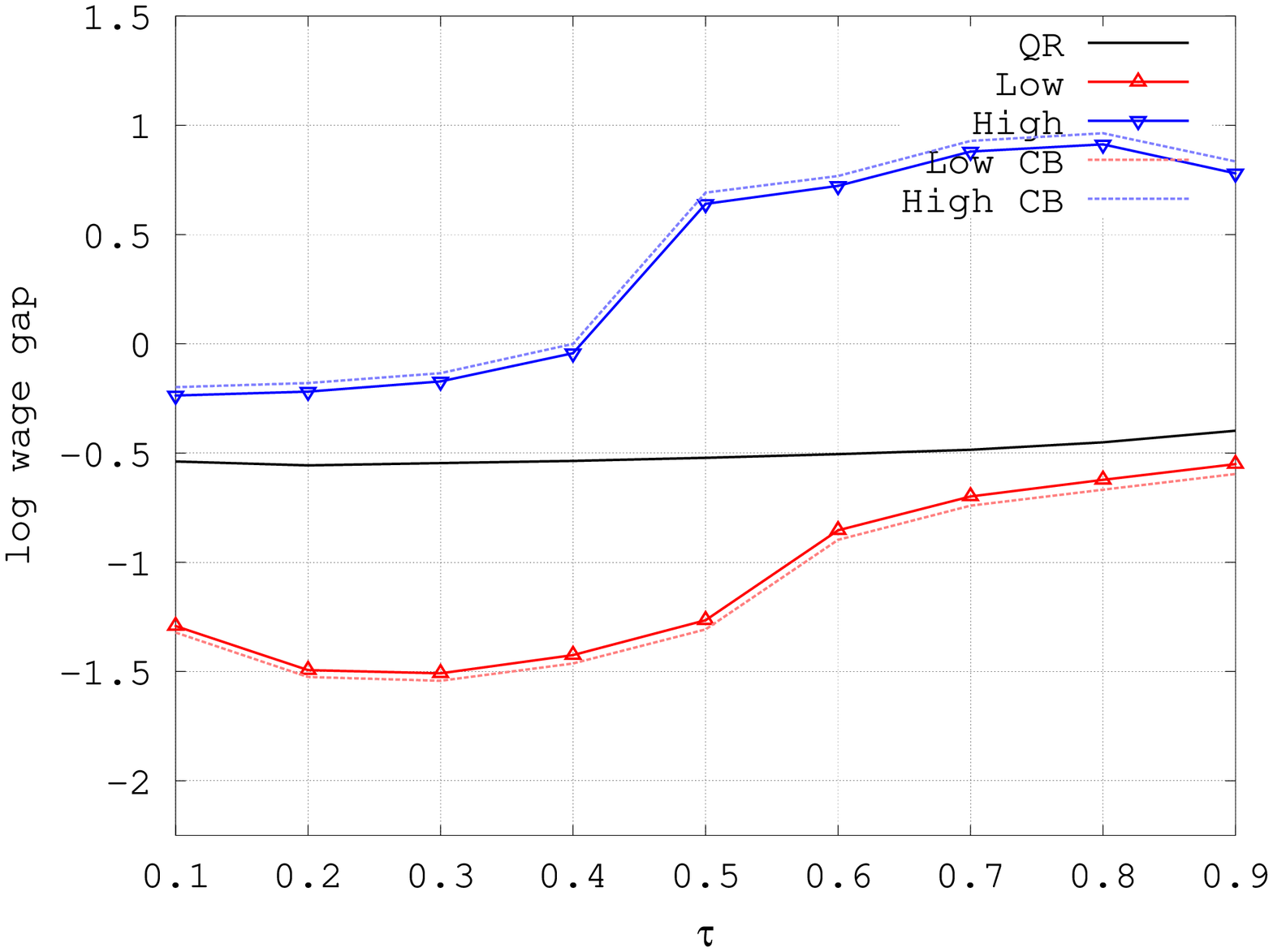} & 
\includegraphics[width=0.48\linewidth]{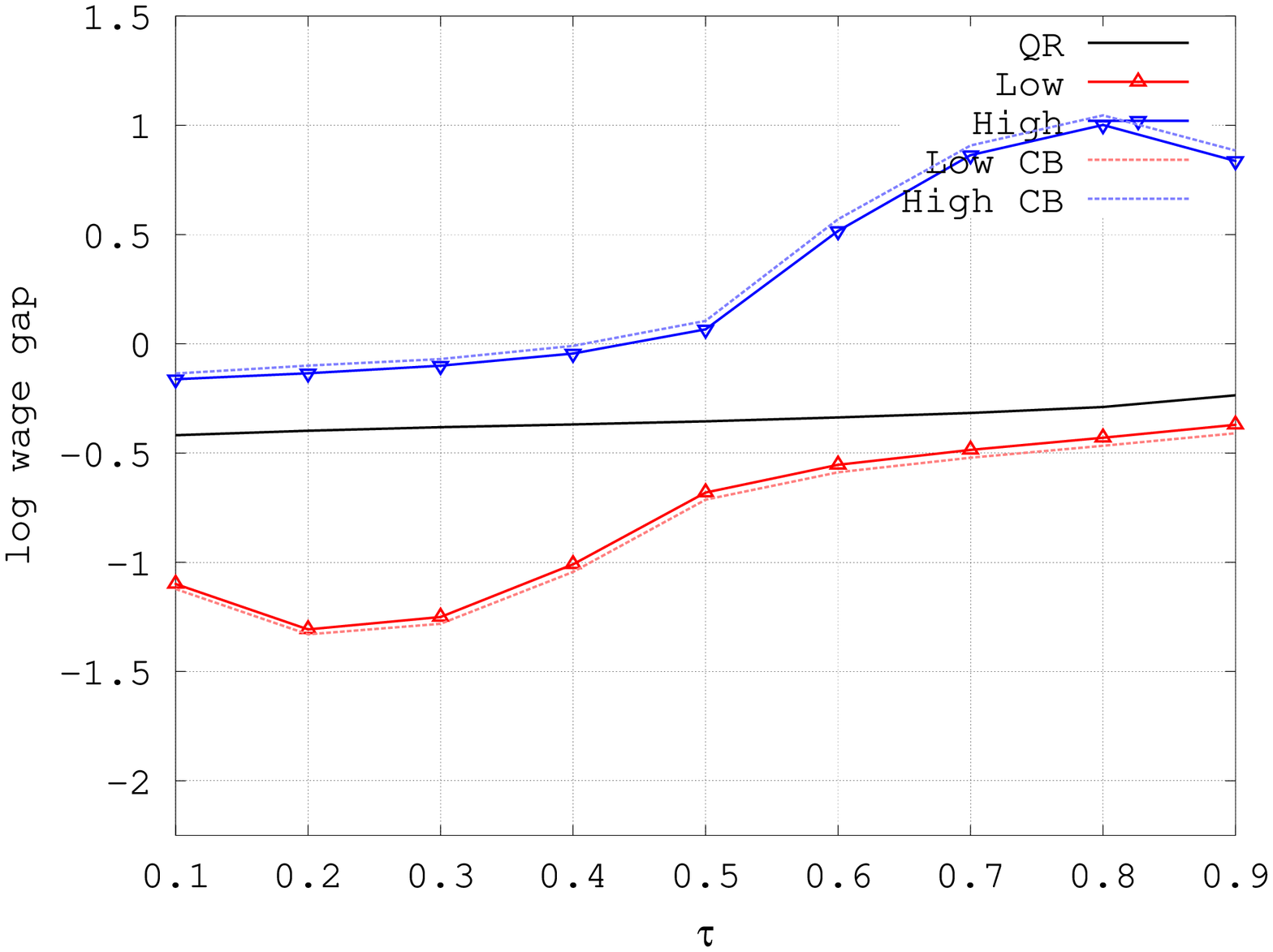}
\end{tabular}
\footnotetext{ This figure shows the estimated quantile gender wage
  gap (female $-$ male) conditional on having average
  characteristics. The solid black line shows the quantile gender wage
  gap when selection is ignored. The blue and red lines with upward
  and downward pointing triangles show upper and lower bounds that
  account for employment selection for females. The dashed lines
  represent a uniform 90\% confidence region for the bounds.}
\myFigureEnd

The bounds in figure \ref{fig:all} are tighter in the 1990s than in
the 1970s. This reflects higher female labor force participation in
the 1990s. To find even tighter bounds, we can repeat the estimation
focusing only on subgroups with higher labor force attachment. 
Figures \ref{fig:single}-\ref{fig:shi} show the estimated quantile
gender wage gap conditional on being in certain subgroups. That is,
rather than reporting the gender wage gap for someone with average
characteristics, these figures show the gender wage gap for someone
with average subgroup characteristics (e.g., unconditional on gender or year, conditional on
subgroup: marital status and educaiton level). To generate
these figures, the entire model was re-estimated using only
observations within each subgroup. 

Figures \ref{fig:single} and \ref{fig:hied} show the results for
singles and people with at least 16 years of education. The results
are broadly similar to the results for the full sample. There is
robust evidence of a gap at low quantiles, although it is only
marginally significant for the highly educated in the 1990s. As
expected, the bounds are tighter than the full sample
bounds. Nonetheless, little can be said about the gap at higher
quantiles or the change in the gap. For comparison, figure
\ref{fig:loed} shows the results for people with no more than a high
school education. These bounds are slightly wider than the full sample
bounds, but otherwise very similar. Figure \ref{fig:shi} shows results
for singles with at least a college degree. These bounds are the
tighter than all others, but still do not allow us to say anything
about the change in the gender wage gap. Also, there is no longer
robust evidence of a gap at low quantiles. A gap is possible, but we
cannot reject the null hypothesis of zero gap at all quantiles at the
10\% level. 

\myFigure{Quantile bounds for singles \label{fig:single}}
\begin{tabular}{cc}
\textbf{1975-1979} & \textbf{1995-1999} \\ 
\includegraphics[width=0.48\linewidth]{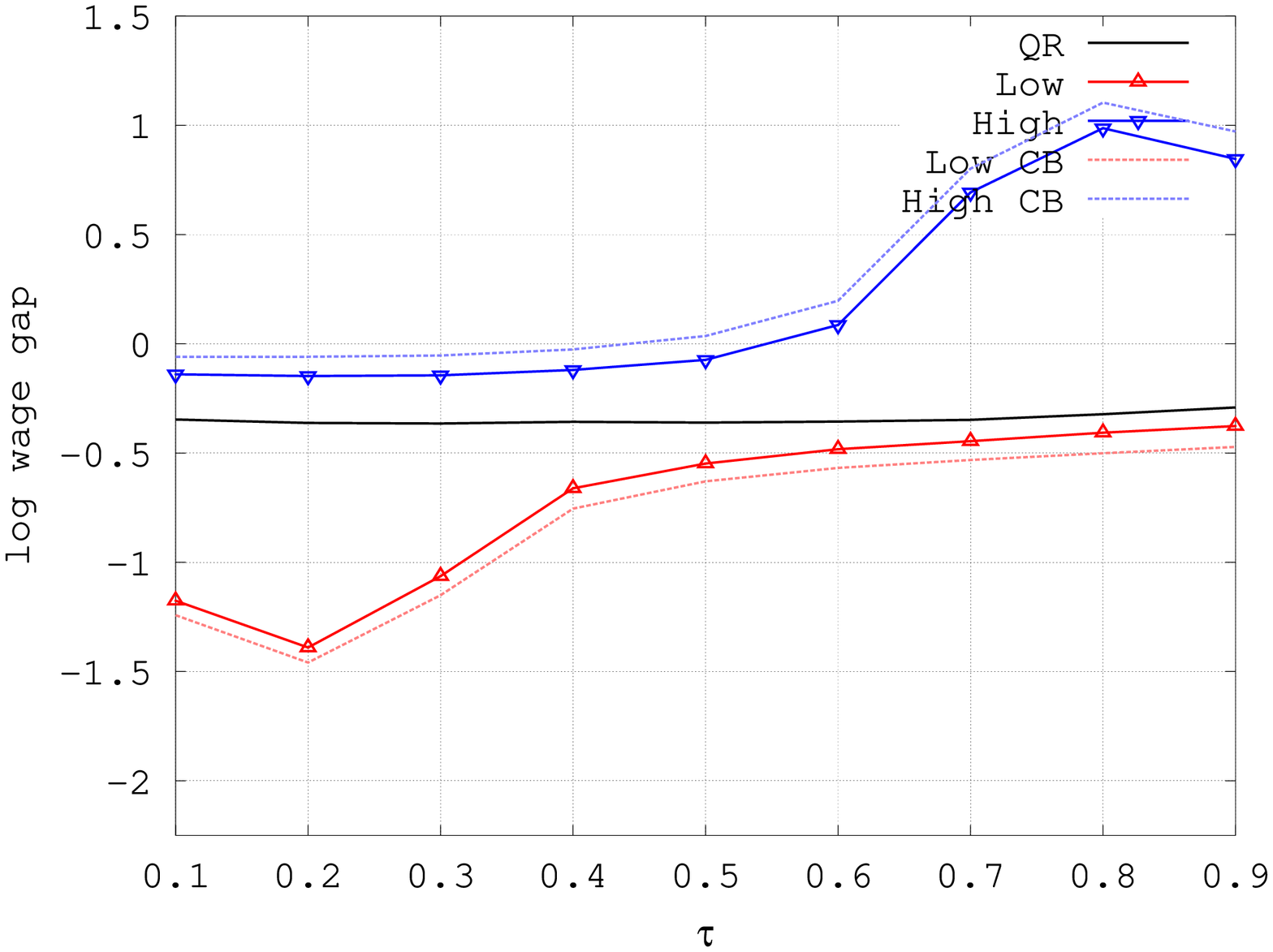} & 
\includegraphics[width=0.48\linewidth]{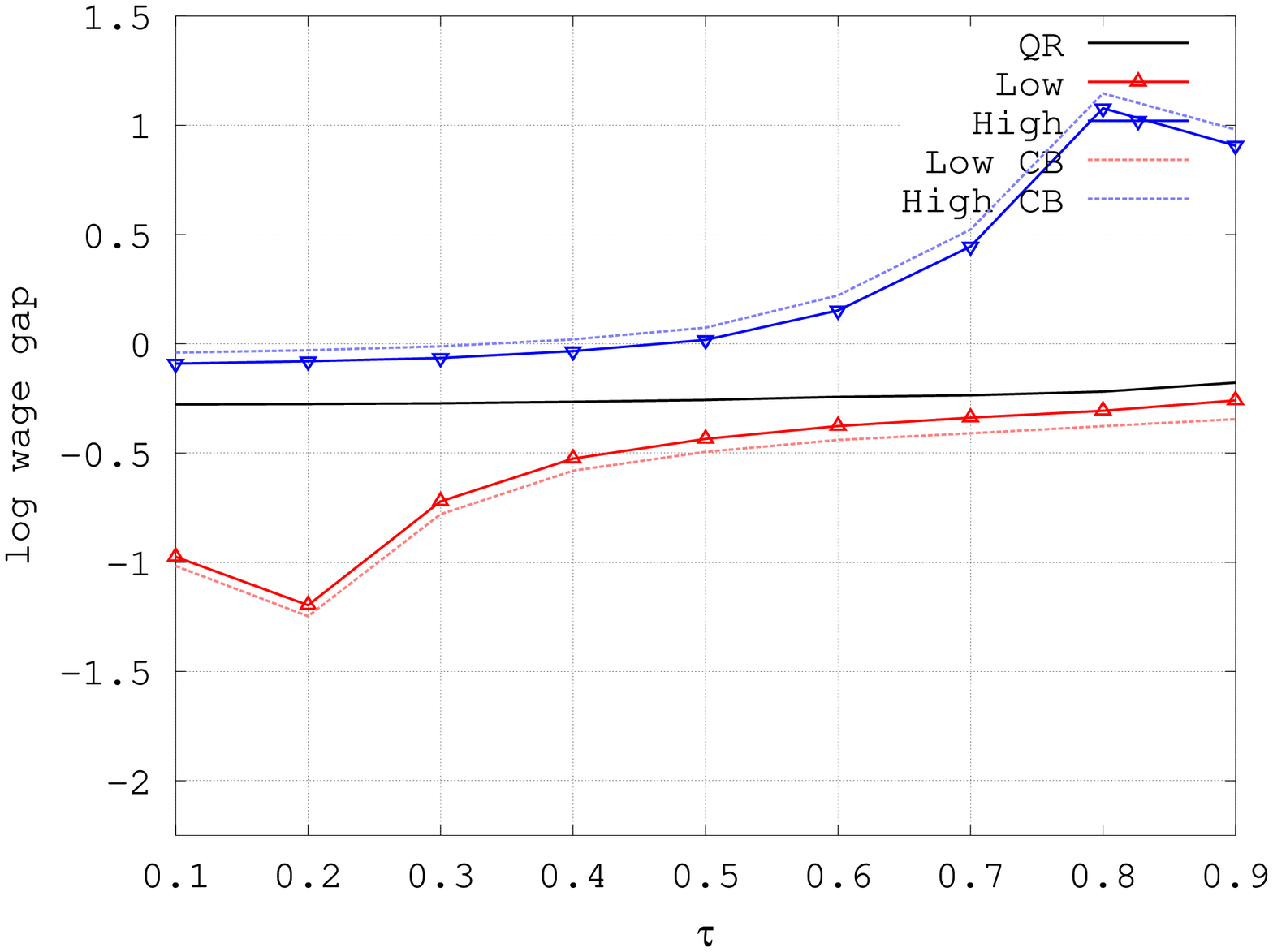}
\end{tabular}
\footnotetext{This figure shows the estimated quantile gender wage gap
  (female $-$ male) conditional on being single with average
  characteristics. The solid black line shows the quantile gender wage
  gap when selection is ignored. The blue and red lines with upward
  and downward pointing triangles show upper and lower bounds that
  account for employment selection for females. The dashed lines
  represent a uniform 90\% confidence region for the bounds.}
\myFigureEnd

\myFigure{Quantile bounds for $\geq 16$ years of education \label{fig:hied}}
\begin{tabular}{cc}
\textbf{1975-1979} & \textbf{1995-1999} \\ 
\includegraphics[width=0.48\linewidth]{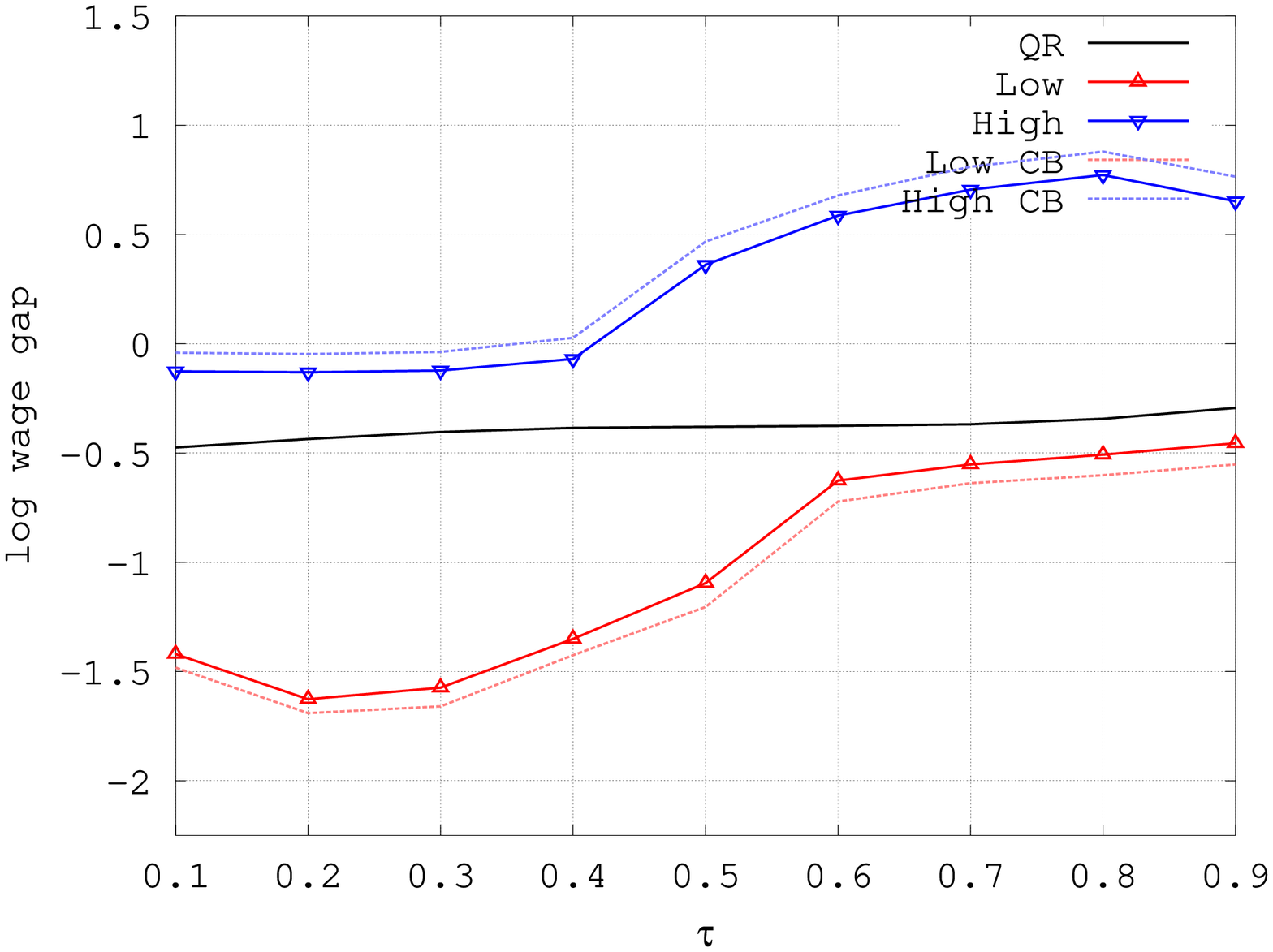} & 
\includegraphics[width=0.48\linewidth]{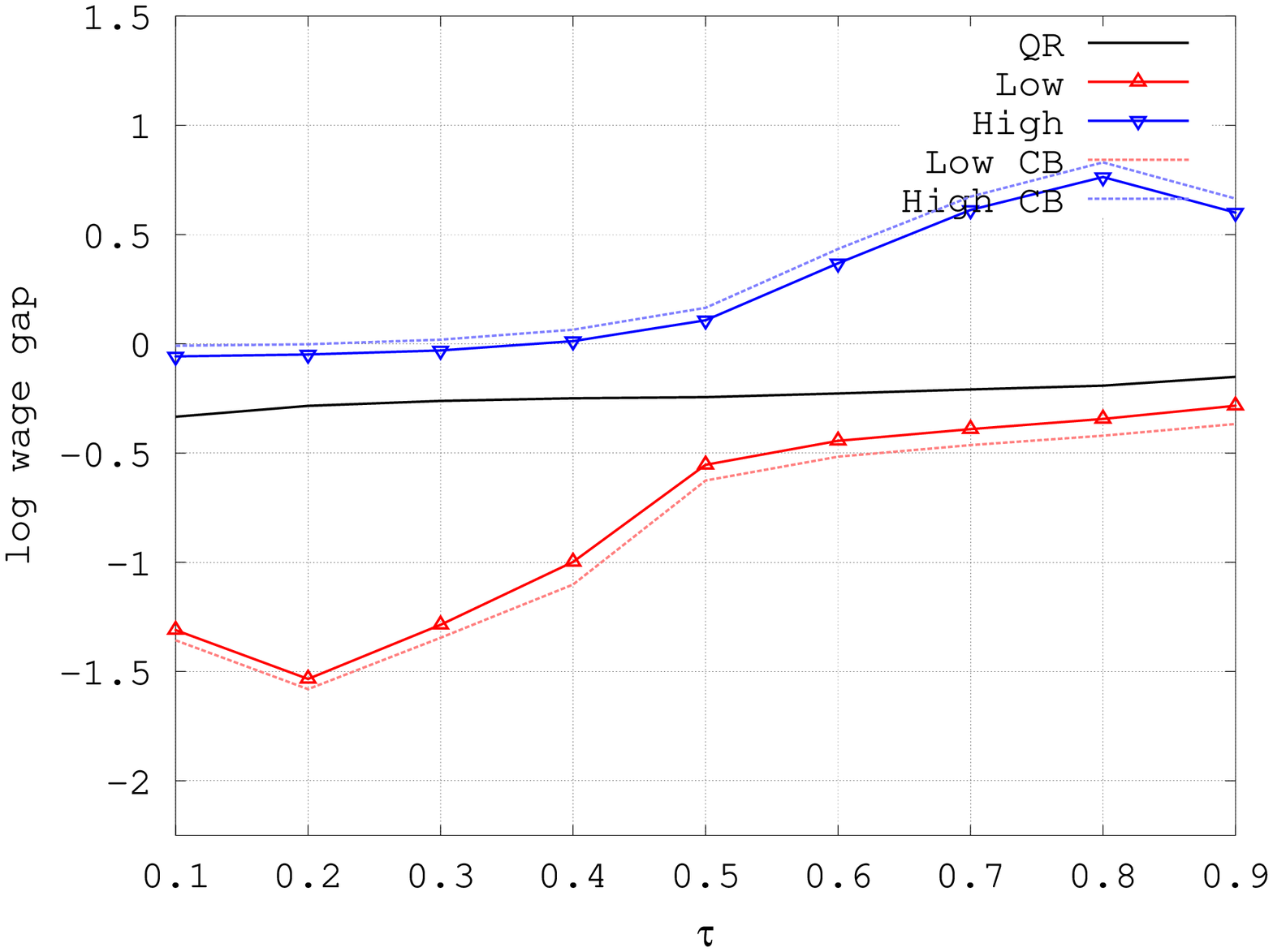}
\end{tabular}
\footnotetext{This figure shows the estimated quantile gender wage gap
  (female $-$ male) conditional on having at least 16 years of
  education with average characteristics. The solid black line shows
  the quantile gender wage gap when selection is ignored. The blue and
  red lines with upward and downward pointing triangles show upper and
  lower bounds that account for employment selection for females. The
  dashed lines represent a uniform 90\% confidence region for the
  bounds.}  
\myFigureEnd

\myFigure{Quantile bounds for $\leq 12$ years of education \label{fig:loed}}
\begin{tabular}{cc} \textbf{1975-1979} & \textbf{1995-1999} \\ 
\includegraphics[width=0.48\linewidth]{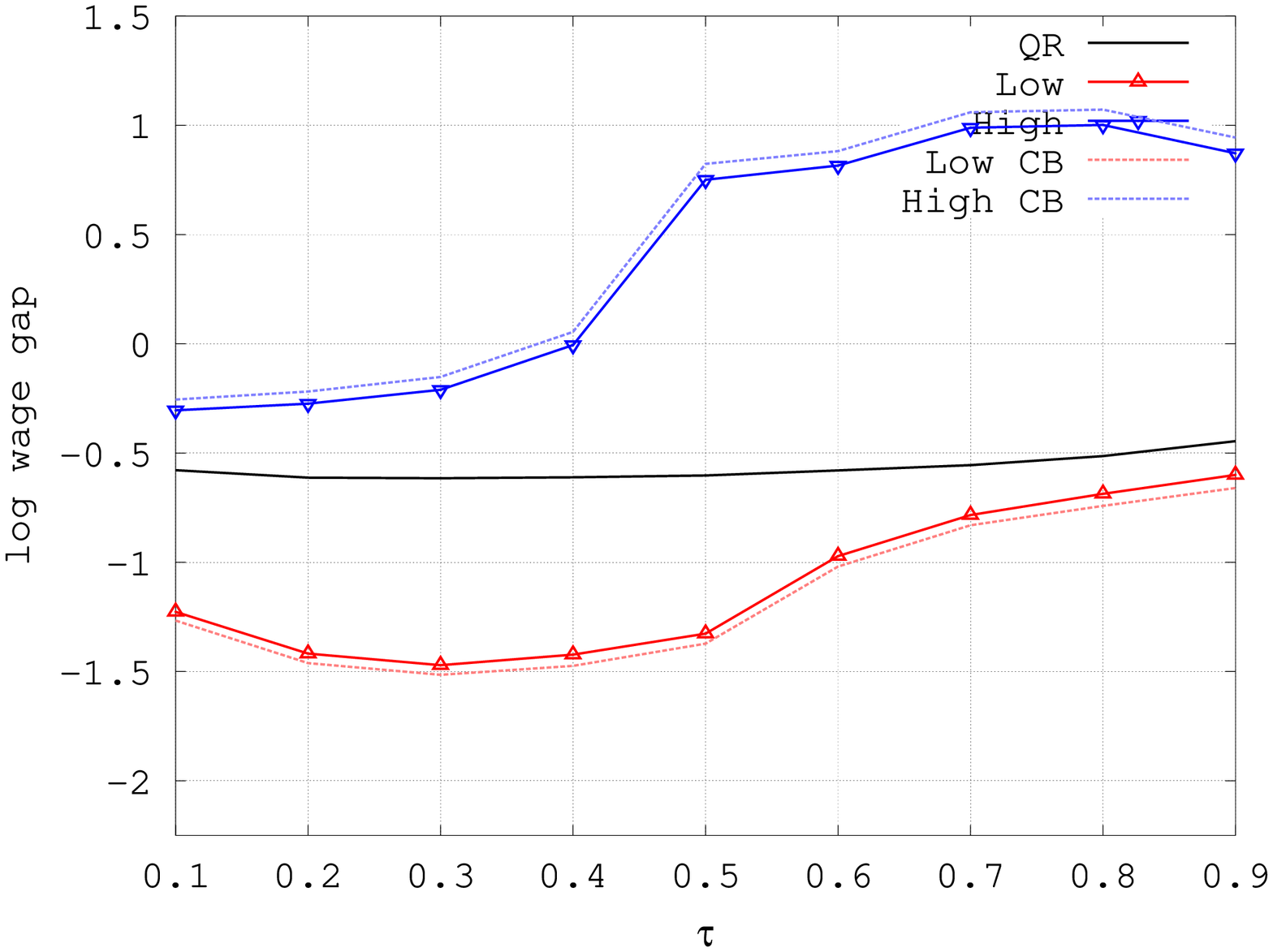} & 
\includegraphics[width=0.48\linewidth]{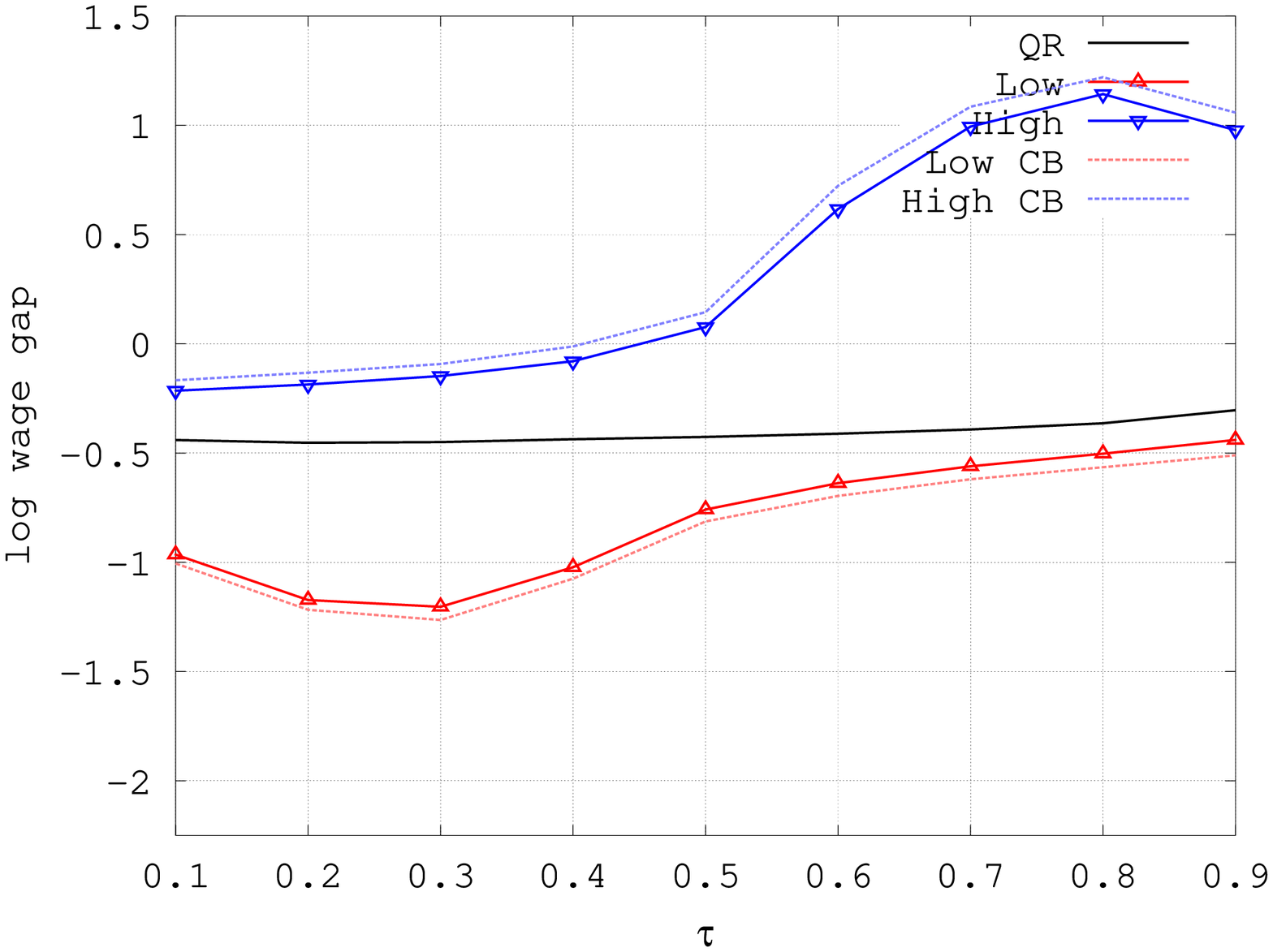}
\end{tabular}
\footnotetext{This figure shows the estimated quantile gender wage gap
  (female $-$ male) conditional on having 12 or fewer years of
  education with average characteristics. The solid black line shows
  the quantile gender wage gap when selection is ignored. The blue and
  red lines with upward and downward pointing triangles show upper and
  lower bounds that account for employment selection for females. The
  dashed lines represent a uniform 90\% confidence region for the
  bounds.}  
\myFigureEnd

\myFigure{Quantile bounds for $\geq 16$ years of education and single \label{fig:shi}}
\begin{tabular}{cc}
\textbf{1975-1979} & \textbf{1995-1999} \\ 
\includegraphics[width=0.48\linewidth]{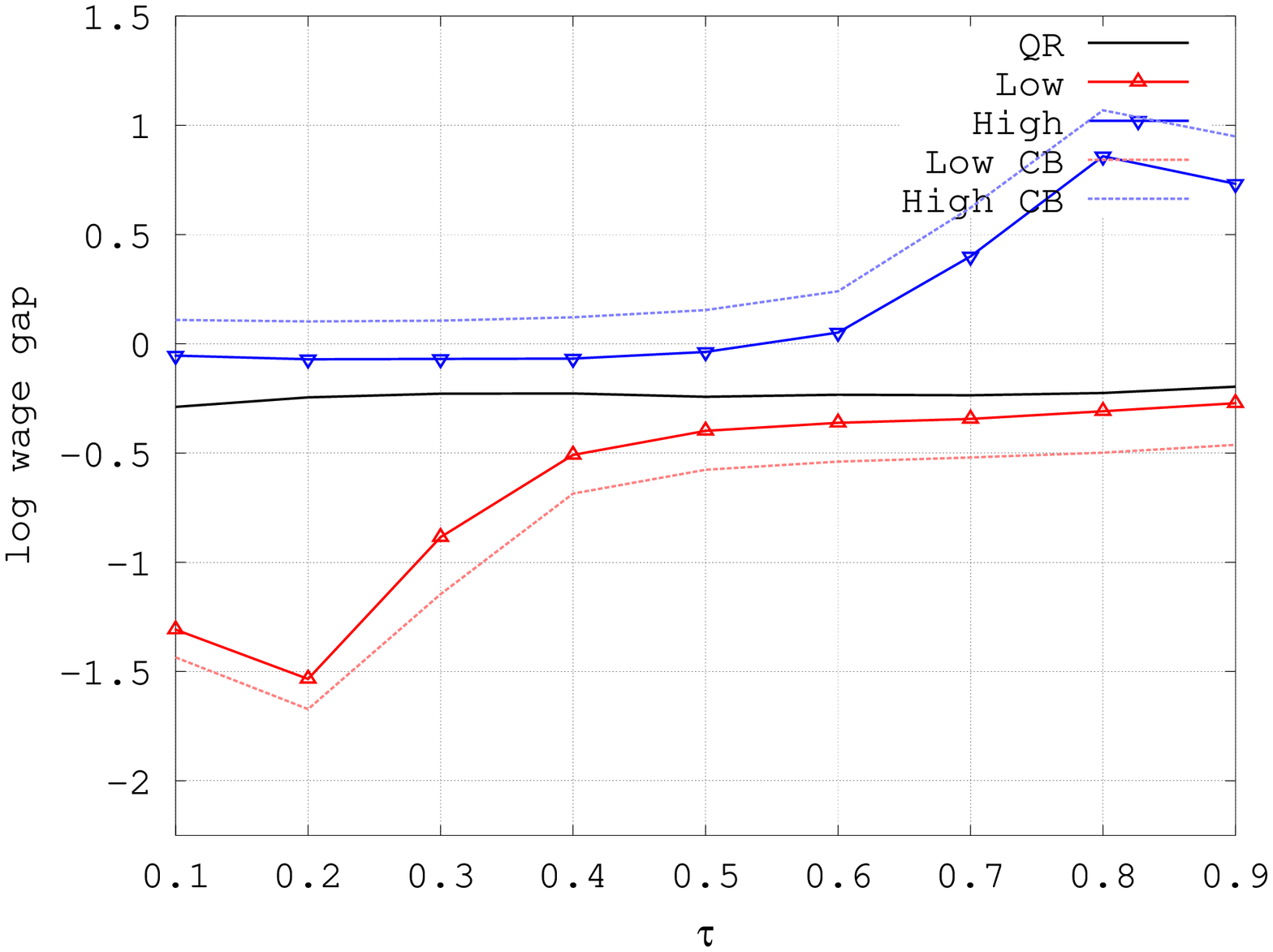} & 
\includegraphics[width=0.48\linewidth]{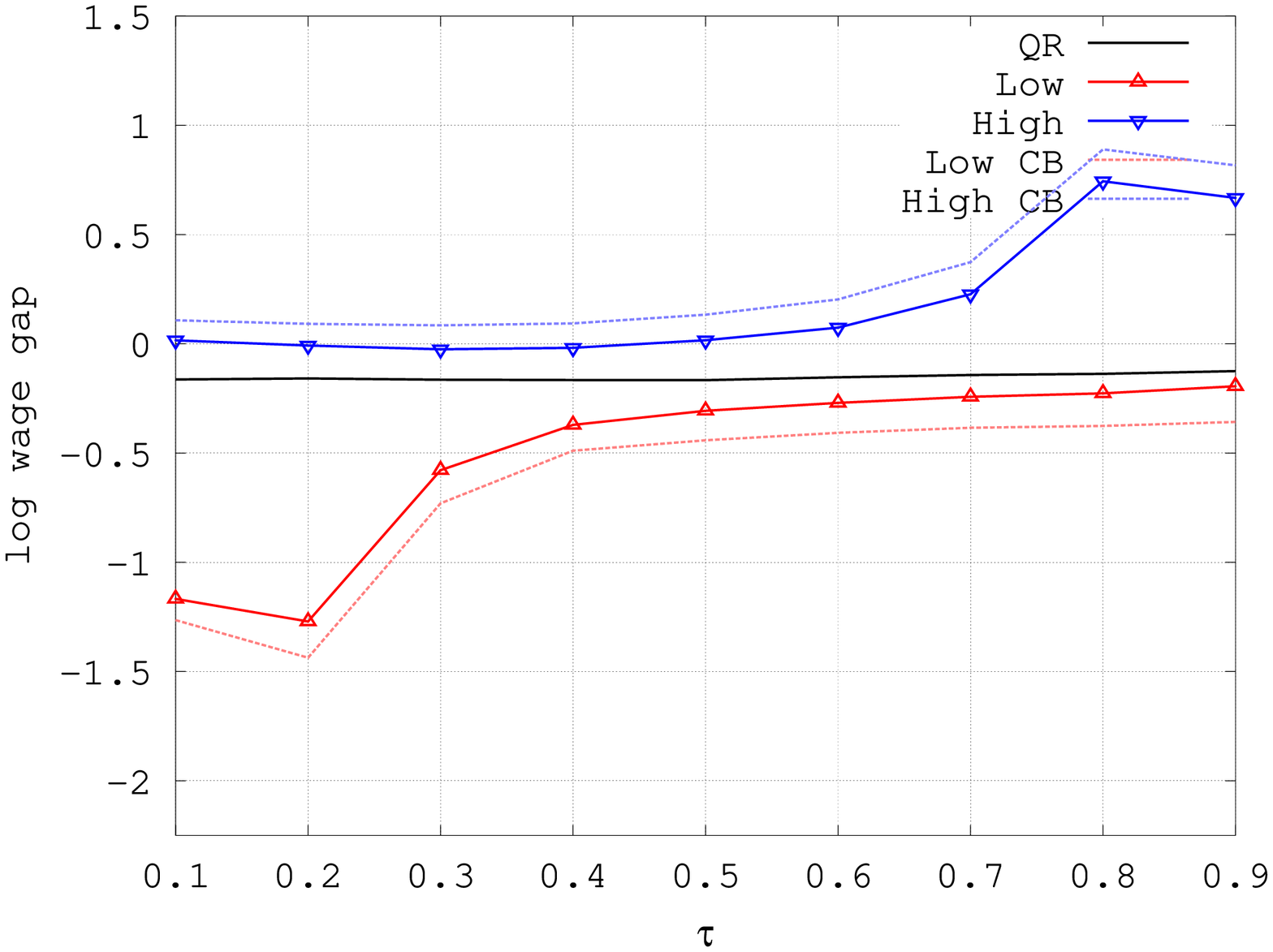}
\end{tabular}
\footnotetext{This figure shows the estimated quantile gender wage gap
  (female $-$ male) conditional on being single with at least 16 years of
  education and average characteristics. The solid black line shows
  the quantile gender wage gap when selection is ignored. The blue and
  red lines with upward and downward pointing triangles show upper and
  lower bounds that account for employment selection for females. The
  dashed lines represent a uniform 90\% confidence region for the
  bounds.}  
\myFigureEnd

\clearpage 

\subsubsection{With restrictions on selection}

\citeasnoun{bgim2007} study changes in the distribution of wages in the
UK. Like us, they allow for selection by estimating quantile
bounds. Also, like us, \citeasnoun{bgim2007} find that the estimated bounds
are quite wide. As a result, they explore various restrictions to
tighten the bound. One restriction is to assume that the wages of the
employed stochastically dominates the distribution of wages for those
not working. This implies that the observed quantiles of wages
conditional on employment are an upper bound the quantiles of wages
not conditional on employment. 

Figure \ref{fig:allsd} shows results imposing stochastic dominance for
the full sample and for highly educated singles. Stochastic dominance
implies that the upper bound coincides with the quantile regression
estimate. With stochastic dominance there is robust evidence of a
gender wage gap at all quantiles in both the 1970s and 1990s, for both
the full sample and the single highly educated subsample. The bounds
with stochastic dominance are much tighter than without. In fact, it
appears that they may be tight enough to say something about the
change in the gender wage gap. Accordingly, figure \ref{fig:changesd}
shows the estimated bounds for the change in the gender wage gap. It
shows results for both the full sample and the single high education
subsample. For the full sample, the estimated bounds include zero at
low and moderate quantiles. At the 0.6 and higher quantiles, there is
significant evidence that the gender wage gap decreased by
approximately 0.15 log dollars. For highly educated singles, the
change in the gender wage gap is not significantly different from zero
for any quantiles.

\myFigure{Quantile bounds for full sample imposing stochastic
  dominance \label{fig:allsd}}
\begin{tabular}{cc}
  \multicolumn{2}{c}{Full Sample} \\
  \textbf{1975-1979} & \textbf{1995-1999} \\ 
\includegraphics[width=0.48\linewidth]{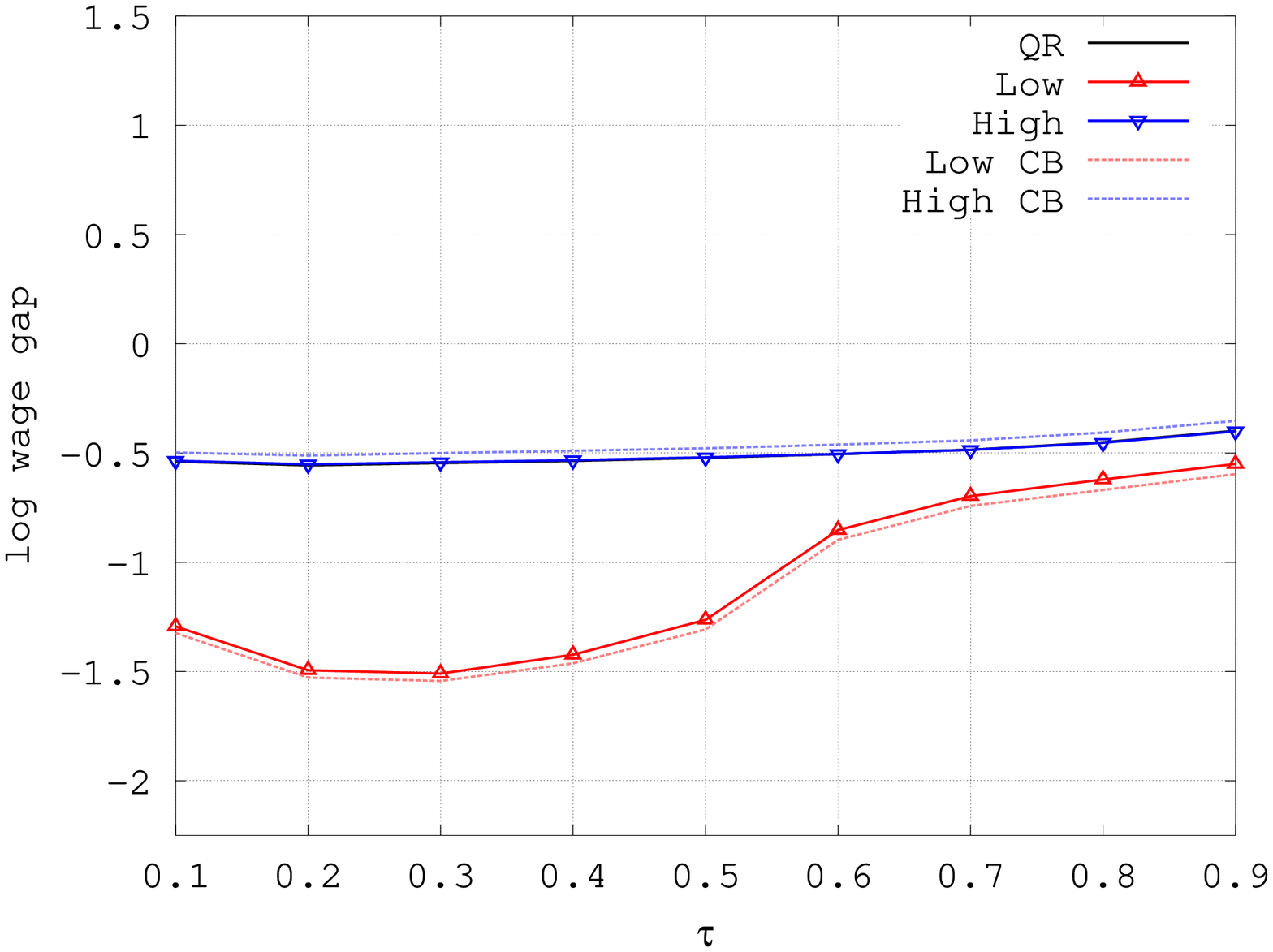} & 
\includegraphics[width=0.48\linewidth]{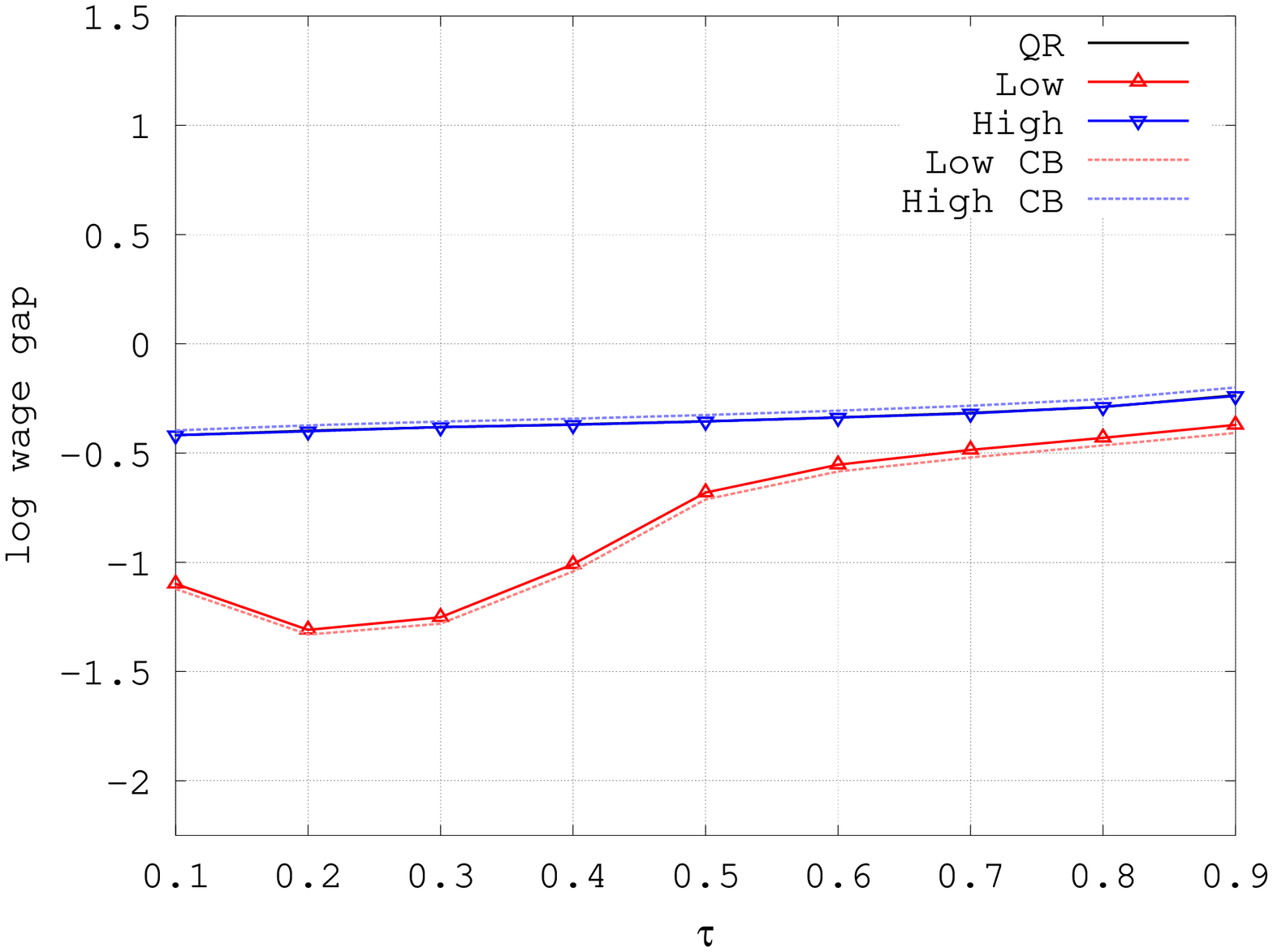} \\
\multicolumn{2}{c}{Single and $\geq 16$ years of education} \\
\textbf{1975-1979} & \textbf{1995-1999} \\ 
\includegraphics[width=0.48\linewidth]{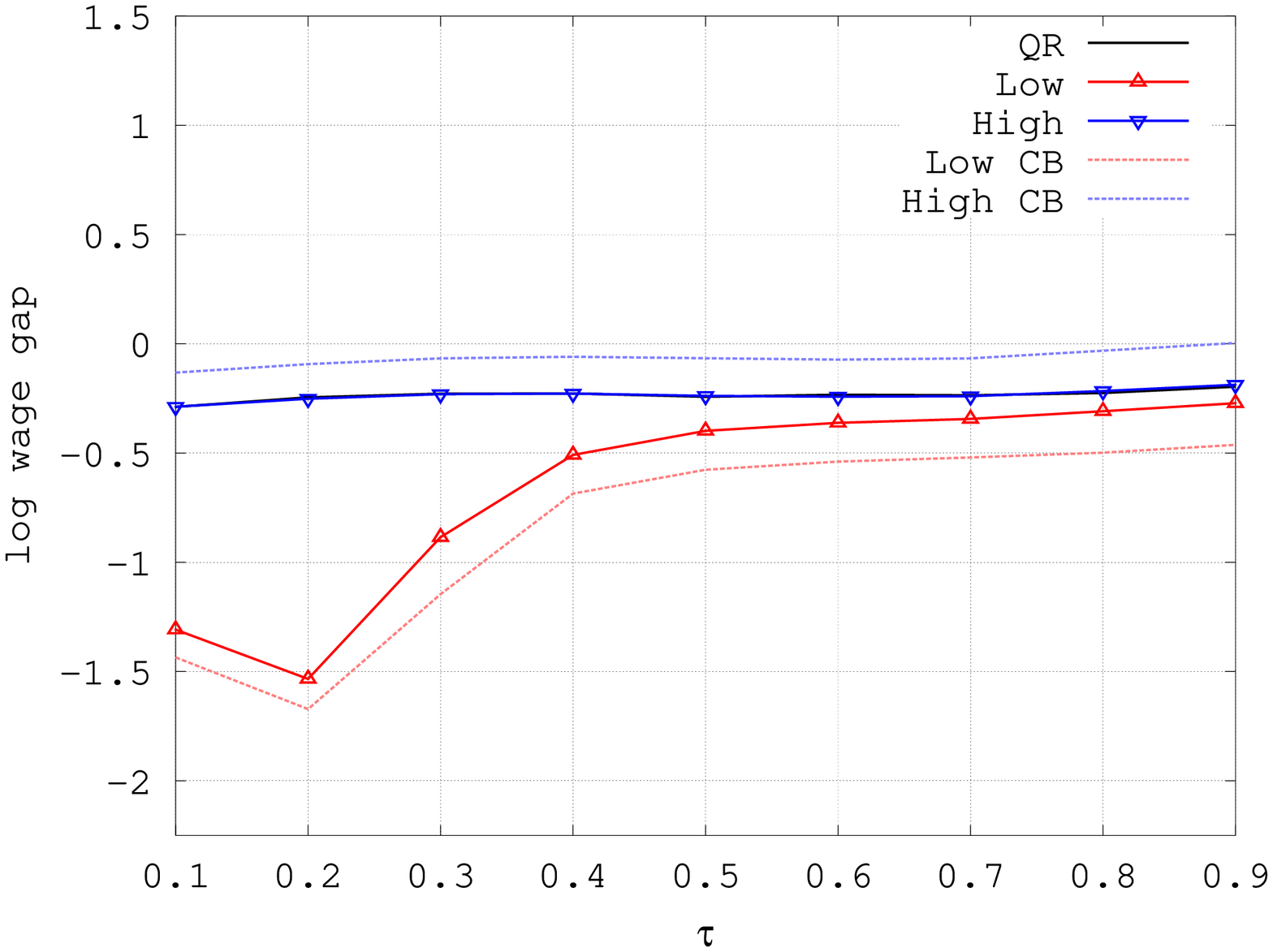} & 
\includegraphics[width=0.48\linewidth]{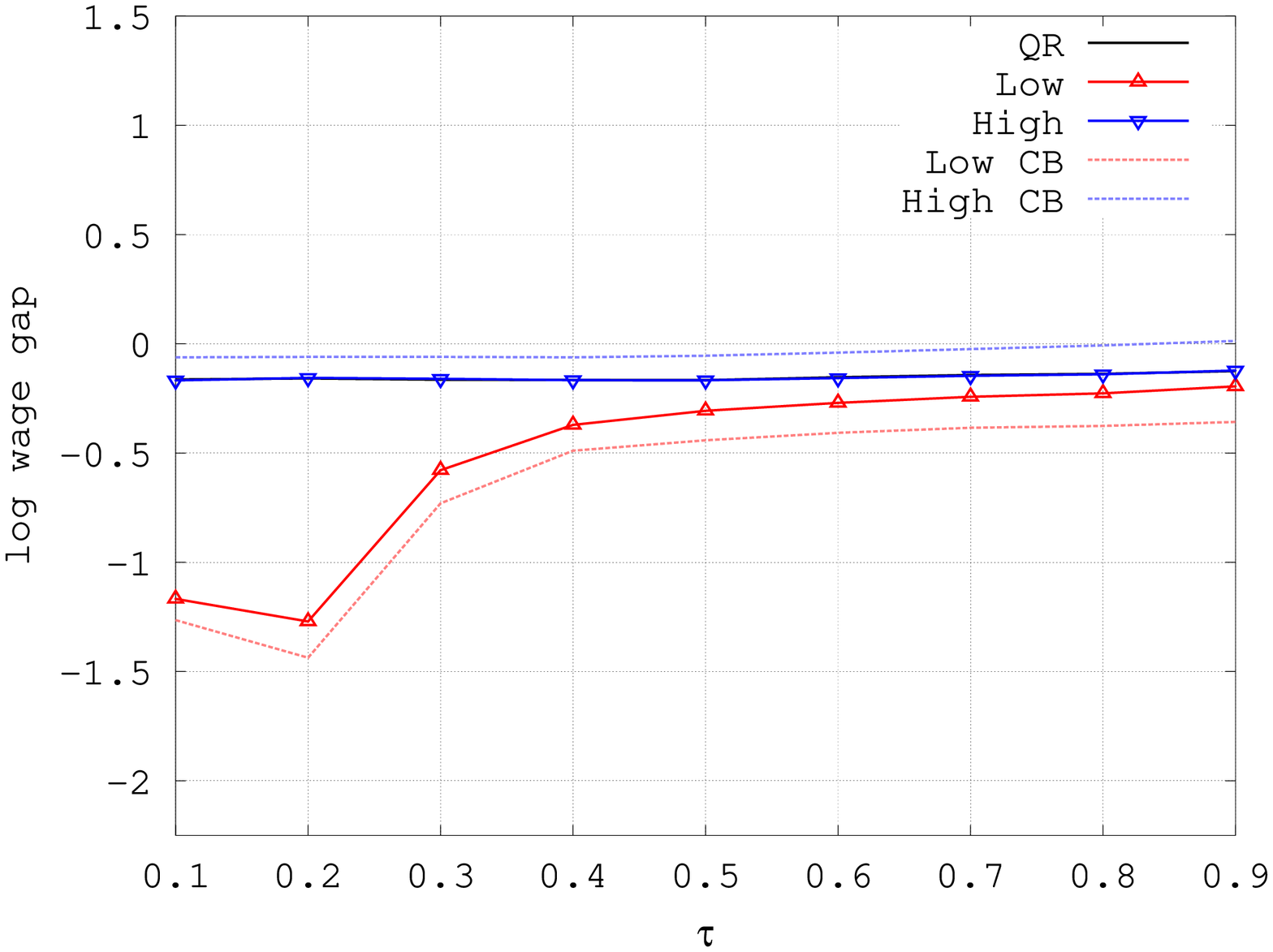}
\end{tabular}
\footnotetext{ This figure shows the estimated quantile gender wage
  (female $-$ male) conditional on average characteristics. The solid
  black line shows the quantile gender wage when selection is
  ignored. The blue and red lines with upward and downward pointing
  triangles show upper and lower bounds that account for employment
  selection for females. The dashed lines represent a uniform 90\%
  confidence region for the bounds.}  
\myFigureEnd

\myFigure{Quantile bounds for the change in the gender wage gap
  imposing stochastic dominance \label{fig:changesd}}  
\begin{tabular}{cc}
\textbf{Full sample} & \textbf{Single and $\geq$ 16 years education} \\ 
\includegraphics[width=0.48\linewidth]{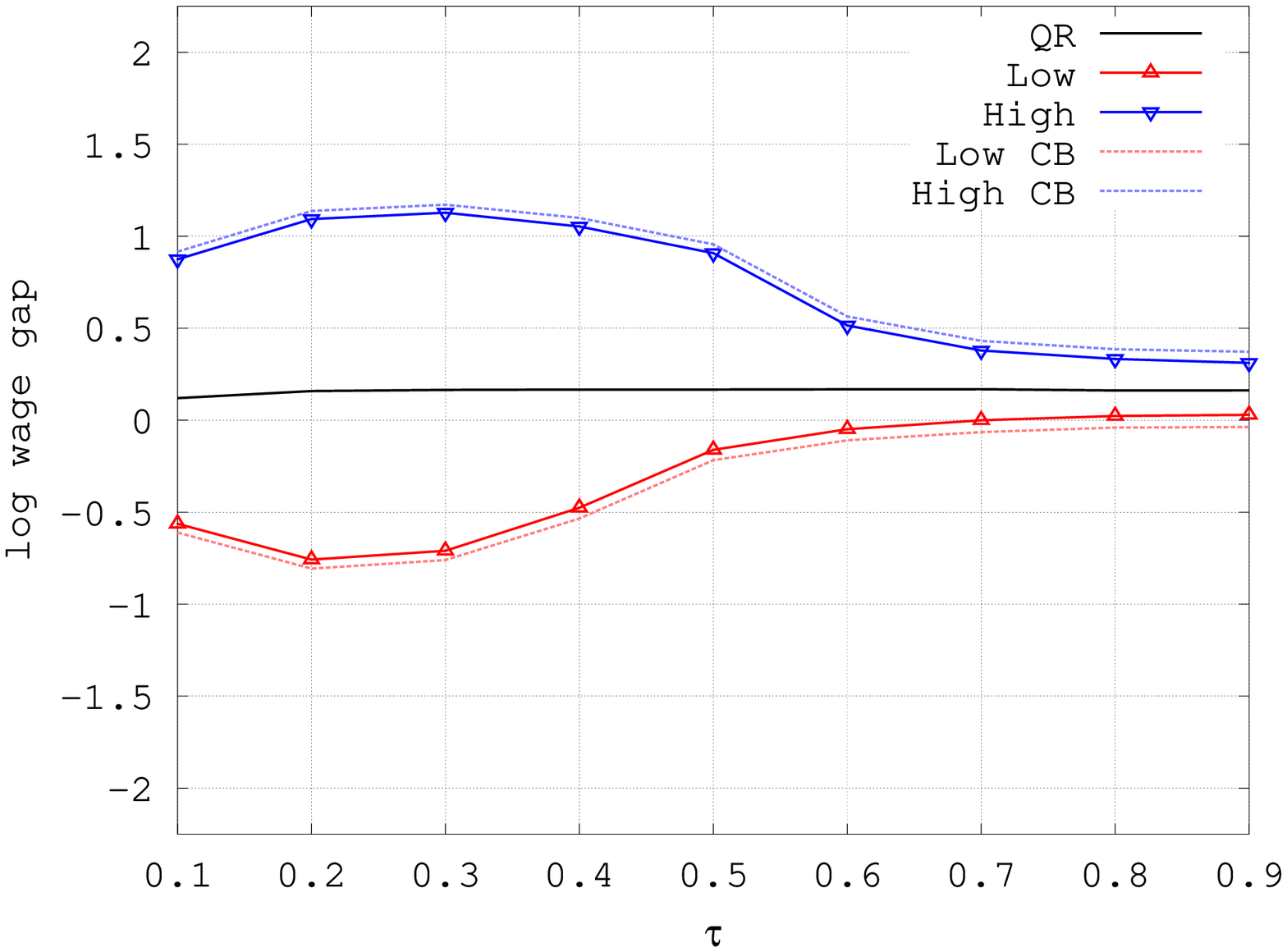} & 
\includegraphics[width=0.48\linewidth]{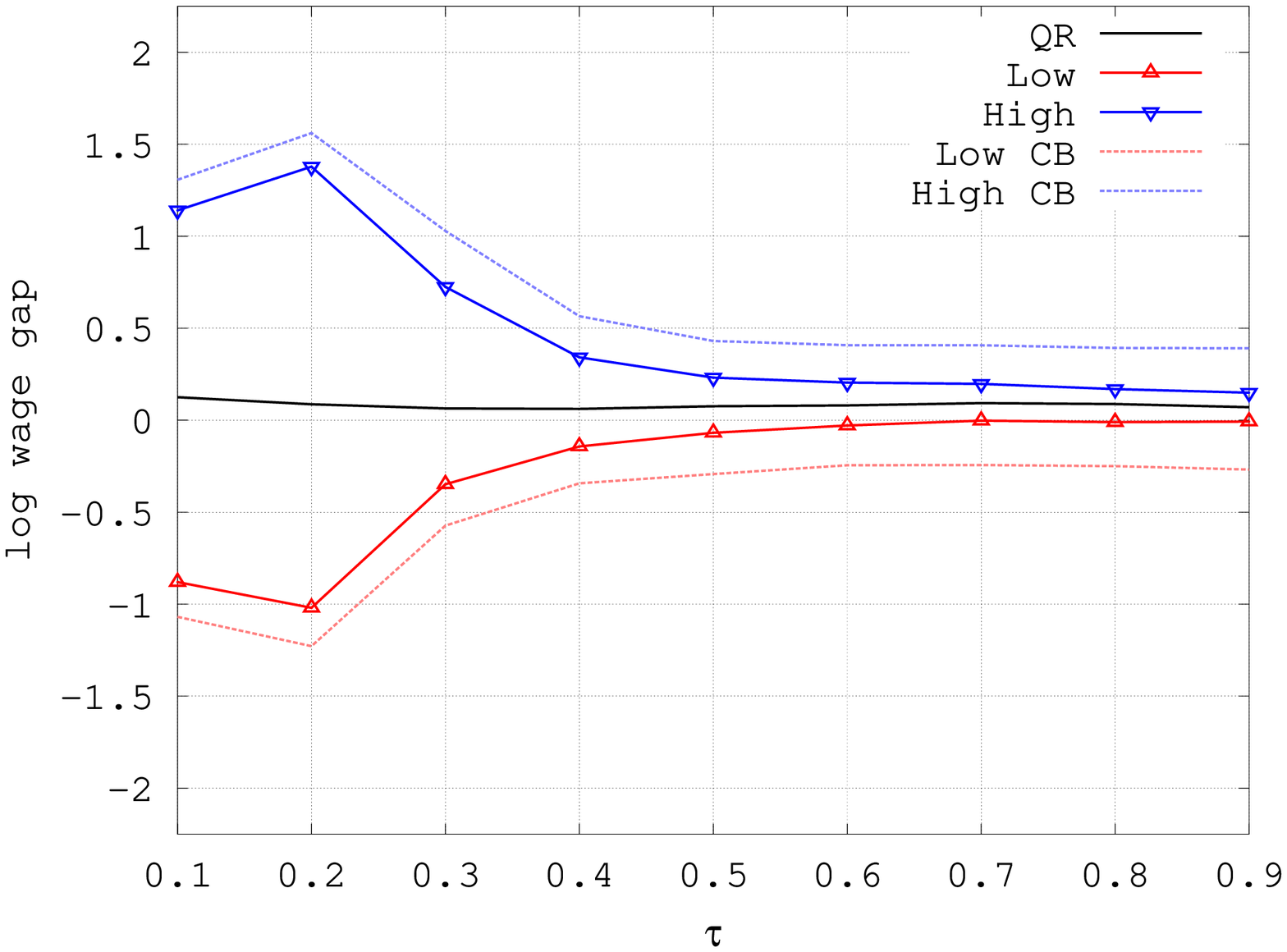}
\end{tabular}
\footnotetext{This figure shows the estimated change (1990s $-$ 1970s)
  in the quantile gender wage gap (female $-$ male) conditional on
  having average characteristics. The solid black line shows the
  quantile gender wage gap when selection is ignored. The blue and red
  lines with upward and downward pointing triangles show upper and
  lower bounds that account for employment selection for females. The
  dashed lines represent a uniform 90\% confidence region for the
  bounds.}
\myFigureEnd

The assumption of positive selection into employment is not
innocuous. It may be violated if there is a strong positive
correlation between potential wages and reservation wages. This may be
the case if there is positive assortative matching in the marriage
market. Women with high potential wages could marry men with high
wages, making these high potential wage women less likely to
work. Also, the conclusion of \citeasnoun{mulligan_selection_2008}
that there was a switch from adverse selection into the labor market
in the 1970s to advantageous selection in the 1990s implies that
stochastic dominance did not hold in the 1970s. Accordingly, we also
explore some weaker restrictions. \citeasnoun{bgim2007} propose a
median restriction --- that the median wage offer for those not
working is less than or equal to the median observed wage. This
restriction implies the following bounds on the distribution of wages
\begin{align*}
  F(y|x,u=1)\mathrm{P}(u=1|x)  + 1\{y \geq Q_y(0.5|x,u=1)\}
  0.5\Pr(u=0|x) \leq \\ \leq F(y|x) \leq
  F(y|x,u=1)\mathrm{P}(u=1|x)+\mathrm{P}(u=0|x),
\end{align*}
where $y$ is wage and $u=1$ indicates employment. Transforming these
into bounds on the conditional quantiles yields
\begin{equation*}
  Q_{0}\left( \indx|x\right) \leq Q_{y}\left( \indx|x\right) \leq Q_{1}\left( 
    \indx|x\right),
\end{equation*}
where 
\begin{align*}
  Q_{0}(\indx|x) &=
  \begin{cases}
    Q_{y}\left( \left. \frac{\indx-\mathrm{P}(u=0|x)}{\mathrm{P}(u=1|x)}
      \right\vert x,u=1\right) & \text{ if }\indx\geq \mathrm{P}(u=0|x) \\ 
    y_{0} & \text{ otherwise}
\end{cases},
\end{align*}
and
\begin{align*}
Q_{1}(\indx|x) &=
\begin{cases}
  Q_{y}\left( \left. \frac{\indx}{\mathrm{P}(u=1|x)}\right\vert x,u=1\right) & 
  \text{ if }\indx < 0.5 \text{ \& } \indx\leq \mathrm{P}(u=1|x) \\ 
  Q_{y}\left( \left. \frac{\indx - 0.5\Pr(u=0|x)}{\mathrm{P}(u=1|x)}\right\vert x,u=1\right) & 
  \text{ if }\indx \geq 0.5 \text{ \& } \indx\leq \frac{1+\mathrm{P}(u=1|x)}{2} \\ 
  y_{1} & \text{ otherwise}
\end{cases}.
\end{align*} 
As above, we can also express $Q_0(\indx|x)$ and $Q_1(\indx|x)$ as the
$\indx$ conditional quantiles of $\tilde{y}_0$ and
$\tilde{y}_1$ where
\begin{align*}
  \tilde{y}_0 = & y1\left\{ u=1\right\} +y_{0}1\left\{ u=0\right\} 
\end{align*}
and
\begin{align*}
  \tilde{y}_1 = &  
  \begin{cases} 
    y 1 \left\{ u = 1 \right\} + y_1 1 \left\{ u = 0  \right\} & 
    \text{ with probability } 0.5 \\
    y 1 \left\{ u = 1 \right\} + Q_y(0.5|x,u=1) 1 \left\{ u = 0  \right\} & 
    \text{ with probability } 0.5 
  \end{cases}.
\end{align*}
We can easily generalize this median restriction by assuming the
$\indx_1$ quantile of wages conditional on working is greater than or
equal to the $\indx_0$ quantile of wages conditional on not
working.  In that case, the bounds can still be expressed as $\indx$
conditional quantiles of $\tilde{y}_0$ and $\tilde{y}_1$ with
$\tilde{y}_0$ as defined above and
\begin{align*}
  \tilde{y}_1 = 
  \begin{cases} 
    y 1 \left\{ u = 1 \right\} + y_1 1 \left\{ u = 0  \right\} & 
    \text{ with probability } (1-\indx_0) \\
    y 1 \left\{ u = 1 \right\} + Q_y(\indx_1|x,u=1) 1 \left\{ u = 0  \right\} & 
    \text{ with probability } \indx_0 
  \end{cases}.
\end{align*}  
We can even impose a set of these restrictions for
$(\indx_1, \indx_0) \in \mathcal{R} \subseteq \indxSet \times
\indxSet$. Stochastic dominance is equivalent to imposing this
restriction for $\indx_1=\indx_0$ for all $\indx_1 \in [0,1]$.  

Figure \ref{fig:allmr} show estimates of the gender wage gap with the
median restriction. The results are qualitatively similar to the
results without the restriction. As without the restriction, we obtain
robust evidence of a gender wage gap at low quantiles in both the
1970s and 1990s, and there is substantial overlap in the bounds
between the two periods, so we cannot say much about the change in the
gender wage gap. The main difference with the median restriction is
that there is also robust evidence of a gender gap at quantiles
0.4-0.7, as well as at lower quantiles. 

\myFigure{Quantile bounds imposing the median restriction
   \label{fig:allmr}}
\begin{tabular}{cc}
  \multicolumn{2}{c}{Full sample} \\
\textbf{1975-1979} & \textbf{1995-1999} \\ 
\includegraphics[width=0.48\linewidth]{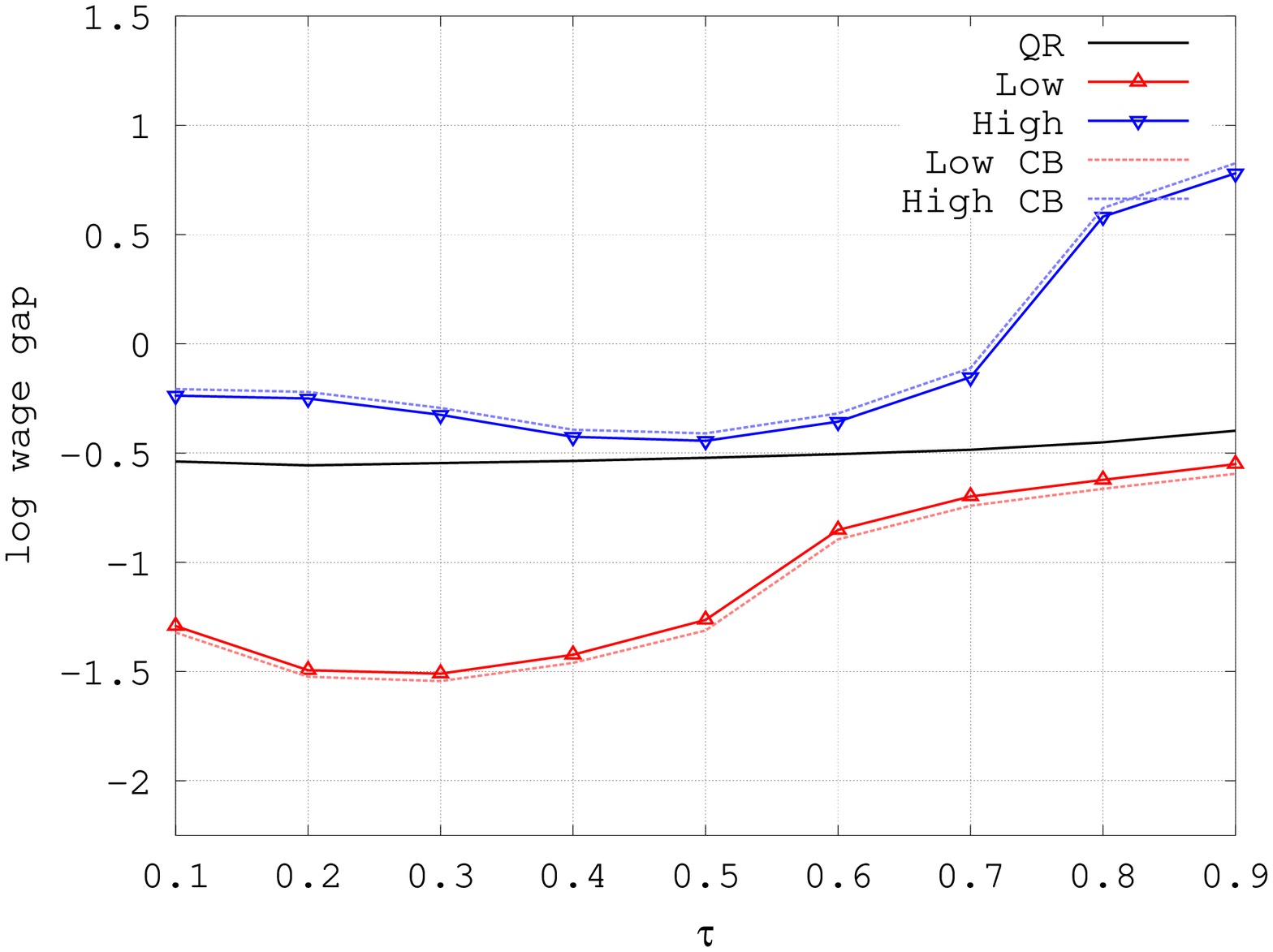} & 
\includegraphics[width=0.48\linewidth]{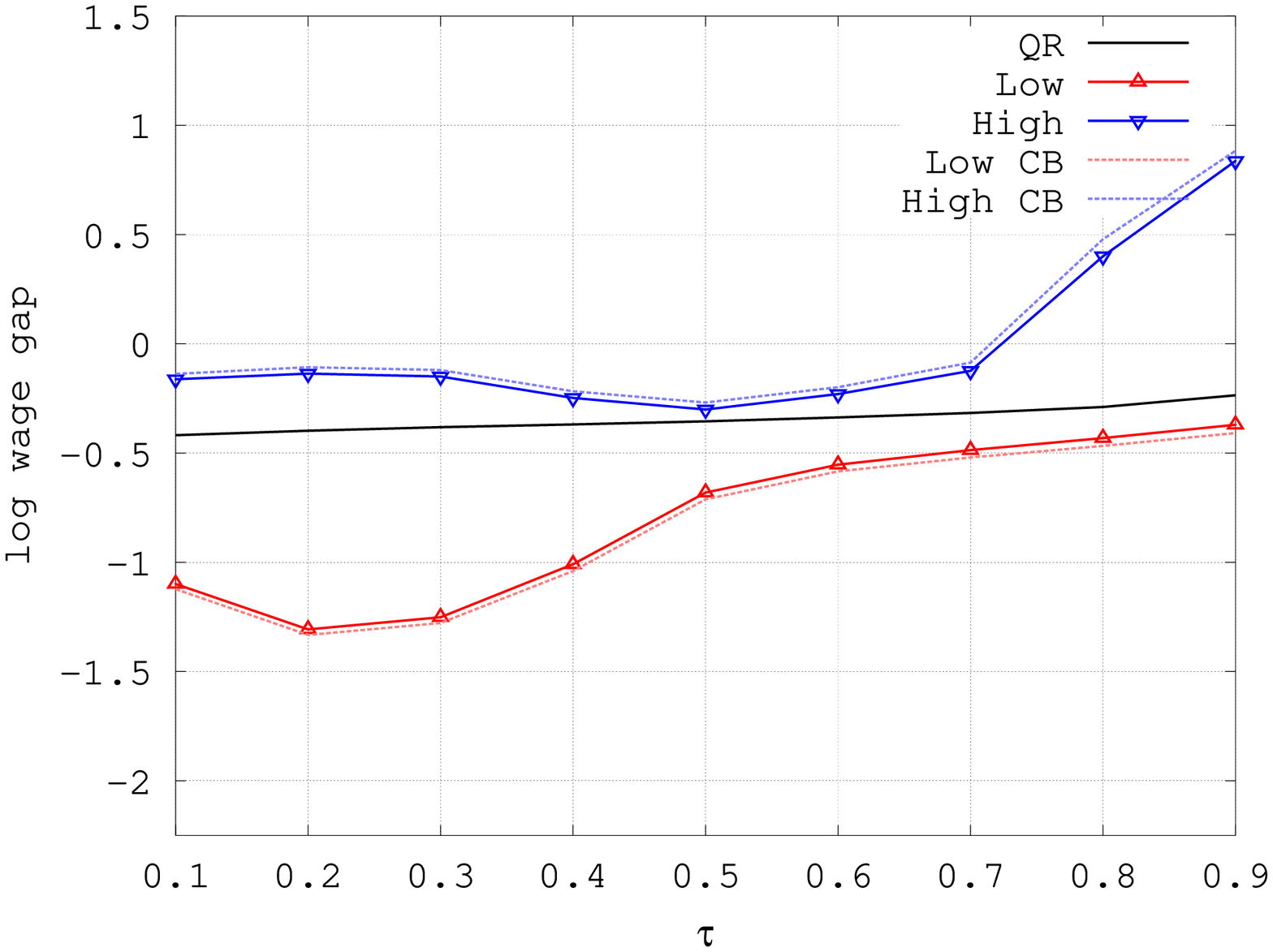} \\
\multicolumn{2}{c}{Single and $\geq 16$ years of education} \\
\textbf{1975-1979} & \textbf{1995-1999} \\ 
\includegraphics[width=0.48\linewidth]{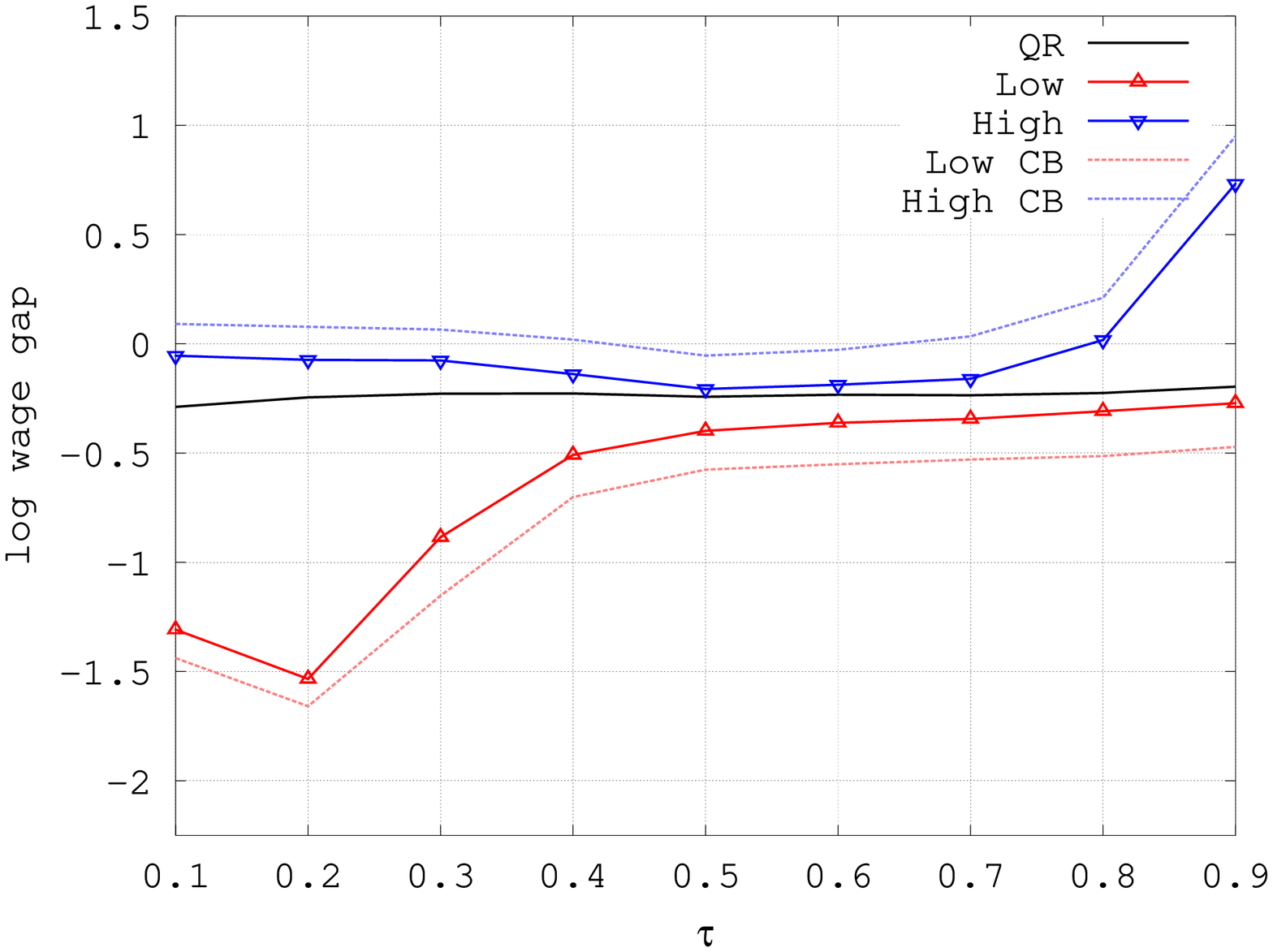} & 
\includegraphics[width=0.48\linewidth]{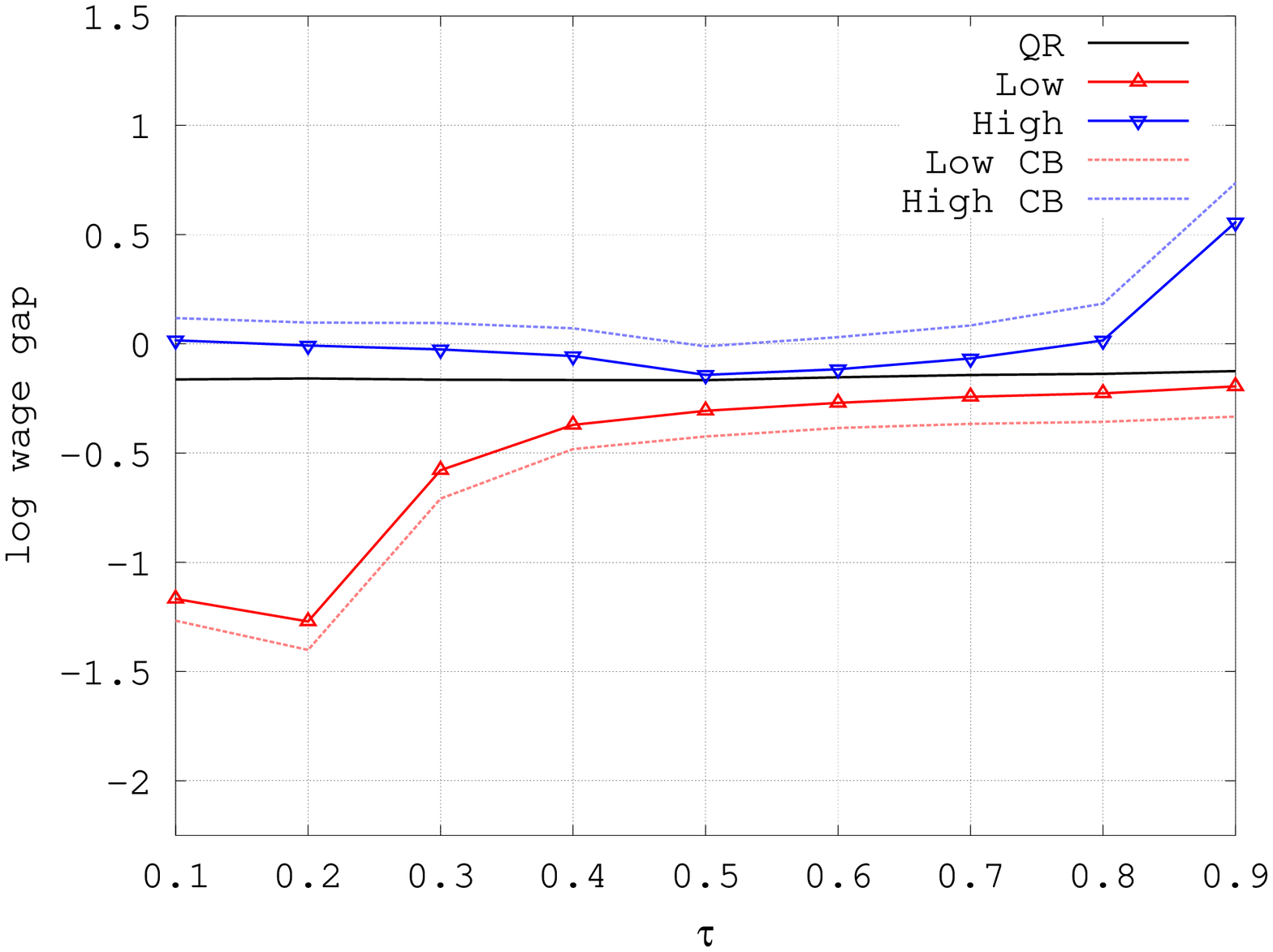} \\
\end{tabular}
\footnotetext{ This figure shows the estimated quantile gender wage
  (female $-$ male) conditional on average characteristics. The solid
  black line shows the quantile gender wage when selection is
  ignored. The blue and red lines with upward and downward pointing
  triangles show upper and lower bounds that account for employment
  selection for females. The dashed lines represent a uniform 90\%
  confidence region for the bounds.}  
\myFigureEnd

\section{Conclusion\label{sec:conclusion}}

This paper provides a novel method for inference on best linear
approximations to functions which are known to lie within a band. It
advances the literature by allowing for bounding functions that may be
estimated parametrically or non-parametrically by series estimators, and
that may carry an index. Our focus on best linear approximations is
motivated by the difficulty to work directly with the sharp identification
region of the functions of interest, especially when the analysis is
conditioned upon a large number of covariates. By contrast, best linear
approximations are tractable and easy to interpret. In particular, the sharp
identification region for the parameters characterizing the best linear
approximation is convex, and as such can be equivalently described via its
support function. The support function can in turn be estimated with a
plug-in method, that replaces moments of the data with their sample analogs,
and the bounding functions with their estimators. We show that the support
function process approximately converges to a Gaussian process. By
``approximately'' we mean that while the process may not converge weakly as
the number of series terms increases to infinity, each subsequence contains
a further subsequence that converges weakly to a tight Gaussian process with
a uniformly equicontinuous and non-degenerate covariance function. We
establish validity of the Bayesian bootstrap for practical inference, and
verify our regularity conditions for a large number of empirically relevant
problems, including mean regression with interval valued outcome data and
interval valued regressor data; quantile and distribution regression with
interval valued data; sample selection problems; and mean, quantile, and
distribution treatment effects.

\clearpage

\appendix

\section{Notation}

\begin{eqnarray*}
\mathcal{S}^{d-1} &:&=\left\{ q\in \mathbb{R}^{d}:\left\Vert q\right\Vert
=1\right\} ; \\
\Gn[h(t)] &:&=\tfrac{1}{\sqrt{n}}\tsum_{i=1}^{n}(h_{i}(t)-\Ep h(t)); \\
\G[h_k],\text{ }\widetilde{\G[h_k]} &:&=\text{ P-Brownian bridge processes,
independent of each other, and} \\
&&\text{ \ \ \ \ with identical distributions;} \\
\LL^{2}(X,\Pr ) &:&=\left\{ g:X\longrightarrow \mathbb{R}\text{ s.t. }%
\int_{X}\left\vert g(x)\right\vert ^{2}d\Pr (x)<\infty \right\} ; \\
\ell ^{\infty }(T) &:&\text{ set of all uniformly bounded real functions on }%
T; \\
BL_{1}(\ell ^{\infty }(T),[0,1]) &:&\text{ set of real functions on }\ell
^{\infty }(T)\text{ with Lipschitz norm bounded by 1;} \\
&\lesssim &\text{ left side bounded by a constant times the right side;} \\
f^{o} &:&=f-\Ep f.
\end{eqnarray*}

\section{Proof of the Results}

Throughout this Appendix, we impose Conditions \ref{c:smooth}-\ref%
{c:complexity function classes}.

\subsection{Proof of Theorems 1 and 2}

\textbf{Step 1.} We can write the difference between the estimated and true
support function as the sum of three differences. 
\begin{equation*}
\hat{\sigma}_{\hat{\theta},\hat{\Sigma}}-\sigma _{\theta ,\Sigma }=\left( 
\hat{\sigma}_{\hat{\theta},\hat{\Sigma}}-\hat{\sigma}_{\theta ,\hat{\Sigma}%
}\right) +\left( \hat{\sigma}_{\theta ,\hat{\Sigma}}-\hat{\sigma}_{\theta
,\Sigma }\right) +\left( \hat{\sigma}_{\theta ,\Sigma }-\sigma _{\theta
,\Sigma }\right)
\end{equation*}%
where $t\in T:=\mathcal{S}^{d-1}\times \indxSet$. Let $\mu :=q^{\prime
}\Sigma $ and 
\begin{equation*}
w_{i,\mu }(\indx):=\left( \theta _{0}(x,\indx){1}(\mu z_{i}<0)+\theta _{1}(x,%
\indx){1}(\mu z_{i}\geq 0)\right) .
\end{equation*}%
We define 
\begin{equation*}
\hat{\sigma}_{\theta ,\hat{\Sigma}}:=\En\left[ q^{\prime }\hat{\Sigma}%
z_{i}w_{i,q^{\prime }\hat{\Sigma}}(\indx)\right] \text{ and }\hat{\sigma}%
_{\theta ,\Sigma }:=\En\left[ q^{\prime }\Sigma z_{i}w_{i,q^{\prime }\Sigma
}(\indx)\right] .
\end{equation*}%
By Lemma \ref{VC lemma: linearization} uniformly in $t\in T$ 
\begin{eqnarray*}
\sqrt{n}\left( \hat{\sigma}_{\hat{\theta},\hat{\Sigma}}-\hat{\sigma}_{\theta
,\hat{\Sigma}}\right) (t) &=&q^{\prime }\Sigma \Er[z_i p_i' 1\{q '\Sigma z_i
>0\}]J_{1}^{-1}(\indx)\Gn [p_i \varphi_{i1}(\indx)] \\
&+&q^{\prime }\Sigma \Er[z_i p_i' 1\{q'\Sigma z_i <0\}]J_{0}^{-1}(\indx)\Gn
[p_i \varphi_{i0}(\indx)]+o_{\Pr }(1).
\end{eqnarray*}%
By Lemma \ref{AP lemma: remove hat sigmas} uniformly in $t\in T$ 
\begin{eqnarray*}
\sqrt{n}\text{$\left( \hat{\sigma}_{\theta ,\hat{\Sigma}}-\hat{\sigma}%
_{\theta ,\Sigma }\right) $}(t) &=&\sqrt{n}q^{\prime }\left( \hat{\Sigma}%
-\Sigma \right) \Ep\left[ z_{i}w_{i,q^{\prime }\Sigma }(\indx)\right]
+o_{\Pr }(1) \\
&=&-q^{\prime }\hat{\Sigma}\Gn [x_i z_i']\Sigma \Ep\left[ z_{i}w_{i,q^{%
\prime }\Sigma }(\indx)\right] +o_{\Pr }(1) \\
&=&-q^{\prime }\Sigma \Gn [x_i z_i']\Sigma \Ep\left[ z_{i}w_{i,q^{\prime
}\Sigma }(\indx)\right] +o_{\Pr }(1).
\end{eqnarray*}%
By definition 
\begin{equation*}
\sqrt{n}\left( \hat{\sigma}_{\theta ,\Sigma }-\sigma _{\theta ,\Sigma
}\right) (t)=\Gn[ q'\Sigma z_i w_{i, q'\Sigma}(u)].
\end{equation*}%
Putting all the terms together uniformly in $t\in T$ 
\begin{equation*}
\sqrt{n}(\hat{\sigma}_{\hat{\theta},\hat{\Sigma}}-\sigma _{\theta ,\Sigma
})(t)=\Gn[h_k(t)]+o_{\Pr }(1),
\end{equation*}%
where for $t:=(q,\indx)\in T=\mathcal{S}^{d-1}\times \indxSet$ 
\begin{eqnarray}
h_{k}(t):= &&q^{\prime }\Sigma \Er[z_i p_i' 1\{q '\Sigma z_i >0\}]J_{1}^{-1}(%
\indx)p_{i}\varphi _{i1}(\indx)  \notag \\
&+&q^{\prime }\Sigma \Er[z_i p_i' 1\{q'\Sigma z_i <0\}]J_{0}^{-1}(\indx%
)p_{i}\varphi _{i0}(\indx)  \notag \\
&-&q^{\prime }\Sigma x_{i}z_{i}^{\prime }\Sigma \Ep\left[ z_{i}w_{i,q^{%
\prime }\Sigma }(\indx)\right]  \notag \\
&+&q^{\prime }\Sigma z_{i}w_{i,q^{\prime }\Sigma }(\indx)  \notag \\
:= &&h_{1i}(t)+h_{2i}(t)+h_{3i}(t)+h_{4i}(t),  \label{eq: define h}
\end{eqnarray}%
where $k$ indexes the number of series terms.

\textbf{Step 2.} (Finite $k$). This case follows from $\mathcal{H}%
=\{h_{i}(t),t\in T\}$ being a Donsker class with square-integrable
envelopes. Indeed, $\mathcal{H}$ is formed as finite products and sums of VC
classes or entropically equivalent classes, so we can apply Lemma \ref%
{lemma: andrews}. The result 
\begin{equation*}
\Gn[h_k(t)]\Rightarrow \mathbb{G}[h_{k}(t)]\text{ in }\ell ^{\infty }(T),
\end{equation*}%
follows, and the assertion that 
\begin{equation*}
\Gn[h_k(t)]=_{d}\mathbb{G}[h_{k}(t)]+o_{\Pr }(1)\text{ in }\ell ^{\infty }(T)
\end{equation*}%
follows from e.g., the Skorohod-Dudley-Whichura construction. (The $=_{d}$
can be replaced by $=$ as in Step 3, in which case $\mathbb{G}[h_{k}(t)]$ is
a sequence of Gaussian processes indexed by $n$ and identically distributed
for each $n$.)

\textbf{Step 3.} (Case with growing $k$.) This case is considerably more
difficult. The main issue here is that the uniform covering entropy of $%
\mathcal{H}_{l}=\{h_{li}(t),t\in T\},$ $l=0,1,$ grows without bound, albeit
at a very slow rate $\log n$. The envelope $H_{l}$ of this class also grows
in general, and so we can not rely on the usual uniform entropy-based
arguments; for similar reasons we can not rely on the bracketing-based
entropy arguments. Instead, we rely on a strong approximation argument,
using ideas in \citeasnoun{ChernozhukovLeeRosen09} and %
\citeasnoun{BelloniChernozhukov2009}, to show that $\Gn[h_k(t)]$ can be
approximated by a tight sequence of Gaussian processes $\mathbb{G}[h(t)]$,
implicitly indexed by $k$, where the latter sequence is very well-behaved.
Even though it may not converge as $k\rightarrow \infty $, for every
subsequence of $k$ there is a further subsequence along which the Gaussian
process converges to a well-behaved Gaussian process. The latter is
sufficient for carrying out the usual inference.

Lemma \ref{VC: coupling lemma} below establishes that 
\begin{equation*}
\Gn[h(t)]=\mathbb{G}[h(t)]+o_{\Pr }(1)\text{ in }\ell ^{\infty }(T),
\end{equation*}%
where $\mathbb{G}[h]$ is a sequence of P-Brownian bridges with the
covariance function $\Ep[h(t) h(t')]-\Ep[h(t)]\Ep[h(t')]$. Lemma \ref{lemma:
covariance properties} below establishes that for some $0<c\leq 1/2$ 
\begin{equation*}
\rho _{2}(h(t),h(t^{\prime }))=\left( \Ep[ h(t) -h(t') ]^{2}\right)
^{1/2}\lesssim \rho (t,t^{\prime }):=\left\Vert t-t^{\prime }\right\Vert
^{c},
\end{equation*}%
and the function $\Ep[h(t) h(t')]-\Ep[h(t)]\Ep[h(t')]$ is equi-continuous on 
$T\times T$ uniformly in $k$. By assumption \ref{c:design conditions} we
have that $\inf_{t\in T}\text{var}[h(t)]>C>0$, with Lemma \ref{lemma:
covariance properties} providing a sufficient condition for this.

An immediate consequence of the above result is that we also obtain the
convergence in the bounded Lipschitz metric 
\begin{equation*}
\sup_{g \in BL_1(\ell^{\infty}(T), [0,1])} \left\vert \Ep \left[g (\Gn %
\lbrack h])\right] - \Ep\left[g(\G [h] ) \right]\right\vert \leq \Ep\sup_{t
\in T}\left\vert \Gn [h(t)]-\G[h(t)] \right\vert \wedge 1 \to 0.
\end{equation*}

\textbf{Step 4.} Let's recognize the fact that $h$ depends on $k$ by using
the notation $h_{k}$ in this step of the proof. Note that $k$ itself is
implicitly indexed by $n$. Let $F_{k}(c):=\Pr \{f(\G [h_k])\leq c\}$ and
observe that by Step 3 and $f\in \F_{c}$ 
\begin{eqnarray*}
&&|\Pr \{f(S_{n})\leq c_{n}+o_{p}(1)\}-\Pr \{f(\G [h_k])\leq c_{n}\}| \\
&\leq &|\Pr \{f(\G [h_k])\leq c_{n}+o_{p}(1)\}-\Pr \{f(\G [h_k])\leq c_{n}\}|
\\
&\leq &\delta _{n}(o_{p}(1))\rightarrow _{P}0,\ \ \text{ for }\delta
_{n}(\epsilon ):=\sup_{c\in \mathbb{R}}|F_{k}(c+\epsilon )-F_{k}(c)|,
\end{eqnarray*}%
where the last step follows by the Extended Continuous Mapping Theorem
(Theorem 18.11 in \citeasnoun{vaart:text}) provided that we can show that
for any $\epsilon _{n}\searrow 0$, $\delta _{n}(\epsilon _{n})\rightarrow 0.$
Suppose otherwise, then there is a subsequence along which $\delta
_{n}(\epsilon _{n})\rightarrow \delta \neq 0$. We can select a further
subsequence say $\{n_{j}\}$ along which the covariance function of $\Gn [h_k]%
,$ denoted $\Omega _{nk}(t,t^{\prime })$ converges to a covariance function $%
\Omega _{0}(t,t^{\prime })$ uniformly on $T\times T$. We can do so by the
Arzel\`{a}-Ascoli theorem in view of the uniform equicontinuity in $k$ of
the sequence of the covariance functions $\Omega _{nk}(t,t^{\prime })$ on $%
T\times T$. Moreover, $\inf_{t\in T}\Omega _{0}(t,t)>C>0$ by our assumption
on $\Omega _{nk}(t,t^{\prime })$. But along this subsequence $\G[h_k]$
converges in $\ell ^{\infty }(T)$ in probability to a tight Gaussian
process, say $Z_{0}$. The latter happens because $\G[h_k]$ converges to $%
Z_{0}$ marginally by Gaussianity and by $\Omega _{nk}(t,t^{\prime
})\rightarrow \Omega _{0}(t,t^{\prime })$ uniformly and hence pointwise on $%
T\times T$ and because $\G[h_k]$ is asymptotically equicontinuous as shown
in the proof of Lemma \ref{VC: coupling lemma}. Thus, along this subsequence
we have that 
\begin{equation*}
F_{k}(c)\rightarrow F_{0}(c)=\Pr \{f(Z_{0})\leq c\},\text{ uniformly in $%
c\in \mathbb{R},$ }
\end{equation*}%
because we have pointwise convergence that implies uniform convergence by
Polya's theorem, since $F_{0}$ is continuous by $f\in \F_{c}$ and by $%
\inf_{t\in T}\Omega _{0}(t,t)>C>0$. This implies that along this subsequence 
$\delta _{n_{j}}(\epsilon _{n_{j}})\rightarrow 0$, which gives a
contradiction.

\textbf{Step 5.} Finally, we observe that $c(1-\critv)=O(1)$ holds by $%
\sup_{t\in T}\Vert \G [h_k(t)]\Vert =O_{\Pr }(1)$ as shown in the proof of
Lemma \ref{VC: coupling lemma}, and the second part of Theorem \ref%
{thm:inference} follows. \qed

\subsection{Proof of Theorems 3 and 4}

\textbf{Step 1.} We can write the difference between a bootstrap and true
support function as the sum of three differences. 
\begin{equation*}
\widetilde{\sigma }_{\widetilde{\theta },\widetilde{\Sigma }}-\sigma
_{\theta ,\Sigma }=\left( \widetilde{\sigma }_{\widetilde{\theta },%
\widetilde{\Sigma }}-\widetilde{\sigma }_{\theta ,\widetilde{\Sigma }%
}\right) +\left( \widetilde{\sigma }_{\theta ,\widetilde{\Sigma }}-%
\widetilde{\sigma }_{\theta ,\Sigma }\right) +\left( \widetilde{\sigma }%
_{\theta ,\Sigma }-\sigma _{\theta ,\Sigma }\right)
\end{equation*}%
where for 
\begin{equation*}
w_{i,\mu }(\indx)=:\left( \theta _{0}(x,\indx){1}(\mu z_{i}<0)+\theta _{1}(x,%
\indx){1}(\mu z_{i}\geq 0)\right)
\end{equation*}%
we define 
\begin{equation*}
\widetilde{\sigma }_{\theta ,\widetilde{\Sigma }}:=\En\left[ (e_{i}/\bar{e}%
)q^{\prime }\widetilde{\Sigma }^{\prime }z_{i}w_{i,q^{\prime }\widetilde{%
\Sigma }}(\indx)\right] \text{ and }\widetilde{\sigma }_{\theta ,\Sigma }=:%
\En\left[ (e_{i}/\bar{e})q^{\prime }\Sigma ^{\prime }z_{i}w_{i,q^{\prime
}\Sigma }(\indx)\right] ,
\end{equation*}%
where $\bar{e}=\En e_{i}\rightarrow _{\Pr }1$.

By Lemma \ref{VC lemma: linearization} uniformly in $t\in T$ 
\begin{eqnarray*}
\sqrt{n}\left( \widetilde{\sigma }_{\widetilde{\theta },\widetilde{\Sigma }}-%
\widetilde{\sigma }_{\theta ,\widetilde{\Sigma }}\right) (t) &=&q^{\prime
}\Sigma \Er[z_i p_i' 1\{q '\Sigma z_i >0\}]J_{1}^{-1}(\indx)\Gn [e_i p_i
\varphi_{i1}(\indx)] \\
&+&q^{\prime }\Sigma \Er[z_i p_i' 1\{q'\Sigma z_i <0\}]J_{0}^{-1}(\indx)\Gn
[e_i p_i \varphi_{i0}(\indx)]+o_{\Pr }(1).
\end{eqnarray*}%
By Lemma \ref{AP lemma: remove hat sigmas} uniformly in $t\in T$ 
\begin{eqnarray*}
\sqrt{n}\text{$\left( \widetilde{\sigma }_{\theta ,\widetilde{\Sigma }}-%
\widetilde{\sigma }_{\theta ,\Sigma }\right) $}(t) &=&\sqrt{n}q^{\prime
}\left( \widetilde{\Sigma }-\Sigma \right) \Ep\left[ z_{i}w_{i,q^{\prime
}\Sigma }(\indx)\right] +o_{\Pr }(1) \\
&=&q^{\prime }\widetilde{\Sigma }\Gn [(e_i/\bar{e} )(x_i z_i')^o]\Sigma \Ep%
\left[ z_{i}w_{i,q^{\prime }\Sigma }(\indx)\right] +o_{\Pr }(1) \\
&=&q^{\prime }\Sigma \Gn[e_i(x_i z_i')^o]\Sigma \Ep\left[ z_{i}w_{i,q^{%
\prime }\Sigma }(\indx)\right] +o_{\Pr }(1).
\end{eqnarray*}%
By definition 
\begin{equation*}
\sqrt{n}\left( \widetilde{\sigma }_{\theta ,\Sigma }-\sigma _{\theta ,\Sigma
}\right) (t)=\Gn[ e_i \left(q'\Sigma z_i w_{i, q'\Sigma}(\indx)\right)^o]/%
\bar{e}=\Gn[ e_i \left(q'\Sigma z_i w_{i, q'\Sigma}(\indx)\right)^o]%
(1+o_{p}(1)).
\end{equation*}%
Putting all the terms together uniformly in $t\in T$ 
\begin{equation*}
\sqrt{n}(\widetilde{\sigma }_{\widetilde{\theta },\widetilde{\Sigma }%
}-\sigma _{\theta ,\Sigma })(t)=\Gn[e_i h^o_i(t)]+o_{\Pr }(1).
\end{equation*}

\textbf{Step 2.} Combining conclusions of Theorems 1 and Step 1 above we
obtain: 
\begin{eqnarray*}
&&\tilde{S}_{n}(t)=\sqrt{n}(\widetilde{\sigma }_{\widetilde{\theta },%
\widetilde{\Sigma }}-\hat{\sigma}_{\hat{\theta},\hat{\Sigma}})(t) \\
&=&\sqrt{n}(\widetilde{\sigma }_{\widetilde{\theta },\widetilde{\Sigma }%
}-\sigma _{\theta ,\Sigma })(t)-\sqrt{n}(\hat{\sigma}_{\hat{\theta},\hat{%
\Sigma}}-\sigma _{\theta ,\Sigma })(t) \\
&=&\Gn[e_i h^o_i(t)]-\Gn[h(t)]+o_{\Pr }(1) \\
&=&\Gn[ e_i^o h_i^o(t)]+o_{\Pr }(1).
\end{eqnarray*}%
Observe that the bootstrap process $\Gn[ e_i^oh_i^o(t)]$ has the
unconditional covariance function 
\begin{equation*}
\Ep[h(t) h(t')]-\Ep[h(t)]\Ep[h(t')],
\end{equation*}%
which is equal to the covariance function of the original process $\Gn [h_i]$%
. Conditional on data the covariance function of this process is 
\begin{equation*}
\En[h(t) h(t')]-\En[h(t)]\En[h(t')].
\end{equation*}

\begin{remark}
\label{remark: bootstrap} Note that if a bootstrap random element $Z_{n}$
taking values in a normed space $(E,\Vert \cdot \Vert )$ converges in
probability $\Pr $ unconditionally, that is $Z_{n}=o_{\Pr }(1)$, then $%
Z_{n}=o_{\Pr^{e}}(1)$ in $L^{1}(P)$ sense and hence probability $\Pr $,
where $\Pr^{e}$ denotes the probability measure conditional on the data. In
other words, $Z_{n}$ also converges in probability conditionally on the
data. This follows because $\Ep_{\Pr }|\Pr^{e}\{\Vert Z_{n}\Vert >\epsilon
\}|=\Pr \{\Vert Z_{n}\Vert >\epsilon \}\rightarrow 0$, so that $%
\Pr^{e}\{\Vert Z_{n}\Vert >\epsilon \}\rightarrow 0$ in $L^{1}(P)$ sense and
hence in probability $\Pr $. Similarly, if $Z_{n}=O_{\Pr }(1)$, then $%
Z_{n}=O_{\Pr^{e}}(1)$ in probability $\Pr $.
\end{remark}

\textbf{Step 3.} (Finite $k$). This case follows from $\mathcal{H}%
=\{h_{i}(t),t\in T\}$ being a Donsker class with square-integrable
envelopes. Indeed, $\mathcal{H}$ is formed as a Lipschitz composition of VC
classes or entropically equivalent classes. Then by the Donsker theorem for
exchangeable bootstraps, see e.g., \citeasnoun{vaartwellner}, we have weak
convergence conditional on the data 
\begin{equation*}
\Gn[e_i^o h^o_i(t)]/\bar{e}\Rightarrow \widetilde{\mathbb{G}[h(t)]}\text{
under $\Pr^{e}$ in }\ell ^{\infty }(T)\text{ in probability $\Pr $},
\end{equation*}%
where $\widetilde{\mathbb{G}[h]}$ is a sequence of P-Brownian bridges
independent of $\mathbb{G}[h]$ and with the same distribution as $\mathbb{G}%
[h].$ In particular, the covariance function of $\widetilde{\mathbb{G}[h]}$
is $\Ep[h(t) h(t')]-\Ep[h(t)]\Ep[h(t')]$. Since $\bar{e}\rightarrow
_{\Pr^{e}}1$, the above implies%
\begin{equation*}
\Gn[e_i^oh^o_i(t)]\Rightarrow \widetilde{\mathbb{G}[h(t)]}\text{ under $%
\Pr^{e}$ in }\ell ^{\infty }(T)\text{ in probability $\Pr $}.
\end{equation*}%
The latter statement simply means 
\begin{equation*}
\sup_{g\in BL_{1}(\ell ^{\infty }(T),[0,1])}|\Ep_{\Pr^{e}}[g(\Gn [h])]-\Ep%
\lbrack g(\widetilde{\mathbb{G}[h]})]|\rightarrow _{\Pr }0.
\end{equation*}%
This statement can be strengthened to a coupling statement as in Step 4.

\textbf{Step 4.} (Growing $k$.) By Lemma \ref{VC: coupling lemma} below we
can show that (on a suitably extended probability space) there exists a
sequence of Gaussian processes $\widetilde{\mathbb{G}[h(t)]}$ such that 
\begin{equation*}
\Gn[e_i^o h^o_i(t)]=\widetilde{\mathbb{G}[h(t)]}+o_{\Pr }(1)\text{ in }\ell
^{\infty }(T),
\end{equation*}%
which implies by Remark \ref{remark: bootstrap} that 
\begin{equation*}
\Gn[e_i^oh^o_i(t)]=\widetilde{\mathbb{G}[h(t)]}+o_{\Pr^{e}}(1)\text{ in }%
\ell ^{\infty }(T)\text{ in probability}.
\end{equation*}%
Here, as above, $\widetilde{\mathbb{G}[h]}$ is a sequence of P-Brownian
bridges independent of $\mathbb{G}[h]$ and with the same distribution as $%
\mathbb{G}[h].$ In particular, the covariance function of $\widetilde{%
\mathbb{G}[h]}$ is $\Ep[h(t) h(t')]-\Ep[h(t)]\Ep[h(t')]$. Lemma \ref{lemma:
covariance properties} describes the properties of this covariance function,
which in turn define the properties of this Gaussian process.

An immediate consequence of the above result is the convergence in bounded
Lipschitz metric 
\begin{equation*}
\sup_{g\in BL_{1}(\ell ^{\infty }(T),[0,1])}\left\vert \Ep_{\Pr^{e}}[g(%
\Gn[e_i^oh^o_i(t)])]-\Ep_{\Pr^{e}}[g(\widetilde{\G [h]})]\right\vert \leq \Ep%
_{\Pr^{e}}\sup_{t\in T}|\Gn[e_i^oh^o_i(t)]-\widetilde{\G[h(t)]}|\wedge
1\rightarrow _{\Pr }0.
\end{equation*}%
Note that $\Ep_{\Pr^{e}}[g(\widetilde{\G [h]})]=\Ep_{\Pr }[g(\widetilde{\G
\lbrack h]})]$, since the covariance function of $\widetilde{\G[h]}$ does
not depend on the data. Therefore 
\begin{equation*}
\sup_{g\in BL_{1}(\ell ^{\infty }(T),[0,1])}|\Ep_{\Pr^{e}}[g(%
\Gn[e_i^oh^o_i(t)])]-\Ep_{\Pr }[g(\widetilde{\G [h]})]|\rightarrow _{\Pr }0.
\end{equation*}

\textbf{Step 5.} Let us recognize the fact that $h$ depends on $k$ by using
the notation $h_{k}$ in this step of the proof. Note that $k$ itself is
implicitly indexed by $n$. By the previous steps and Theorem 1 there exist $%
\epsilon _{n}\searrow 0$ such that $\pi _{1}=\Pr^{e}\{|f(\widetilde{S}%
_{n})-f(\widetilde{\G[h_k]})|>\epsilon _{n}\}$ and $\pi _{2}=\Pr \{|f(%
\widetilde{S}_{n})-f(\G[h_k])|>\epsilon _{n}\}$ obey $\Ep[\pi_1]\rightarrow
_{P}0$ and $\pi _{2}\rightarrow 0$. Let 
\begin{equation*}
F(c):=\Pr \{f(\G [h_k])\leq c\}=\Pr \{f(\widetilde{\G[h_k]})\leq
c\}=\Pr^{e}\{f(\G [h_k])\leq c\},
\end{equation*}%
where the equality holds because $\G [h_k]$ and $\widetilde{\G[h_k]}$ are
P-Brownian bridges with the same covariance kernel, which in the case of the
bootstrap does not depend on the data.

For any $c_{n}$ which is a measurable function of the data, 
\begin{eqnarray*}
&&\Ep|\Pr^{e}\{f(\widetilde{S}_{n})\leq c_{n}\}-\Pr \{f(S_{n})\leq c_{n}\}|
\\
&\leq &\Ep\lbrack \Pr^{e}\{f(\widetilde{\G[h_k]})\leq c_{n}+\epsilon
_{n}\}-\Pr \{f(\G [h_k])\leq c_{n}-\epsilon _{n}\}+\pi _{1}+\pi _{2}] \\
&=&\Ep F(c_{n}+\epsilon _{n})-\Ep F(c_{n}-\epsilon _{n})+o(1) \\
&\leq &\sup_{c\in \mathbb{R}}|F(c+\epsilon _{n})-F(c-\epsilon
_{n})|+o(1)=o(1),
\end{eqnarray*}%
where the last step follows from the proof of Theorem 1. This proves the
first claim of Theorem 4 by the Chebyshev inequality. The second claim of
Theorem 4 follows similarly to Step 5 in the proof of Theorems 1-2. \qed

\subsection{Main Lemmas for the Proofs of Theorems 1 and 2}

\begin{lemma}[Linearization]
\label{VC lemma: linearization} 1. (Sample) We have that uniformly in $t\in
T $ 
\begin{eqnarray*}
\sqrt{n}\left( \hat{\sigma}_{\hat{\theta},\hat{\Sigma}}-\hat{\sigma}_{\theta
,\hat{\Sigma}}\right) &=&q^{\prime }\hat{\Sigma}\sqrt{n}\left( \En %
z_{i}\left( \hat{\theta}_{1,i}(\indx)-\theta _{1,i}(\indx)\right) {1}\left\{
q^{\prime }\hat{\Sigma}z_{i}>0\right\} \right) \\
&+&q^{\prime }\hat{\Sigma}\sqrt{n}\left( \En z_{i}\left( \hat{\theta}_{0,i}(%
\indx)-\theta _{0,i}(\indx)\right) {1}\left\{ q^{\prime }\hat{\Sigma}%
z_{i}<0\right\} \right) \\
&=&q^{\prime }\Sigma \Er[z_i p_i' 1\{q' \Sigma z_i >0\}]J_{1}^{-1}(\indx)\Gn
[p_i \varphi_{i1}(\indx)] \\
&+&q^{\prime }\Sigma \Er[z_i p_i' 1\{q' \Sigma z_i <0\}]J_{0}^{-1}(\indx)\Gn
[p_i \varphi_{i0}(\indx)]+o_{\Pr }(1).
\end{eqnarray*}%
2. (Bootstrap) We have that uniformly in $t\in T$ 
\begin{eqnarray*}
\sqrt{n}\left( \widetilde{\sigma }_{\widetilde{\theta },\widetilde{\Sigma }}-%
\widetilde{\sigma }_{\theta ,\widetilde{\Sigma }}\right) (t) &=&q^{\prime }%
\widetilde{\Sigma }\sqrt{n}\left( \En(e_{i}/\bar{e})z_{i}\left( \widetilde{%
\theta }_{1,i}(\indx)-\theta _{1,i}(\indx)\right) {1}\left\{ q^{\prime }%
\widetilde{\Sigma }z_{i}>0\right\} \right) \\
&+&q^{\prime }\widetilde{\Sigma }\sqrt{n}\left( \En(e_{i}/\bar{e}%
)z_{i}\left( \widetilde{\theta }_{0,i}(\indx)-\theta _{0,i}(\indx)\right) {1}%
\left\{ q^{\prime }\widetilde{\Sigma }z_{i}<0\right\} \right) \\
&=&q^{\prime }\Sigma \Er[z_i p_i' 1\{q' \Sigma z_i >0\}]J_{1}^{-1}(\indx)\Gn
[e_i p_i \varphi_{i1}(\indx)] \\
&+&q^{\prime }\Sigma \Er[z_i p_i' 1\{q' \Sigma z_i <0\}]J_{0}^{-1}(\indx)\Gn
[e_i p_i \varphi_{i0}(\indx)]+o_{\Pr }(1).
\end{eqnarray*}
\end{lemma}

\textbf{Proof of Lemma \ref{VC lemma: linearization}.} In order to cover
both cases with one proof, we will use $\bar{\theta}$ to mean either the
unweighted estimator $\hat{\theta}$ or the weighted estimator $\tilde{\theta}
$ and so on, and $v_{i}$ to mean either $1$ in the case of the unweighted
estimator or exponential weights $e_{i}$ in the case of the weighted
estimator. We also observe that $\bar{\Sigma}\rightarrow _{P}\Sigma $ by the
law of large numbers and the continuous mapping theorem.

\textbf{Step 1.} It will suffice to show that 
\begin{eqnarray*}
&&q^{\prime }\bar{\Sigma}\sqrt{n}\left( \En(v_{i}/\bar{v})z_{i}\left( \bar{%
\theta}_{1,i}(\indx)-\theta _{1,i}(\indx)\right) {1}\left\{ q^{\prime }\bar{%
\Sigma}z_{i}>0\right\} \right) \\
&=&q^{\prime }\Sigma \Er[z_i p_i' 1\{q \Sigma z_i >0\}]J_{1}^{-1}(\indx)\Gn
[v_i p_i \varphi_{i1}(\indx)]+o_{\Pr }(1)
\end{eqnarray*}%
and that 
\begin{eqnarray*}
&&q^{\prime }\bar{\Sigma}\sqrt{n}\left( \En(v_{i}/\bar{v})z_{i}\left( \bar{%
\theta}_{0,i}(\indx)-\theta _{0,i}(\indx)\right) {1}\left\{ q^{\prime }\bar{%
\Sigma}z_{i}<0\right\} \right) \\
&=&q^{\prime }\Sigma \Er[z_i p_i' 1\{q \Sigma z_i <0\}]J_{0}^{-1}(\indx)\Gn
[v_i p_i \varphi_{i0}(\indx)]+o_{\Pr }(1).
\end{eqnarray*}%
We show the argument for the first part; the argument for the second part is
identical. We also drop the index $\ell =1$ to ease notation. By Assumption %
\ref{c:linearization} we write 
\begin{eqnarray*}
q^{\prime }\bar{\Sigma}\sqrt{n}\En(v_{i}/\bar{v})z_{i}\left( \bar{\theta}%
_{i}-\theta _{i}\right) {1}\left\{ q^{\prime }\bar{\Sigma}z_{i}>0\right\}
&=&\left\{ q^{\prime }\bar{\Sigma}\mathbb{E}_{n}[(v_{i}z_{i}p_{i}^{\prime
}1\{q^{\prime }\bar{\Sigma}z_{i}>0\}]J^{-1}(\indx)\mathbb{G}_{n}\left[
v_{i}p\varphi _{i}(\indx)\right] \right. \\
&&\left. +q^{\prime }\bar{\Sigma}\mathbb{E}_{n}[v_{i}z_{i}\bar{R}_{i}(\indx%
)1\{q^{\prime }\bar{\Sigma}z_{i}>0\}]\right\} /\bar{v} \\
&=&:(a(\indx)+b(\indx))/(1+o_{\Pr }(1)).
\end{eqnarray*}%
We have from the assumptions of the theorem 
\begin{equation*}
\sup_{\indx\in \indxSet}|b(\indx)|\leq \Vert q^{\prime }\bar{\Sigma}\Vert
\cdot \Vert \Vert z_{i}\Vert \Vert _{\Pn,2}\Vert \Vert v_{i}\Vert \Vert _{\Pn%
,2}\cdot \sup_{\indx\in \indxSet}\Vert \bar{R}_{i}(\indx)\Vert _{\Pn%
,2}=O_{\Pr }(1)O_{\Pr }(1)o_{\Pr }(1)=o_{\Pr }(1).
\end{equation*}%
Write $a(\indx)=c(\indx)+d(\indx)$, where 
\begin{eqnarray}
c(\indx):= &&q^{\prime }\Sigma \Er[z_i p_i' 1\{q \Sigma z_i >0\}]J^{-1}(\indx%
)\Gn\left[ v_{i}p_{i}\varphi _{i}(\indx)\right]  \notag \\
d(\indx):= &&\bar{\mu}^{\prime }J^{-1}(\indx)\Gn\left[ v_{i}p_{i}\varphi
_{i}(\indx)\right]  \notag \\
\bar{\mu}^{\prime }:= &&q^{\prime }\bar{\Sigma}\En[v_i z_i p_i' 1\{
q'\bar{\Sigma} z_i>0\}]-q^{\prime }\Sigma \Er[z_i p_i' 1 \{ q'\Sigma z_i>0\}]
\label{eq:mu_bar}
\end{eqnarray}%
The claim follows after showing that $\sup_{\indx\in \indxSet}|d(\indx%
)|=o_{\Pr }(1)$, which is shown in subsequent steps below.

\textbf{Step 2.} (Special case, with $k$ fixed). This is the parametric
case, which is trivial. In this step we have to show $\sup_{\indx\in \indxSet%
}|d(\indx)|=o_{\Pr }(1)$. We can write 
\begin{equation*}
d(\indx)=\Gn [\bar{\mu}'f_{\indx}],\ \ \ f_{\indx}:=\left( f_{\indx %
j},j=1,...,k\right) ,\ f_{\indx j}:=J^{-1}(\indx)v_{i}p_{ij}\varphi _{i}(%
\indx)
\end{equation*}%
and define the function class $\F:=\{f_{\indx j},\indx\in \indxSet%
,j=1,...,k\}$. Since $k$ is finite, and given the assumptions on $\F%
_{1}=\{\varphi (\indx),\indx\in \indxSet\}$, application of Lemmas \ref%
{lemma: andrews} and \ref{lemma: complexities}-2(a) yields 
\begin{equation*}
\sup_{Q}\log N(\epsilon \Vert F\Vert _{Q,2},\mathcal{F},L_{2}(Q))\lesssim
\log (1/\epsilon ).
\end{equation*}%
and the envelope is $\Pr$-square integrable. Therefore, $\F$ is P-Donsker and 
\begin{equation*}
\sup_{\indx\in \indxSet}|\Gn [f_{\indx}]|\lesssim _{\Pr }1
\end{equation*}%
and $\sup_{\indx\in \indxSet}|d(\indx)|\lesssim _{\Pr }k\Vert \bar{\mu}\Vert
\rightarrow _{\Pr }0$.

\textbf{Step 3.} (General case, with $k\rightarrow \infty $). In this step
we have to show $\sup_{\indx\in \indxSet}|d(\indx)|=o_{\Pr }(1)$. The case
of $k\rightarrow \infty $ is much more difficult if we want to impose rather
weak conditions on the number of series terms. We can write 
\begin{equation*}
d(\indx)=\Gn[f_{\indx n}],\ \ \ f_{\indx n}:=\bar{\mu}^{\prime }J^{-1}(\indx%
)v_{i}p_{i}\varphi _{i}(\indx)
\end{equation*}%
and define the function class $\F_{3}:=\{f_{\indx n},\indx\in \indxSet\},$
see equation (\ref{eq:F_3}) below. By Lemma \ref{lemma: complexities} the
random entropy of this function class obeys 
\begin{equation*}
\log N(\epsilon \Vert F_{3}\Vert _{\Pn,2},\mathcal{F}_{3},L_{2}(\Pn%
))\lesssim _{\Pr }\log n+\log (1/\epsilon ).
\end{equation*}%
Therefore by Lemma \ref{lemma: maximal}, conditional on $%
X_{n}=(x_{i},z_{i},i=1,...,n)$, for each $\delta >0$ there exists a constant 
$K_{\delta }$, that does not depend on $n$, such that for all $n$: 
\begin{equation*}
\Pr \left\lbrace \sup_{\indx\in \indxSet}|d(\indx)|\geq K_{\delta }\sqrt{\log n}
\left( 
\sup_{\indx\in \indxSet}\Vert f_{\indx n}\Vert _{\Pn,2}\vee \sup_{\indx\in %
\indxSet}\Vert f_{\indx n}\Vert _{\Pr |X_{n},2} \right) \right\rbrace \leq \delta ,
\end{equation*}%
where $\Pr |X_{n}$ denotes the probability measure conditional on $X_{n}$.
The conclusion follows if we can demonstrate that $\sup_{\indx\in \indxSet%
}\Vert f_{\indx n}\Vert _{\Pn,2}\vee \sup_{\indx\in \indxSet}\Vert f_{\indx %
n}\Vert _{\Pr |X_{n},2}\rightarrow _{\Pr }0$. To show this note that%
\begin{equation*}
\sup_{\indx\in \indxSet}\Vert f_{\indx n}\Vert _{\Pn,2}\leq \Vert \bar{\mu}%
\Vert \sup_{\indx\in \indxSet}\Vert J^{-1}(\indx)\Vert \Vert \En %
p_{i}p_{i}^{\prime }\Vert \cdot \sup_{i\leq n}v_{i}\sup_{i\leq n,\indx\in %
\indxSet}|\varphi _{i}(\indx)|\rightarrow _{\Pr }0,
\end{equation*}%
where the convergence to zero in probability follows because%
\begin{equation*}
\Vert \bar{\mu}\Vert \lesssim _{\Pr }n^{-m/4}+\sqrt{(k/n)\cdot \log n}\cdot
(\log n\max_{i}\Vert z_{i}\Vert )\wedge \xi _{k},\ \ \ \sup_{i\leq
n}v_{i}\lesssim _{\Pr }\log n
\end{equation*}%
by Step 4 below, $\sup_{\indx\in \indxSet}\Vert J^{-1}(\indx)\Vert \lesssim
1 $ by assumption \ref{c:design conditions}, $\Vert \En p_{i}p_{i}^{\prime
}\Vert \lesssim _{\Pr }1$ by Lemma \ref{lemma: rudelson}, and 
\begin{equation*}
\log ^{2}n\left( n^{-m/4}+\sqrt{(k/n)\cdot \log n}\cdot \max_{i}\Vert
z_{i}\Vert \wedge \xi _{k}\right) \sup_{i\leq n,\indx\in \indxSet}|\varphi
_{i}(\indx)|\rightarrow _{\Pr }0
\end{equation*}%
by assumption \ref{c:growth}. Also note that 
\begin{equation*}
\sup_{\indx\in \indxSet}\Vert f_{\indx n}\Vert _{\Pr |X_{n},2}\leq \Vert 
\bar{\mu}\Vert \sup_{\indx\in \indxSet}\Vert J^{-1}(\indx)\Vert \Vert \En %
p_{i}p_{i}^{\prime }\Vert \cdot (\Er [v_i^2])^{1/2}\cdot \sup_{i\leq n,\indx%
\in \indxSet}\left[ \Ep[\varphi^2_i(u)|x_i,z_i]\right] ^{1/2}\rightarrow
_{\Pr }0,
\end{equation*}%
by the preceding argument and $\Er[\varphi^2_i(u)|x_i,z_i]$ uniformly
bounded in $\indx$ and $i$ by assumption \ref{c:design conditions}.

\textbf{Step 4.} In this step we show that 
\begin{equation*}
\Vert \bar{\mu}\Vert \lesssim _{\Pr }n^{-m/4}+\sqrt{(k/n)\cdot \log n}\cdot
(\log n\max_{i}\Vert z_{i}\Vert )\wedge \xi _{k}.
\end{equation*}%
We can bound 
\begin{equation*}
\Vert \bar{\mu}\Vert \leq \Vert \Sigma -\bar{\Sigma}\Vert \Vert \Er\left[
z_{i}p_{i}1\{q^{\prime }\Sigma z_{i}>0\}\right] \Vert +\Vert \bar{\Sigma}%
\Vert \mu _{1}+\Vert \bar{\Sigma}\Vert \mu _{2},
\end{equation*}%
where 
\begin{eqnarray*}
\mu _{1} &=&\Vert \En\left[ v_{i}z_{i}p_{i}^{\prime }1\{q^{\prime }\Sigma
z_{i}>0\}\right] -\Er\left[ z_{i}p_{i}^{\prime }1\{q^{\prime }\Sigma
z_{i}>0\}\right] \Vert \\
\mu _{2} &=&\Vert \En [v_i z_i p_i '\{1 \{ q'\Sigma z_i>0 \} - 1 \{
q'\bar{\Sigma} z_i>0 \}\} ]\Vert .
\end{eqnarray*}%
By Lemma \ref{lemma: rudelson}, $\Vert \Sigma -\bar{\Sigma}\Vert =o_{\Pr
}(1) $, and from Assumption \ref{c:design conditions} $\Vert \Er\left[
z_{i}p_{i}1\{q^{\prime }\Sigma z_{i}>0\}\right] \Vert \lesssim 1$.

By elementary inequalities 
\begin{equation*}
\bar{\mu}_{2}^{2}\leq \En\Vert v_{i}\Vert ^{2}\En\Vert z_{i}\Vert ^{2}\Vert %
\En [p_i p_i']\Vert \En[ \{1 \{ q'\Sigma z_i>0 \} - 1 \{ q'\bar{\Sigma}
z_i>0 \} \}^2 ]\lesssim _{\Pr }n^{-m/2},
\end{equation*}%
where we used the Chebyshev inequality along with $\Ep\Vert v_{i}\Vert
^{2}=1 $ and $\Ep[\Vert z_{i}\Vert ^{2}]<\infty $, $\Vert \En [p_i p_i']\Vert
\lesssim _{\Pr }1$ by Lemma \ref{lemma: rudelson}, and $%
\En[ \{1 \{ q'\Sigma z_i>0 \} - 1 \{ q'\bar{\Sigma}
z_i>0 \} \}^2 ]\lesssim _{\Pr }n^{-m/2}$ by Step 5 below.

We can write $\mu _{1}=\sup_{g\in \GG}|\En g-\Er g|$, where $\GG%
:=\{v_{i}\gamma ^{\prime }z_{i}p_{i}^{\prime }\eta 1\{q^{\prime }\Sigma
z_{i}>0\},\Vert \gamma \Vert =1,\Vert \eta \Vert =1\}.$ The function class $%
\GG$ obeys 
\begin{equation*}
\sup_{Q}\log N(\epsilon \Vert G\Vert _{Q,2},\GG,L_{2}(Q))\lesssim \left(
\dim (z_{i})+\dim (p_{i})\right) \log (1/\epsilon )\lesssim k\log
(1/\epsilon )
\end{equation*}%
for the envelope $G_{i}=v_{i}\Vert z\Vert _{i}\cdot \xi _{k}$ that obeys $%
\max_{i}\log G_{i}\lesssim _{\Pr }\log n$ by $E|v_{i}|^{p}<\infty $ for any $%
p>0$, $\Er\Vert z_{i}\Vert ^{2}<\infty $ and $\log \xi _{k}\lesssim _{\Pr
}\log n.$ Invoking Lemma \ref{lemma: maximal} we obtain 
\begin{equation*}
\mu _{1}\lesssim _{\Pr }\sqrt{(k/n)\cdot \log n}\times \sup_{g\in \GG}\Vert
g\Vert _{\Pn,2}\vee \sup_{g\in \GG}\Vert g\Vert _{\Pr ,2},
\end{equation*}%
where 
\begin{gather*}
\sup_{g\in \GG}\Vert g\Vert _{\Pn,2}\lesssim _{\Pr }\left( \max_{i}\Vert
z_{i}\Vert \max_{i}v_{i}\cdot \Vert \En [p_i p_i']\Vert \right) \wedge
\left( \left[ \En\Vert v_{i}z_{i}\Vert ^{2}\right] ^{1/2}\xi _{k}\right) \\
\lesssim _{\Pr }(\max_{i}v_{i}\max_{i}\Vert z_{i}\Vert )\wedge \xi
_{k}\lesssim _{\Pr }(\log n\max_{i}\Vert z_{i}\Vert )\wedge \xi _{k}
\end{gather*}%
by $\Ep \Vert z_{i}\Vert ^{2}<\infty $ and by $\En [p_i p_i']\lesssim _{\Pr }1$%
, $\max_{i}v_{i}\lesssim _{\Pr }\log n$ and $\sup_{g\in \GG}\Vert g\Vert
_{\Pr ,2}=\Vert Ez_{i}p_{i}^{\prime }\Vert \lesssim 1$ by Assumption \ref%
{c:design conditions}. Thus 
\begin{equation*}
\mu _{1}\lesssim _{\Pr }\sqrt{(k/n)\cdot \log n}(\max_{i}\Vert z_{i}\Vert
\log n)\wedge \xi _{k},
\end{equation*}%
and the claim of the step follows.

\textbf{Step 5.} Here we show 
\begin{equation*}
\sup_{q\in \mathcal{S}^{d-1}}\En\left[ \left( 1(q^{\prime }\Sigma z_{i}<0)-{1%
}(q^{\prime }\bar{\Sigma}z_{i}<0)\right) ^{2}\right] \lesssim _{\Pr
}n^{-m/2}.
\end{equation*}%
Note $\left( 1(q^{\prime }\Sigma z_{i}<0)-{1}(q^{\prime }\bar{\Sigma}%
z_{i}<0)\right) ^{2}\ =1(q^{\prime }\Sigma z_{i}<0<q^{\prime }\bar{\Sigma}%
z_{i})+1(q^{\prime }\Sigma z_{i}>0>q^{\prime }\bar{\Sigma}z_{i})$. The set 
\begin{equation*}
\mathcal{F}=\left\{ 1(q^{\prime }\Sigma z_{i}<0<q^{\prime }\tilde{\Sigma}%
z_{i})+1(q^{\prime }\Sigma z_{i}>0>q^{\prime }\tilde{\Sigma}z_{i}),q\in 
\mathcal{S}^{d-1},\Vert \Sigma \Vert \leq M,\Vert \tilde{\Sigma}\Vert \leq
M\right\}
\end{equation*}%
is P-Donsker because it is a VC class with a constant envelope. Therefore, $|%
\En f-\Er f|\lesssim _{\Pr }n^{-1/2}$ uniformly on $f\in \F$. Hence
uniformly in $q\in \mathcal{S}^{d-1}$, $%
\En[(1(q^{\prime}\Sigma
z_{i}<0)-{1}(q^{\prime}\bar{\Sigma}^{\prime}z_{i}<0))^{2}]$ is equal to 
\begin{eqnarray*}
&&\Ep\left[ 1(q^{\prime }\Sigma z_{i}<0<q^{\prime }\bar{\Sigma}^{\prime
}z_{i})+1(q^{\prime }\Sigma z_{i}>0>q^{\prime }\bar{\Sigma}^{\prime }z_{i})%
\right] +O_{\Pr }\left( n^{-1/2}\right) \\
&=&\Pr \left( \left\vert q^{\prime }\Sigma z_{i}\right\vert <\left\vert
q^{\prime }\left( \Sigma -\bar{\Sigma}\right) z_{i}\right\vert \right)
+O_{\Pr }\left( n^{-1/2}\right) \\
&\leq &\Vert \Sigma -\bar{\Sigma}\Vert ^{m}+O_{\Pr }\left( n^{-1/2}\right)
\lesssim _{\Pr }n^{-m/2}+n^{-1/2}\lesssim _{\Pr }n^{-m/2}
\end{eqnarray*}%
where we are using that for $0<m\leq 1$ 
\begin{equation*}
\Pr \left( \left\vert q^{\prime }\Sigma z_{i}\right\vert <\left\vert
q^{\prime }\left( \Sigma -\bar{\Sigma}\right) z_{i}\right\vert \right) \leq
\Pr \left( |q^{\prime }\Sigma z_{i}/\Vert z_{i}\Vert |<\Vert q\Vert \Vert
\Sigma -\bar{\Sigma}\Vert \right) \lesssim \Vert \bar{\Sigma}-\Sigma \Vert
^{m},
\end{equation*}%
where the last inequality holds by Assumption \ref{c:smooth}, which gives
that $\Pr \left( |q^{\prime }\Sigma z_{i}/\Vert z_{i}\Vert |<\delta \right)
/\delta ^{m}\lesssim 1$. $\qed$

\begin{lemma}
\label{AP lemma: remove hat sigmas} Let $w_{i,\mu }(\indx)=:\left( \theta
_{0}(x,\indx){1}(\mu z_{i}<0)+\theta _{1}(x,\indx){1}(\mu z_{i}\geq
0)\right) $. 1. (Sample) Then uniformly in $t\in T$ 
\begin{equation*}
\sqrt{n}\text{$\left( \hat{\sigma}_{\theta ,\hat{\Sigma}}-\hat{\sigma}%
_{\theta ,\Sigma }\right) $}(t)=\sqrt{n}q^{\prime }\left( \hat{\Sigma}%
-\Sigma \right) \Ep\left[ z_{i}w_{i,q^{\prime }\Sigma }(\indx)\right]
+o_{\Pr }(1)
\end{equation*}%
2. (Bootstrap) Then uniformly in $t\in T$ 
\begin{equation*}
\sqrt{n}\text{$\left( \tilde{\sigma}_{\theta ,\tilde{\Sigma}}-\tilde{\sigma}%
_{\theta ,\Sigma }\right) $}(t)=\sqrt{n}q^{\prime }\left( \tilde{\Sigma}%
-\Sigma \right) \Ep\left[ z_{i}w_{i,q^{\prime }\Sigma }(\indx)\right]
+o_{\Pr }(1)
\end{equation*}
\end{lemma}

\textbf{Proof of Lemma \ref{AP lemma: remove hat sigmas}.} In order to cover
both cases with one proof, we will use $\bar{\theta}$ to mean either the
unweighted estimator $\hat \theta$ or the weighted estimator $\tilde \theta$
and so on, and $v_i$ to mean either $1$ in the case of the unweighted
estimator or exponential weights $e_i$ in the case of the weighted
estimator. We also observe that $\bar{\Sigma} \to_{\Pr} \Sigma$ by the law of
large numbers (Lemma \ref{lemma: rudelson}) and the continuous mapping
theorem.

\textbf{Step 1.} Define $\mathcal{F}=\left\{ q^{\prime }\Sigma
z_{i}w_{i,q^{\prime }\Sigma }(t):t\in T,\Vert \Sigma \Vert \leq C\right\} .$
We have that for $\bar{f}_{i}(t)=q^{\prime }\bar{\Sigma}z_{i}w_{i,q^{\prime }%
\bar{\Sigma}}(t)$ and $f_{i}(t)=q^{\prime }\Sigma z_{i}w_{i,q^{\prime
}\Sigma }(t)$ by definition 
\begin{eqnarray*}
\sqrt{n}\text{$\left( \bar{\sigma}_{\theta ,\bar{\Sigma}}-\bar{\sigma}%
_{\theta ,\Sigma }\right) $}(t) &=&\sqrt{n}\En[(v_i/\bar{v})(\bar{f_i}(t)-
f_{i}(t))] \\
&=&\sqrt{n}\Ep[\bar{f_i}(t)- f_{i}(t)]+\Gn[v_i(\bar{f_i}(t)-f_{i}(t))^o]/%
\bar{v} \\
&=&\sqrt{n}\Ep[\bar{f_i}(t)- f_{i}(t)]+\Gn[v_i(\bar{f_i}(t)-f_{i}(t))^o]%
/(1+o_{\Pr }(1)).
\end{eqnarray*}%
By intermediate value expansion and Lemma \ref{lemma: derivative}, uniformly
in $\indx\in \indxSet$ and $q\in \mathcal{S}^{d-1}$ 
\begin{equation*}
\sqrt{n}\left( \Ep[\bar{f_i}(t)- f_{i}(t)]\right) =\sqrt{n}(q^{\prime }\bar{%
\Sigma}-q^{\prime }\Sigma )\Ep [z_i w_{i, q' \Sigma^*(t)}(t)]=\sqrt{n}%
(q^{\prime }\bar{\Sigma}-q^{\prime }\Sigma )\Ep [z_i w_{i,q' \Sigma}(t)]%
+o_{\Pr }(1),
\end{equation*}%
for $q\Sigma ^{\ast }(t)$ on the line connecting $q^{\prime }\Sigma $ and $%
q^{\prime }\bar{\Sigma}$, where the last step follows by the uniform
continuity of the mapping $(\indx,q^{\prime }\Sigma )\mapsto 
\Ep  [z_i w_{i,q'
\Sigma}(t)]$ and $q^{\prime }\bar{\Sigma}-q^{\prime }\Sigma \rightarrow
_{\Pr }0$. Furthermore $\sup_{t\in T}|\Gn[v_i(\bar{f_i}(t)-f_{i}(t))^o]%
|\rightarrow _{\Pr }0$ by Step 2 below, proving the claim of the Lemma.

\textbf{Step 2.} It suffices to show that for any $t\in T$, we have that $%
\mathbb{G}_{n}[v_{i}\left[ \bar{f}_{i}(t)-f_{i}(t)\right] ^{o}]\rightarrow
_{\Pr }0$. By Lemma 19.24 from \citeasnoun{vaart:text} it follows that
if $v_{i}\left[ \bar{f}_{i}(t)-f_{i}(t)\right] ^{o}\in \mathcal{G}%
=v_{i}(\left( \mathcal{F}-\mathcal{F}\right) ^{o})$ is such that 
\begin{equation*}
\left( \Er\left[ (v_{i}(\bar{f}_{i}(t)-f_{i}(t))^{o})^{2}\right] \right)
^{1/2}\leq 2\left( \Er\left[ (v_{i}(\bar{f}_{i}(t)-f_{i}(t)))^{2}\right]
\right) ^{1/2}\rightarrow _{\Pr }0,
\end{equation*}%
and $\mathcal{G}$ is P-Donsker, then $\mathbb{G}_{n}[v_{i}(\bar{f}%
_{i}(t)-f_{i}(t))^{o}]\rightarrow _{\Pr }0$. Here $\mathcal{G}$ is P-Donsker
because $\mathcal{F}$ is a P- Donsker class formed by taking products of $\F%
_{2}\supseteq \left\{ \theta _{i\ell }(\indx):\indx\in \indxSet,\ell
=0,1\right\} ,$ which possess a square-integrable envelope, with bounded VC
classes $\{1(q^{\prime }\Sigma z_{i}>0),q\in \mathcal{S}^{d-1},\Vert \Sigma
\Vert \leq C\}$ and $\{1(q^{\prime }\Sigma z_{i}\leq 0),q\in \mathcal{S}%
^{d-1},\Vert \Sigma \Vert \leq C\}$ and then summing followed by demeaning.
The difference $\left( \mathcal{F}-\mathcal{F}\right) ^{o}$ is also
P-Donsker, and its product with the independent square-integrable variable $%
v_{i}$ is still a P-Donsker class with a square-integrable envelope. The
functions class has a square-integrable envelope. Note that 
\begin{eqnarray*}
&&\Er[\bar{f}_i(t)-f_i(t)]^{2}=\Er\left( 
\begin{array}{cc}
& (q^{\prime }\bar{\Sigma}-q^{\prime }\Sigma )z_{i}\theta _{0i}(\indx){1}%
(q^{\prime }\bar{\Sigma}^{\prime }z_{i}<0)1\left( q^{\prime }\Sigma
z_{i}<0\right) \\ 
& +(q^{\prime }\bar{\Sigma}-q^{\prime }\Sigma )z_{i}\theta _{1i}(\indx){1}%
\left( q^{\prime }\bar{\Sigma}z_{i}>0\right) {1}\left( q^{\prime }\Sigma
z_{i}>0\right) \\ 
& +\left( q^{\prime }\bar{\Sigma}z_{i}\theta _{0i}(\indx)-q^{\prime }\Sigma
z_{i}\theta _{i1}(\indx)\right) {1}\left( q^{\prime }\bar{\Sigma}%
z_{i}<0<q^{\prime }\Sigma z_{i}\right) \\ 
& +\left( q^{\prime }\bar{\Sigma}z_{i}\theta _{1i}(\indx)-q^{\prime }\Sigma
z_{i}\theta _{0i}(\indx)\right) {1}\left( q^{\prime }\bar{\Sigma}%
z_{i}>0>q^{\prime }\Sigma z_{i}\right)%
\end{array}%
\right) ^{2} \\
&\lesssim &\left\Vert \Vert \bar{\Sigma}-\Sigma \Vert ^{2}\right\Vert _{\Pr
,2}\cdot \left\Vert \Vert z_{i}\Vert ^{2}\right\Vert _{\Pr ,2}\max_{\indx\in %
\indxSet,\ell \in \{0,1\}}\Vert \theta _{\ell i}^{2}(\indx)\Vert _{\Pr ,2} \\
&&+\left( \Vert \bar{\Sigma}\Vert _{\Pr }^{2}\vee \Vert \Sigma \Vert
^{2}\right) \cdot \Vert \Vert z_{i}\Vert ^{2}\Vert _{\Pr ,2}\max_{\indx\in %
\indxSet,\ell \in \{0,1\}}\Vert \theta _{\ell i}^{2}(\indx)\Vert _{\Pr
,2}\cdot \sup_{q\in \mathcal{S}^{d-1}}\Pr \left[ \left\vert q^{\prime
}\Sigma z_{i}\right\vert <\left\vert q^{\prime }\left( \bar{\Sigma}-\Sigma
\right) z\right\vert \right] ^{1/2} \\
&\lesssim &_{\Pr }\Vert \bar{\Sigma}-\Sigma \Vert ^{2}+\sup_{q\in \mathcal{S}%
^{d-1}}\Pr \left[ \left\vert q^{\prime }\Sigma z_{i}/\Vert z_{i}\Vert
\right\vert <\Vert \bar{\Sigma}-\Sigma \Vert \right] ^{1/2}\rightarrow 0,
\end{eqnarray*}%
where we invoked the moment and smoothness assumptions. \qed

\begin{lemma}[A Uniform Derivative]
\label{lemma: derivative} Let $\sigma _{i,\mu }(\indx)=\mu z_{i}\left(
\theta _{0i}(\indx){1}(\mu z_{i}<0)+\theta _{1i}(\indx){1}(\mu
z_{i}>0)\right) $. Uniformly in $\mu \in M=\{q^{\prime }\Sigma :q\in 
\mathcal{S}^{d-1},\Vert \Sigma \Vert \leq C\}$ and $\indx\in \indxSet$ 
\begin{equation*}
\frac{\partial \Ep[\sigma_{i, \mu}(\indx)]}{\partial \mu }=\Ep [z_i w_{i,
\mu}(\indx)],
\end{equation*}%
where the right hand side is uniformly continuous in $\mu $ and $\indx$.
\end{lemma}

\textbf{Proof:} The continuity of the mapping $(\mu ,\indx )\mapsto 
\Ep  [z_i
w_{i, \mu}(u)]$ follows by an application of the dominated convergence
theorem and stated assumptions on the envelopes.

Note that for any $\Vert \delta \Vert \rightarrow 0$ 
\begin{equation*}
\frac{\Ep[(\mu+ \delta) z_i w_{i, \mu+ \delta}(\indx)]-\Ep[\mu z_i w_{i,
\mu}(\indx)]}{\left\Vert \delta \right\Vert }=\displaystyle\frac{\delta }{%
\left\Vert \delta \right\Vert }\Ep [z_i w_{i, \mu}(\indx)]+\frac{1}{%
\left\Vert \delta \right\Vert }\Ep[R_i(\delta,\mu, \indx)],
\end{equation*}%
where 
\begin{eqnarray*}
R_{i}(\delta ,\mu ,\indx):= &&\left( \mu +\delta \right) z_{i}\left( \theta
_{1i}(\indx)-\theta _{0i}(\indx)\right) {1}\left( \mu z_{i}<0<\left( \mu
+\delta \right) z_{i}\right) \\
&+&\left( \mu +\delta \right) z_{i}\left( \theta _{0i}(\indx)-\theta _{1i}(%
\indx)\right) {1}\left( \mu z_{i}>0>\left( \mu +\delta \right) z_{i}\right) .
\end{eqnarray*}%
By Cauchy-Schwarz and the maintained assumptions 
\begin{eqnarray*}
\sup_{\mu \in M,\indx\in \indxSet}\Ep|R_{i}(\delta ,\mu ,\indx)| &\lesssim
&\Vert \delta z\Vert _{\Pr ,2}\cdot \sup_{\indx\in \indxSet,\ell \in
\{0,1\}}\Vert \theta _{\ell i}(\indx)\Vert _{\Pr ,2}\sup_{\mu \in M,\indx\in %
\indxSet}[\Pr \left( |\mu z|<|\delta z|\right) ]^{1/2} \\
&\lesssim &\Vert \delta z\Vert _{\Pr ,2}\cdot 1\cdot \delta ^{m/2}.
\end{eqnarray*}%
Therefore, as $\Vert \delta \Vert \rightarrow 0$ 
\begin{equation*}
\sup_{\mu \in M,\indx\in \indxSet}\frac{1}{\left\Vert \delta \right\Vert }|%
\Ep [R_i(\delta,\mu, u)]|\leq \sup_{\mu \in M,\indx\in \indxSet}\frac{1}{%
\left\Vert \delta \right\Vert }\Ep|R_{i}(\delta ,\mu ,\indx)|\lesssim \delta
^{m/2}\rightarrow 0.\qed
\end{equation*}

\quad

\begin{lemma}[Coupling Lemma]
\label{VC: coupling lemma} 1. (Sample) We have that 
\begin{equation*}
\Gn[h(t)]=\G[h(t)]+o_{\Pr }(1)\text{ in }\ell ^{\infty }(T),
\end{equation*}%
where $\G$ is a P-Brownian bridge with covariance function $\Ep[h(t) h(t')]-%
\Ep[h(t)]\Ep[h(t')]$.

2. (Bootstrap). We have that 
\begin{equation*}
\Gn[e^oh^o(t)]=\widetilde{\G[h(t)]}+o_{\Pr }(1)\text{ in }\ell ^{\infty }(T),
\end{equation*}%
where $\widetilde{\G}$ is a P-Brownian bridge with covariance function $%
\Ep[h(t) h(t')]-\Ep[h(t)]\Ep[h(t')]$.
\end{lemma}

\begin{proof}
The proof can be accomplished by using a single common notation.
Specifically it will suffice to show that for either the case $g_{i}=1$ or $%
g_{i}=e_{i}-1$ 
\begin{equation*}
\Gn[g h^o]=\G^{g}[h(t)]+o_{\Pr }(1)\text{ in }\ell ^{\infty }(T),
\end{equation*}%
where $\G$ is a P-Brownian bridge with covariance function $\Ep[h(t) h(t')]-%
\Ep[h(t)]\Ep[h(t')]$. The process $\G^{g}$ for the case of $g_{i}=1$ is
different (in fact independent) of the process $\G^{g}$ for the case of $%
g_{i}=e_{i}-1$, but they both have identical distributions. Once we
understand this, we can drop the index $g$ for the process.

Within this proof, it will be convenient to define: 
\begin{equation*}
S_{n}(t):=\Gn[g h^o(t)]\ \text{ and }\ Z_{n}(t):=\G[h(t)].
\end{equation*}

Let $B_{jk},j=1,...,p$ be a partition of $T$ into sets of diameter at most $%
j^{-1}$. We need at most 
\begin{equation*}
p\lesssim j^{d},\ \ d=\dim (T)
\end{equation*}%
such partition sets. Choose $t_{jk}$ as arbitrary points in $B_{jk}$, for
all $j=1,...,p$. We define the sequence of projections $\pi
_{j}:T\rightarrow T$, $j=0,1,2,\ldots ,\infty $ by $\pi _{j}(t)=t_{jk}$ if $%
t\in B_{jk}$.

In what follows, given a process $Z$ in $\ell ^{\infty }(T)$ and its
projection $Z\circ \pi _{j}$, whose paths are constant over the partition
set, we shall identify the process $Z\circ \pi _{j}$ with a random vector $%
Z\circ \pi _{j}$ in $\mathbb{R}^{p}$, when convenient. Analogously, given a
random vector $Z$ in $\mathbb{R}^{p}$ we identify it with a process $Z$ in $%
\ell ^{\infty }(T)$, whose paths are constant over the elements of the
partition sets.

The result follows from the following relations proven below:\smallskip

\textbf{1. Finite-Dimensional Approximation.} As $j/\log n\rightarrow \infty 
$, then $\Delta _{1}$ $=$ $\sup_{t\in T}$ $\Vert S_{n}(t)-S_{n}\circ \pi
_{j}(t)\Vert \rightarrow _{\Pr }0.$  

\textbf{2. Coupling with a Normal Vector.} There exists $\N_{nj}=_{d}N(0,%
\text{var}[S_{n}\circ \pi _{j}])$ such that, if $p^{5}\xi
_{k}^{2}/n\rightarrow 0$, then $\Delta _{2}=\sup_{j}|\N_{nj}-S_{n}\circ \pi
_{j}|\rightarrow _{\Pr }0.$

\textbf{3. Embedding a Normal Vector into a Gaussian Process.} There exists
a Gaussian process $Z_{n}$ with the properties stated in the lemma such that 
$\N_{nj}=Z_{n}\circ \pi _{j}$ almost surely.

\textbf{4. Infinite-Dimensional Approximation.} if $j\rightarrow \infty $,
then $\Delta _{3}$ $=$ $\sup_{t\in T}$ $|Z_{n}(t)-Z_{n}\circ \pi
_{j}(t)|\rightarrow _{\Pr }0.$\smallskip

We can select the sequence $j=\log ^{2}n$ such that the conditions on $j$
stated in relations (1)-(4) hold. We then conclude using the triangle
inequality that 
\begin{equation*}
\sup_{t\in T}|S_{n}(t)-Z_{n}(t)|\leq \Delta _{1}+\Delta _{2}+\Delta
_{3}\rightarrow _{\Pr }0.
\end{equation*}

Relation 1 follows from 
\begin{equation*}
\Delta _{1}=\sup_{t\in T}|S_{n}(t)-S_{n}\circ \pi _{j}(t)|\leq
\sup_{\left\Vert t-t^{\prime }\right\Vert \leq
j^{-1}}|S_{n}(t)-S_{n}(t^{\prime })|\rightarrow _{\Pr }0,
\end{equation*}%
where the last inequality holds by Lemma \ref{VC: lemma bound modulus of
continuity}.

Relation 2 follows from the use of Yurinskii's coupling (%
\citeasnoun[page
244]{Pollard02}): Let $\zeta _{1},\ldots ,\zeta _{n}$ be independent $p$%
-vectors with $\Ep\zeta _{i}=0$ for each $i$, and $\kappa :=\sum_{i}\Ep\left[
\Vert \zeta _{i}\Vert ^{3}\right] $ finite. Let $S=\zeta _{1}+\cdot +\zeta
_{n}$. For each $\delta >0$ there exists a random vector $T$ with a $N(0,%
\text{var}(S))$ distribution such that 
\begin{equation*}
\Pp\{\Vert S-T\Vert >3\delta \}\leq C_{0}B\left( 1+\frac{|\log (1/B)|}{p}%
\right) \text{ where }B:=\kappa p\delta ^{-3},
\end{equation*}%
for some universal constant $C_{0}$.

In order to apply the coupling, we collapse $S_{n}\circ \pi _{j}$ to a $p$%
-vector, and we let 
\begin{equation*}
\zeta _{i}=\zeta _{1i}+...+\zeta _{4i}\in \mathbb{R}^{p},\ \ \zeta
_{li}=g_{i}h_{li}^{o}\circ \pi \in \mathbb{R}^{p},
\end{equation*}%
where $h_{li},l=1,...,4$ are defined in (\ref{eq: define h}), so that $%
S_{n}\circ \pi _{j}=\sum_{i=1}^{n}\zeta _{i}/\sqrt{n}$. Now note that since $%
\Ep\lbrack \Vert \zeta _{i}\Vert ^{3}]\lesssim \max_{1\leq l\leq 4}%
\Ep[\|\zeta_{li}\|^3]$ and 
\begin{eqnarray*}
\Ep\Vert \zeta _{li}\Vert ^{3} &=&p^{3/2}\Ep\left( \frac{1}{p}%
\sum_{k=1}^{p}|g_{i}h_{li}^{o}(t_{kj})|^{2}\right) ^{3/2}\leq p^{3/2}\Ep%
\left( \frac{1}{p}\sum_{k=1}^{p}|g_{i}h_{li}^{o}(t_{kj})|^{3}\right) \\
&\leq &p^{3/2}\sup_{t\in T}\Ep|h_{li}^{o}(t_{kj})|^{3}\Ep|g_{i}|^{3},
\end{eqnarray*}%
where we use the independence of $g_{i}$, we have that 
\begin{equation*}
\Ep [\|\zeta_i\|^3]\lesssim p^{3/2}\max_{1\leq l\leq 4}\sup_{t\in T}\Ep%
|h_{li}^{o}(t)|^{3}\Ep|g_{i}|^{3}.
\end{equation*}

Next we bound the right side of the display above for each $l$. First, for $%
A(t):=q^{\prime }\Sigma \Er[z_i p_i' 1\{q '\Sigma z_i >0\}]J_{1}^{-1}(\indx)$
\begin{eqnarray*}
\sup_{t\in T}\Ep|h_{1i}^{o}(t))|^{3} &=&\sup_{t\in T}\Ep\left\vert
A(t)p_{i}\varphi _{i1}(\indx)\right\vert ^{3}\leq \sup_{t\in T}\Vert
A(t)\Vert ^{3}\cdot \sup_{\Vert \delta \Vert =1}\Ep|\delta ^{\prime
}p_{i}|^{3}\sup_{\indx\in \indxSet,x\in X}\Er [|\varphi_i(\indx)|^3|x_i=x] \\
&\lesssim &\sup_{\Vert \delta \Vert =1}\Ep|\delta ^{\prime }p_{i}|^{3}\sup_{%
\indx\in \indxSet,x\in X}\Er [|\varphi_i(\indx)|^3|x_i=x] \\
&\lesssim &\xi _{k}\sup_{\Vert \delta \Vert =1}\Ep|\delta ^{\prime
}p_{i}|^{2}\sup_{\indx\in \indxSet,x\in X}\Er [|\varphi_i(\indx)|^3|x_i=x]%
\lesssim \xi _{k},
\end{eqnarray*}%
where we used the assumption that $\sup_{\indx\in \indxSet,x\in X}\Er\lbrack
|\varphi _{i}(\indx)|^{3}|x_{i}=x]\lesssim 1$, $\Vert \Ep p_{i}p_{i}\Vert
\lesssim 1$, and that 
\begin{equation*}
\sup_{t\in T}\Vert A(t)\Vert \leq \sup_{\Vert \delta \Vert =1}[%
\Ep[z_i'\delta]^{2}]^{1/2}\sup_{\Vert \delta \Vert =1}[\Ep[p_i'\delta]%
^{2}]^{1/2}\sup_{\indx\in \indxSet}\Vert J^{-1}(\indx)\Vert \lesssim 1,
\end{equation*}%
where the last bound is true by assumption. Similarly $\Ep%
|h_{2i}^{o}(t))|^{3}\lesssim \xi _{k}$. Next 
\begin{eqnarray}
\sup_{t\in T}\Ep|h_{3i}^{o}(t)|^{3} &=&\sup_{t\in T}\Ep|q^{\prime }\Sigma
\left( x_{i}z_{i}^{\prime }\right) ^{o}\Sigma \Ep[z_i, w_{i, q'\Sigma}
(\indx)]|^{3}  \notag \\
&\lesssim &\Ep\Vert \left( x_{i}z_{i}^{\prime }\right) ^{o}\Vert
^{3}\sup_{t\in T}\Vert \left[ \Ep[z_i w_{i, q'\Sigma}(\indx)]\right]
\Vert ^{3}  \notag \\
&\lesssim &\left( \Ep\left\Vert \left( x_{i}z_{i}^{\prime }\right)
\right\Vert ^{3}+\left\Vert \Ep\left( x_{i}z_{i}^{\prime }\right)
\right\Vert ^{3}\right) \left( \Ep\Vert z_{i}\Vert ^{2}\right) ^{3/2}\sup_{%
\indx\in \indxSet}\Ep\left[ \left\vert \theta _{li}\left( \indx\right)
\right\vert ^{2}\right] ^{3/2} \notag \\
& \lesssim & 1,
\end{eqnarray}%
where the last bound follows from assumptions \ref{c:design conditions}.
Finally, 
\begin{eqnarray*}
\sup_{t\in T}\left[ \Ep|h_{4i}^{o}(t)|^{3}\right] ^{1/3} &=&\sup_{t\in T}%
\left[ \Ep|q^{\prime }\Sigma z_{i}w_{i,q^{\prime }\Sigma }(\indx)|^{3}\right]
^{1/3}+\sup_{t\in T}|\Ep q^{\prime }\Sigma z_{i}w_{i,q^{\prime }\Sigma }(%
\indx)| \\
&\leq &2\sup_{t\in T}[\Ep|q^{\prime }\Sigma z_{i}|^{6}]^{1/6}\left[ \Ep%
|w_{i,q^{\prime }\Sigma }(\indx)|^{6}\right] ^{1/6} \\
&\lesssim &[\Ep|z_{i}|^{6}]^{1/6}\sup_{\indx\in \indxSet}\Ep [
|\theta_{li}(\indx)|^6 ]^{1/6}\lesssim 1,
\end{eqnarray*}%
where the last line follows from assumption \ref{c:design conditions}.

Therefore, by Yurinskii's coupling, observing that in our case by the above
arguments $\kappa =\frac{p^{3/2}\xi _{k}n}{(\sqrt{n})^{3}},$ for each $%
\delta >0$ if $p^{5}\xi _{k}^{2}/n\rightarrow 0$, 
\begin{equation*}
\Pp\left\{ \left\Vert \frac{\sum_{i=1}^{n}\zeta _{i}}{\sqrt{n}}-\mathcal{N}%
_{n,j}\right\Vert \geq 3\delta \right\} \lesssim \frac{npp^{3/2}\xi _{k}}{%
(\delta \sqrt{n})^{3}}=\frac{p^{5/2}\xi _{k}}{(\delta ^{3}n^{1/2})}%
\rightarrow 0.
\end{equation*}%
This verifies relation (2).

Relation (3) follows from the a.s. embedding of a finite-dimensional random
normal vector into a path of a Gaussian process whose paths are continuous
with respect to the standard metric $\rho _{2}$, defined in Lemma \ref%
{lemma: covariance properties}, which is shown e.g., in Belloni and
Chernozhukov (2009). Moreover, since $\rho _{2}$ is continuous with respect
to the Euclidian metric on $T$, as shown in part 2 of Lemma \ref{lemma:
covariance properties}, the paths of the process are continuous with respect
to the Euclidian metric as well.

Relation (4) follows from the inequality 
\begin{equation*}
\Delta _{3}=\sup_{t\in T}|Z_{n}(t)-Z_{n}\circ \pi _{j}(t)|\leq
\sup_{\left\Vert t-t^{\prime }\right\Vert \leq
j^{-1}}|Z_{n}(t)-Z_{n}(t^{\prime })|\lesssim _{\Pr }(1/j)^{c}\log
(1/j)^{c}\rightarrow 0,
\end{equation*}%
where $0<c\leq 1/2$ is defined in Lemma \ref{lemma: covariance properties}.
This inequality follows from the entropy inequality for Gaussian processes
(Corollary 2.2.8 of \citeasnoun{vaartwellner}) 
\begin{equation*}
\Ep\sup_{\rho _{2}(t,t^{\prime })\leq \delta }\left\vert
Z_{n}(t)-Z_{n}(t_{0}^{\prime })\right\vert \leq \tint_{0}^{\delta }\sqrt{%
\log N(\epsilon ,T,\rho _{2})}d\epsilon
\end{equation*}%
and parts 2 and 3 of Lemma \ref{lemma: covariance properties}. From part 2
of Lemma \ref{lemma: covariance properties} we first conclude that 
\begin{equation*}
\log N(\epsilon ,T,\rho _{2})\lesssim \log (1/\epsilon ),
\end{equation*}%
and second that $\Vert t-t^{\prime }\Vert \leq (1/j)$ implies $\rho
_{2}\left( t,t^{\prime }\right) \leq \left( 1/j\right) ^{c},$ so that 
\begin{equation*}
\Ep\sup_{\Vert t-t^{\prime }\Vert \leq 1/j}|Z_{n}(t)-Z_{n}(t^{\prime })|\leq
(1/j)^{c}\log (1/j)^{c}\text{ as }j\rightarrow \infty .
\end{equation*}%
The claimed inequality then follows by Markov inequality.
\end{proof}

\begin{lemma}[Bounded Oscillations]
\label{VC: lemma bound modulus of continuity}

1. (Sample) For $\epsilon_n = o( (\log n)^{-1/(2c)})$, we have that 
\begin{equation*}
\sup_{\| t -t^{\prime }\| \leq \epsilon_n} | \Gn [h(t) - h(t')]| \to_{\Pr} 0.
\end{equation*}

2. (Bootstrap). For $\epsilon_n = o( (\log n)^{-1/(2c)})$, we have that 
\begin{equation*}
\sup_{\| t -t^{\prime }\| \leq \epsilon_n} | \Gn [(e_i-1) ( h^o(t) -
h^o(t'))]| \to_{\Pr} 0.
\end{equation*}
\end{lemma}

\textbf{Proof.} To show both statements, it will suffice to show that for
either the case $g_{i}=1$ or $g_{i}=e_{i}-1$, we have that 
\begin{equation*}
\sup_{\Vert t-t^{\prime }\Vert \leq \epsilon _{n}}|\Gn [g_i (h^o(t) -
h^o(t'))]|\rightarrow _{\Pr }0.
\end{equation*}

\textbf{Step 1.} Since 
\begin{equation*}
\sup_{\Vert t-t^{\prime }\Vert \leq \epsilon _{n}}|\Gn [g_i (h^o(t) -
h^o(t'))]|\lesssim \max_{1\leq \ell \leq 4}\sup_{\Vert t-t^{\prime }\Vert
\leq \epsilon _{n}}|\Gn [g_i (h^o_\ell(t) - h^o_\ell(t'))]|,
\end{equation*}%
we bound the latter for each $\ell $. Using the results in Lemma \ref{lemma:
complexities} that bound the random entropy of $\mathcal{H}_{1}$ and $%
\mathcal{H}_{2}$ and the results in Lemma \ref{lemma: maximal} we have that
for $\ell =1$ and $2$ 
\begin{equation*}
\Delta _{n\ell }=\sup_{\Vert t-t^{\prime }\Vert \leq \epsilon _{n}}|\Gn [g_i
(h^o_\ell(t) - h^o_\ell(t'))]|\lesssim _{\Pr }\sqrt{\log n}\sup_{\Vert
t-t^{\prime }\Vert \leq \epsilon _{n}}\max_{\mathbb{P}\in \{\Pr ,\Pn\}}\Vert
g_{i}(h_{\ell }^{o}(t)-h_{\ell }^{o}(t^{\prime }))\Vert _{\mathbb{P},2}.
\end{equation*}%
By Lemma \ref{lemma: complexities} that bounds the entropy of $g_{i}(%
\mathcal{H}_{\ell }^{o}-\mathcal{H}_{\ell }^{o})^{2}$ and Lemma \ref{lemma:
maximal} we have that for $\ell =1$ or $\ell =2$, 
\begin{equation*}
\sup_{\Vert t-t^{\prime }\Vert \leq \epsilon _{n}}\Big |\Vert g(h_{\ell
}^{o}(t)-h_{\ell }^{o}(t^{\prime }))\Vert _{\Pn,2}-\Vert g(h_{\ell
}^{o}(t)-h_{\ell }^{o}(t^{\prime }))\Vert _{\Pr ,2}\Big |\lesssim _{\Pr }%
\sqrt{\frac{\log n}{n}}\sup_{t\in T}\max_{\mathbb{P}\in \{\Pr ,\Pn\}}\Vert
g^{2}h_{\ell }^{o2}(t)\Vert _{\mathbb{P},2}.
\end{equation*}%
By Step 2 below we have 
\begin{equation*}
\sup_{t\in T}\max_{\mathbb{P}\in \{\Pr ,\Pn\}}\Vert g^{2}h_{\ell
}^{o2}(t)\Vert _{\mathbb{P},2}\lesssim _{\Pr }\sqrt{\xi _{k}^{2}\max_{i\leq
n}F_{1}^{4}\max_{i}|g_{i}|^{4}}\lesssim _{\Pr }\sqrt{\xi _{k}^{2}\max_{i\leq
n}F_{1}^{4}(\log n)^{4}},
\end{equation*}%
and by Lemma \ref{lemma: covariance properties}, $\Vert g(h_{\ell
}(t)-h_{\ell }(t^{\prime }))\Vert _{\Pr ,2}\lesssim \left\Vert t-t^{\prime
}\right\Vert ^{c}.$ Putting the terms together we conclude 
\begin{equation*}
\Delta _{n\ell }\lesssim _{\Pr }\sqrt{\log n}\left( \epsilon _{n}^{c}+\sqrt{%
\frac{\log n}{n}}\xi _{k}\sqrt{\max_{i\leq n}F_{1}^{4}}(\log n)^{2}\right)
\rightarrow _{\Pr }0,
\end{equation*}%
by assumption and the choice of $\epsilon _{n}$.

For $\ell =3$ and $\ell =4$, by Lemma \ref{lemma: complexities}, $g(\mathcal{%
H}_{3}^{o}-\mathcal{H}_{3}^{o})$ and $g(\mathcal{H}_{4}^{o}-\mathcal{H}%
_{4}^{o})$ are P-Donsker, so that 
\begin{equation*}
\Delta _{n\ell }\leq \sup_{\rho _{2}\left( t,t_{n}^{\prime }\right) \leq
\epsilon _{n}^{c}}|\Gn [g (h^o_{\ell}(t) - h^o_{\ell}(t'))]|\rightarrow
_{\Pr }0.\qed
\end{equation*}

\textbf{Step 2.} Since $\Vert g^{2}h_{\ell }^{o2}(t)\Vert _{\mathbb{P}%
,2}\leq 2\Vert g^{2}h_{\ell }^{2}(t)\Vert _{\mathbb{P},2}+2\Vert
g^{2}E[h_{\ell }^{o}(t)]^{2}\Vert _{\mathbb{P},2},$ for $\mathbb{P}\in \{\Pr
,\Pn\}$, it suffices to bound each term separately.

Uniformly in $t\in T$ for $\ell =1,2$ 
\begin{gather*}
\En[g h_\ell(t)]^{4}\leq \max_{i}g_{i}^{4}\cdot \Vert \Sigma \Vert ^{4}\Vert %
\En[z_i p_i 1\{ q'\Sigma z_i<0 \}]\Vert \Vert J_{0}^{-1}(\indx)\Vert \cdot
\sup_{\Vert \delta \Vert =1}\En[[\delta'p_i]^4 \varphi_{i0}^4(\indx)] \\
\lesssim _{\Pr }(\log n)^{4}\cdot 1\cdot \xi _{k}^{2}\Vert \En[ p_ip_i']%
\Vert \max_{i\leq n}F_{1}^{4}\lesssim _{\Pr }\xi _{k}^{2}\max_{i\leq
n}F_{1}^{4}(\log n)^{4},
\end{gather*}%
where we used assumptions \ref{c:design conditions} and \ref{c:complexity
function classes} and the fact that $\left\Vert \En\left[ z_{i}p_{i}1\{q^{%
\prime }\Sigma z_{i}<0\}\right] \right\Vert \lesssim _{\Pr }1$ and $\En\left[
p_{i}p_{i}^{\prime }\right] \lesssim _{\Pr }1$ as shown in the proof of
Lemma \ref{VC lemma: linearization}.

Uniformly in $t\in T$ for $\ell =1,2$%
\begin{eqnarray*}
\Ep[g h_\ell(t)]^{4} &\leq &\Ep [g^4]\Vert \Sigma \Vert ^{4}\Vert \Ep[z_i
p_i 1\{ q'\Sigma z_i<0 \}]\Vert \Vert J_{0}^{-1}(\indx )\Vert \cdot
\sup_{\Vert \delta \Vert =1}\Ep[[\delta'p_i]^4 \varphi_{i0}^4(\indx)] \\
\lesssim _{\Pr } &&1\cdot \xi _{k}^{2}\Vert \Ep[ p_ip_i']\Vert \sup_{x\in X,%
\indx \in \indxSet}\Ep[\varphi^4_{i0}(\indx)|x_i=x]\lesssim \xi _{k}^{2},
\end{eqnarray*}%
where we used assumption \ref{c:design conditions}.

Uniformly in $t\in T$ for $\ell =1,2$ 
\begin{equation*}
\En[ g^4 \Er[h^{o}_\ell(t)]^4 ]\leq \En g^{4}\Er[h^{o}_\ell(t)]^{4}\lesssim
_{P}1\cdot \Er[h^{o2}_\ell(t)]^{2}\lesssim 1,
\end{equation*}%
and 
\begin{equation*}
\Ep[ g^4 \Er[h^{o}_\ell(t)]^4 ]\leq \Ep g^{4}\Er[h^{o}_\ell(t)]^{4}\lesssim
1\cdot \Er[h^{o2}_\ell(t)]^{2}\lesssim 1.
\end{equation*}%
where the bound in $\Er[h^{o2}_\ell(t)]^{2}$ follows from calculations given
in the proof of Lemma \ref{lemma: covariance properties}. \qed

\begin{lemma}[Covariance Properties]
\label{lemma: covariance properties}

\begin{itemize}
\item[1.] For some $0<c\leq 1/2$ 
\begin{equation*}
\rho _{2}(h(t),h(t^{\prime }))=\left( \Ep[ h(t) -h(t') ]^{2}\right)
^{1/2}\lesssim \rho (t,t^{\prime }):=\left\Vert t-t^{\prime }\right\Vert ^{c}
\end{equation*}

\item[2.] The covariance function $\Ep[h(t) h(t')] - \Ep[h(t)] \Ep[h(t')]$
is equi-continuous on $T\times T$ uniformly in $k$.

\item[3.] A sufficient condition for the variance function to be bounded
away from zero, $\inf_{t\in T}\Var(h(t))\geq L>0$, uniformly in $k$ is that
the following matrices have minimal eigenvalues bounded away from zero
uniformly in $k$: $\text{var}\left( 
\begin{bmatrix}
\varphi _{i1}(\indx) & \varphi _{i0}(\indx)%
\end{bmatrix}%
^{\prime }|x_{i},z_{i}\right) $ $\Er[p_i p_i']$, $J_{0}^{-1}(\indx)$, $%
J_{1}^{-1}(\indx)$, $b_{0}^{\prime }b_{0}$, and $b_{1}^{\prime }b_{1}$,
where $b_{1}=\Er[z_{i}p_{i}'1\{q'\Sigma   z_{i}>0\}]$ and $b_{0}=%
\Er[z_{i}p_{i}'1\{q'\Sigma z_{i}<0\}]$.
\end{itemize}
\end{lemma}

\begin{remark}
We emphasize that claim 3 only gives sufficient conditions for $\text{var}%
\left( h(t)\right) $ to be bounded away from zero. In particular, the
assumption that 
\begin{equation*}
\mineig\left( \text{var}\left( 
\begin{bmatrix}
\varphi _{i1}(\indx) & \varphi _{i0}(\indx)%
\end{bmatrix}%
^{\prime }|x_{i},z_{i}\right) \right) \geq L
\end{equation*}%
is not necessary, and does not hold in all relevant situations. For example,
when the upper and lower bounds have first-order equivalent asymptotics,
which can occur in the point-identified and local to point-identified cases,
this condition fails. However, the result still follows from equation (\ref%
{eq:varstar}) under the assumption that 
\begin{equation*}
\text{var}\left( \varphi _{i1}(\indx)|x_{i},z_{i}\right) =\text{var}\left(
\varphi _{i0}(\indx)|x_{i},z_{i}\right) \geq L
\end{equation*}
\end{remark}

\begin{proof}
\textbf{Claim 1.} Observe that $\rho _{2}\left( h(t),h(\tilde{t})\right)
\lesssim \max_{j}\rho _{2}\left( h_{j}(t),h_{j}(\tilde{t})\right) $. We will
bound each of these four terms. For the first term, we have 
\begin{align*}
\rho _{2}(h_{1}(t),h_{1}(\tilde{t}))=& \Ep\left[ \left( 
\begin{array}{c}
q^{\prime }\Sigma \Ep\left[ z_{i}p_{i}^{\prime }{1}\left\{ q^{\prime }\Sigma
z_{i}>0\right\} \right] J_{1}^{-1}\left( \indx\right) p_{i}\varphi
_{i1}\left( \indx\right) - \\ 
-\tilde{q}^{\prime }\Sigma \Ep\left[ z_{i}p_{i}^{\prime }{1}\left\{ \tilde{q}%
^{\prime }\Sigma z_{i}>0\right\} \right] J_{1}^{-1}\left( \tilde{\indx}%
\right) p_{i}\varphi _{i1}\left( \tilde{\indx}\right)%
\end{array}%
\right) ^{2}\right] ^{1/2} \\
\leq & \Ep\left[ \left( (q-\tilde{q})^{\prime }\Sigma \Ep\left[
z_{i}p_{i}^{\prime }{1}\left\{ q^{\prime }\Sigma z_{i}>0\right\} \right]
J_{1}^{-1}\left( \indx\right) p_{i}\varphi _{i1}\left( \indx\right) \right)
^{2}\right] ^{1/2}+ \\
& +\Ep\left[ \left( \tilde{q}^{\prime }\Sigma \left( 
\begin{array}{c}
\Ep\left[ z_{i}p_{i}^{\prime }{1}\left\{ q^{\prime }\Sigma z_{i}>0\right\} %
\right] - \\ 
-\Ep\left[ z_{i}p_{i}^{\prime }{1}\left\{ \tilde{q}^{\prime }\Sigma
z_{i}>0\right\} \right]%
\end{array}%
\right) J_{1}^{-1}\left( \indx\right) p_{i}\varphi _{i1}\left( \indx\right)
\right) ^{2}\right] ^{1/2}+ \\
& +\Ep\left[ \left( \tilde{q}^{\prime }\Sigma \Ep\left[ z_{i}p_{i}^{\prime }{%
1}\left\{ \tilde{q}^{\prime }\Sigma z_{i}>0\right\} \right] \left(
J_{1}^{-1}\left( \indx\right) -J_{1}^{-1}\left( \tilde{\indx}\right) \right)
p_{i}\varphi _{i1}\left( \indx\right) \right) ^{2}\right] ^{1/2}+ \\
& +\Ep\left[ \left( \tilde{q}^{\prime }\Sigma \Ep\left[ z_{i}p_{i}^{\prime }{%
1}\left\{ \tilde{q}^{\prime }\Sigma z_{i}>0\right\} \right] J_{1}^{-1}\left( 
\tilde{\indx}\right) p_{i}\left( \varphi _{i1}\left( \indx\right) -\varphi
_{i1}\left( \tilde{\indx}\right) \right) \right) ^{2}\right] ^{1/2}
\end{align*}%
For the first term we have 
\begin{align*}
& \Ep\left[ \left( (q-\tilde{q})^{\prime }\Sigma \Ep\left[
z_{i}p_{i}^{\prime }{1}\left\{ q^{\prime }\Sigma z_{i}>0\right\} \right]
J_{1}^{-1}\left( \indx\right) p_{i}\varphi _{i1}\left( \indx\right) \right)
^{2}\right] ^{1/2}\leq \\
& \,\,\,\,\leq \Vert q-\tilde{q}\Vert \norm{\Sigma}\norm{\Ep\left[z_i p_i'
\right]}\left\Vert J_{1}^{-1}(\indx)\right\Vert \Ep[\norm{p_i p_i'}^2]%
^{1/2}\sup_{x_{i},z_{i}}\Ep[\varphi_{i\ell}(\indx)^4|x_i,z_i]^{1/2}
\end{align*}%
By assumption \ref{c:design conditions}, $\left\Vert \Ep\left[
z_{i}p_{i}^{\prime }\right] \right\Vert $, $\left\Vert J_{1}^{-1}(\indx%
)\right\Vert $, $\Ep[\norm{p_i
  p_i'}^2]$, and $\sup_{x_{i},z_{i}}\Ep[\varphi_{i1}(\indx)^4|x_i,z_i]$ are
bounded uniformly in $k$ and $\indx$. Therefore, 
\begin{equation*}
\Ep\left[ \left( (q-\tilde{q})^{\prime }\Sigma \Ep\left[ z_{i}p_{i}^{\prime }%
{1}\left\{ q^{\prime }\Sigma z_{i}>0\right\} \right] J_{1}^{-1}\left( \indx%
\right) p_{i}\varphi _{i1}\left( \indx\right) \right) ^{2}\right]
^{1/2}\lesssim \norm{ q-\tilde{q} }
\end{equation*}%
The same conditions give the following bound on the second term. 
\begin{align*}
& \Ep\left[ \left( \tilde{q}^{\prime }\Sigma \left( 
\begin{array}{c}
\Ep\left[ z_{i}p_{i}^{\prime }{1}\left\{ q^{\prime }\Sigma z_{i}>0\right\} %
\right] - \\ 
-\Ep\left[ z_{i}p_{i}^{\prime }{1}\left\{ \tilde{q}^{\prime }\Sigma
z_{i}>0\right\} \right]%
\end{array}%
\right) J_{1}^{-1}\left( \indx\right) p_{i}\varphi _{i1}\left( \indx\right)
\right) ^{2}\right] ^{1/2} \\
& \,\,\,\,\lesssim \left\Vert \Ep\left[ {1}\left\{ q^{\prime }\Sigma
z_{i}>0\right\} -{1}\left\{ \tilde{q}^{\prime }\Sigma z_{i}>0\right\} \right]
\right\Vert \\
& \,\,\,\,\lesssim \Ep\left[ \left( {1}\left\{ q^{\prime }\Sigma
z_{i}>0\right\} -{1}\left\{ \tilde{q}^{\prime }\Sigma z_{i}>0\right\}
\right) ^{2}\right] ^{1/2} \\
&
\end{align*}%
As in step 5 of the proof of Lemma \ref{VC lemma: linearization}, the
assumption that $\Pr \left( |q^{\prime }\Sigma z_{i}/\Vert z_{i}\Vert
|<\delta \right) /\delta ^{m}\lesssim 1$ implies 
\begin{equation*}
\Ep\left[ \left( {1}\left\{ q^{\prime }\Sigma z_{i}>0\right\} -{1}\left\{ 
\tilde{q}^{\prime }\Sigma z_{i}>0\right\} \right) ^{2}\right] ^{1/2}\lesssim
\left\Vert q-\tilde{q}\right\Vert ^{m/2}
\end{equation*}

Similarly, the third term is bounded as follows: 
\begin{equation*}
\Ep\left[ \left( \tilde{q}^{\prime }\Sigma \Ep\left[ z_{i}p_{i}^{\prime }{1}%
\left\{ \tilde{q}^{\prime }\Sigma z_{i}>0\right\} \right] \left(
J_{1}^{-1}\left( \indx\right) -J_{1}^{-1}\left( \tilde{\indx}\right) \right)
p_{i}\varphi _{i1}\left( \indx\right) \right) ^{2}\right] ^{1/2}\lesssim
\left\Vert J_{1}^{-1}\left( \indx\right) -J_{1}^{-1}\left( \tilde{\indx}%
\right) \right\Vert .
\end{equation*}%
Note that $J_{1}^{-1}(\indx)$ is uniformly Lipschitz in $\indx\in \indxSet$
by assumption \ref{c:design conditions}, so $\left\Vert J_{1}^{-1}\left( %
\indx\right) -J_{1}^{-1}\left( \tilde{\indx}\right) \right\Vert \lesssim
\left\Vert \indx-\tilde{\indx}\right\Vert $ . Finally, the fourth term is
bounded by 
\begin{align*}
\Ep& \left[ \left( \tilde{q}^{\prime }\Sigma \Ep\left[ z_{i}p_{i}^{\prime }{1%
}\left\{ \tilde{q}^{\prime }\Sigma z_{i}>0\right\} \right] J_{1}^{-1}\left( 
\tilde{\indx}\right) p_{i}\left( \varphi _{i1}(\indx)-\varphi _{i1}(\tilde{%
\indx})\right) \right) ^{2}\right] ^{1/2}\lesssim \\
& \,\,\,\,\,\lesssim \sup_{x_{i},z_{i}}\Ep\left[ \left( \varphi _{i1}(\indx%
)-\varphi _{i1}(\tilde{\indx})\right) ^{4}|x_{i},z_{i}\right] ^{1/2} \\
& \,\,\,\,\,\lesssim \left\Vert \indx-\tilde{\indx}\right\Vert ^{\gamma
_{\varphi }}
\end{align*}%
where we used the assumption that $\Ep\left[ \left( \varphi _{i1}(\indx%
)-\varphi _{i1}(\tilde{\indx})\right) ^{4}|x_{i},z_{i}\right] ^{1/2}$ is
uniformly $\gamma _{\varphi }$-H\"{o}lder continuous in $\indx$. Combining,
we have 
\begin{align*}
\rho _{2}(h_{1}(t),h_{1}(\tilde{t}))\lesssim & \left\Vert q-\tilde{q}%
\right\Vert +\left\Vert q-\tilde{q}\right\Vert ^{m/2}+\left\Vert \indx-%
\tilde{\indx}\right\Vert +\left\Vert \indx-\tilde{\indx}\right\Vert ^{\gamma
_{\varphi }} \\
\lesssim & \left\Vert t-t^{\prime }\right\Vert ^{1\wedge m/2\wedge \gamma
_{\varphi }}
\end{align*}

An identical argument shows that $\rho _{2}\left( h_{2}(t),h_{2}(t^{\prime
})\right) \lesssim \left\Vert t-t^{\prime }\right\Vert ^{1\wedge m/2\wedge
\gamma _{\varphi }}$.

The third and fourth components of $h(t)$ can be bounded using similar
arguments. For $h_{3}(t)$, we have 
\begin{eqnarray*}
\rho _{2}\left( h_{3}(t),h_{3}(\tilde{t})\right) &=&\Ep\left[ \left(
q^{\prime }\Sigma x_{i}z_{i}^{\prime }\Sigma \Ep\left[ z_{i}w_{i,q^{\prime
}\Sigma }(\indx )\right] -\tilde{q}^{\prime }\Sigma x_{i}z_{i}^{\prime
}\Sigma \Ep\left[ z_{i}w_{i,\tilde{q}^{\prime }\Sigma }(\tilde{\indx})\right]
\right) ^{2}\right] ^{1/2} \\
&\lesssim &\Ep\left[ \left( (q-\tilde{q})^{\prime }\Sigma x_{i}z_{i}^{\prime
}\Sigma \Ep\left[ z_{i}w_{i,q^{\prime }\Sigma }(\indx )\right] \right) ^{2}%
\right] ^{1/2}+ \\
&&+\Ep\left[ \left( \tilde{q}^{\prime }\Sigma x_{i}z_{i}^{\prime }\Sigma
\left( \Ep\left[ z_{i}w_{i,q^{\prime }\Sigma }(\indx )\right] -\Ep\left[
z_{i}w_{i,\tilde{q}^{\prime }\Sigma }(\indx )\right] \right) \right) ^{2}%
\right] ^{1/2}+ \\
&&+\Ep\left[ \left( \tilde{q}^{\prime }\Sigma x_{i}z_{i}^{\prime }\Sigma
\left( \Ep\left[ z_{i}w_{i,\tilde{q}^{\prime }\Sigma }(\indx )\right] -\Ep%
\left[ z_{i}w_{i,\tilde{q}^{\prime }\Sigma }(\tilde{\indx})\right] \right)
\right) ^{2}\right] ^{1/2} \\
&\lesssim &\left\Vert q-\tilde{q}\right\Vert \norm{\Sigma} {\Ep[z_i'z_i
\max_{\ell \in \{0,1\}} \theta_\ell(x_i,\indx)^2]}^{1/2}+ \\
&&+\norm{\Sigma} {\Ep\left[
z_i'z_i\left(\theta_1(x_i,\indx)-\theta_0(x_i,\indx)\right)^2
{1}\{\left|(q-\tilde{q})'\Sigma z_i \right| \geq |q'\Sigma z_i|\} \right]}%
^{1/2}+ \\
&&+\norm{\Sigma} {\Ep\left[z_i'z_i\max_{\ell\in\{0,1\}}
\left(\theta_\ell\left(x_i,\indx\right) -
\theta_\ell\left(x_i,\tilde{\indx}\right)\right)^2\right]}^{1/2}
\end{eqnarray*}%
By assumption, $\Ep\left[ z_{i}^{\prime }z_{i}\theta _{\ell }(x_{i},\indx
)^{2}\right] \leq \left( \Ep\left[ \norm{z_i}^{4}\right] \Er\left[ \theta
_{\ell }(x_{i},\indx )^{4}\right] \right) ^{1/2}$ is bounded uniformly in $%
\indx $. Also, 
\begin{align*}
\Ep& \left[ z_{i}^{\prime }z_{i}\left( \theta _{1}(x_{i},\indx )-\theta
_{0}(x_{i},\indx )\right) ^{2}{1}\{|q^{\prime }\Sigma z_{i}|/\norm{z_i}\leq %
\norm{q-\tilde{q}}\}\right] \lesssim \\
& \,\,\,\,\lesssim \Ep\left[ \norm{z_i}^{4}\right] ^{1/2}\left( \Ep\left[
\theta _{1}(x_{i},\indx )^{4}\right] ^{1/2}+\Ep\left[ \theta _{0}(x_{i},%
\indx )^{4}\right] ^{1/2}\right) \Ep\left[ {1}\left\{ |q^{\prime }\Sigma
z_{i}|/\norm{z_i}\leq \norm{q-\tilde{q}}\right\} \right] ^{1/2} \\
& \,\,\,\,\lesssim \norm{q-\tilde{q}}^{m/2}
\end{align*}%
where we have used the smoothness condition (\ref{c:smooth}) and the fact
that $\Er[\norm{z_i}^4]<\infty $ and $\Er[\theta_\ell(x_i,\indx)^4]<\infty $
uniformly in $\indx $.

By assumption, $\theta _{\ell }(x,\indx )$ are H\"{o}lder continuous in $%
\indx $ with coefficient $L(x)$, so 
\begin{align*}
\Ep\left[ z_{i}^{\prime }z_{i}\max_{\ell \in \{0,1\}}\left( \theta _{\ell
}\left( x_{i},\indx \right) -\theta _{\ell }\left( x_{i},\tilde{\indx}%
\right) \right) ^{2}\right] ^{1/2}\lesssim & \Ep\left[ \norm{z_i}^{4}\right]
^{1/2}\Ep\left[ L(x_{i})^{4}\right] ^{1/2}\norm{\indx-\tilde{\indx}}^{\gamma
_{\theta }} \\
\lesssim & \norm{\indx-\tilde{\indx}}^{\gamma _{\theta }}
\end{align*}

Thus, 
\begin{align*}
\rho _{2}(h_{3}(t),h_{3}(\tilde{t}))\lesssim & \left\Vert q-\tilde{q}%
\right\Vert +\left\Vert q-\tilde{q}\right\Vert ^{m/2}+\left\Vert \indx-%
\tilde{\indx}\right\Vert ^{\gamma _{\theta }} \\
\lesssim & \left\Vert t-t^{\prime }\right\Vert ^{1\wedge m/2\wedge \gamma
_{\theta }}
\end{align*}

For $h_{4}$, we have 
\begin{align*}
\rho _{2}\left( h_{4}(t),h_{4}(\tilde{t})\right) =& \Ep\left[ \left(
q^{\prime }\Sigma z_{i}w_{i,q^{\prime }\Sigma }(\indx )-\tilde{q}^{\prime
}\Sigma z_{i}w_{i,\tilde{q}^{\prime }\Sigma }(\tilde{\indx})\right) ^{2}%
\right] ^{1/2} \\
\leq & \Ep\left[ \left( \left( q-\tilde{q}\right) ^{\prime }\Sigma
z_{i}w_{i,q^{\prime }\Sigma }(\indx )\right) ^{2}\right] ^{1/2}+ \\
& +\Ep\left[ \left( \tilde{q}^{\prime }\Sigma z_{i}\left( w_{i,q^{\prime
}\Sigma }(\indx )-w_{i,\tilde{q}^{\prime }\Sigma }(\indx )\right) \right)
^{2}\right] ^{1/2}+ \\
& +\Ep\left[ \left( \tilde{q}^{\prime }\Sigma z_{i}\left( w_{i,\tilde{q}%
^{\prime }\Sigma }(\indx )-w_{i,\tilde{q}^{\prime }\Sigma }(\tilde{\indx}%
)\right) \right) ^{2}\right] ^{1/2} \\
\lesssim & \norm{q-\tilde{q}}+\norm{q-\tilde{q}}^{m/2}+\norm{\indx-\tilde{%
\indx}}^{\gamma _{\theta }}
\end{align*}%
by the exact same arguments used for $h_{3}$.

\textbf{Claim 2.} It suffices to show that $\Er\left[ h_{j}(t)\right] $ for $%
j=1,...,4$ and $\Er\left[ h_{j}(t)h_{k}(t^{\prime })\right] $ for $j=1,...,4$
and $k=1,...,4$ are uniformly equicontinuous. H\"{o}lder continuity implies
equicontinuity, so we show that each of these functions are uniformly H\"{o}%
lder continuous.

Jensen's inequality and the result in Part 1 show that $\Er[h_j(t)]$ are
uniformly H\"{o}lder. 
\begin{equation*}
\left\vert \Er[h_j(t)]-\Er[h_j(t')]\right\vert \leq {\Er\left[ \left(
h_{j}(t)-h_{j}(t^{\prime })\right) ^{2}\right] }^{1/2}\lesssim \left\Vert
t-t^{\prime }\right\Vert ^{c}
\end{equation*}%
Given this, a simple calculation shows that $\Er\left[
h_{j}(t_{1})h_{k}(t_{2})\right] $ are uniformly H\"{o}lder as well. 
\begin{align*}
\left\vert \Er\left[ h_{j}(t_{1})h_{k}(t_{2})-h_{j}(t_{1}^{\prime
})h_{k}(t_{2}^{\prime })\right] \right\vert =& \left\vert \Er\left[ 
\begin{array}{c}
\left( h_{j}(t_{1})-h_{j}(t_{1}^{\prime })\right) h_{k}(t_{2})+ \\ 
+h_{j}(t_{1}^{\prime })\left( h_{k}(t_{2})-h_{k}(t_{2}^{\prime })\right)%
\end{array}%
\right] \right\vert \\
\leq & {\Er\left[ \left( h_{j}(t_{1})-h_{j}(t_{1}^{\prime })\right) ^{2}%
\right] }^{1/2}{\Er[h_k(t_2)^2]}^{1/2}+ \\
& +{\ \Er[h_j(t_1')^2]}^{1/2}{\Er\left[ \left(
h_{k}(t_{2})-h_{k}(t_{2}^{\prime })\right) ^{2}\right] }^{1/2} \\
\lesssim & \left\Vert t_{1}-t_{1}^{\prime }\right\Vert ^{c}\vee \left\Vert
t_{2}-t_{2}^{\prime }\right\Vert ^{c}
\end{align*}

\textbf{Claim 3.} By the law of total variance, 
\begin{equation*}
\text{var}(h(t))=\Ep\left[ \text{var}\left( h(t)|x_{i},z_{i}\right) \right] +%
\text{var}\left( \Ep\left[ h(t)|x_{i},z_{i}\right] \right) .
\end{equation*}%
Note that $h_{3}(t)$ and $h_{4}(t)$ are constant conditional on $x_{i},z_{i}$%
, so 
\begin{align}
\text{var}\left( h(t)|x_{i},z_{i}\right) =& \text{var}\left(
h_{1}(t)+h_{2}(t)|x_{i},z_{i}\right)  \notag \\
=& 
\begin{bmatrix}
q^{\prime }\Sigma \Er[z_{i}p_{i}'1\{q'\Sigma z_{i}>0\}]J_{1}^{-1}(\indx)p_{i}
\\ 
q^{\prime }\Sigma \Er[z_{i}p_{i}'1\{q'\Sigma z_{i}<0\}]J_{0}^{-1}(\indx)p_{i}%
\end{bmatrix}%
^{\prime }\text{var}\left( 
\begin{bmatrix}
\varphi _{i1}(\indx) \\ 
\varphi _{i0}(\indx)%
\end{bmatrix}%
|x_{i},z_{i}\right) \times  \notag \\
& \times 
\begin{bmatrix}
q^{\prime }\Sigma \Er[z_{i}p_{i}'1\{q'\Sigma z_{i}>0\}]J_{1}^{-1}(\indx)p_{i}
\\ 
q^{\prime }\Sigma \Er[z_{i}p_{i}'1\{q'\Sigma z_{i}<0\}]J_{0}^{-1}(\indx)p_{i}%
\end{bmatrix}
\label{eq:varstar}
\end{align}%
Recall that $b_{1}=\Er[z_{i}p_{i}'1\{q'\Sigma z_{i}>0\}]$ and $b_{0}=%
\Er[z_{i}p_{i}'1\{q'\Sigma z_{i}<0\}]$. Let $\gamma _{\ell }=q^{\prime
}\Sigma b_{\ell }$, and $\mineig(M)$ denote the minimal eigenvalue of any
matrix $M$. By assumption, 
\begin{equation*}
\mineig\left( \text{var}\left( 
\begin{bmatrix}
\varphi _{i1}(\indx) & \varphi _{i0}(\indx)%
\end{bmatrix}%
^{\prime }|x_{i},z_{i}\right) \right) >L,
\end{equation*}%
so 
\begin{align*}
\Er\left[ \text{var}\left( h(t)|x_{i},z_{i}\right) \right] \gtrsim & \Er%
\left[ \left\Vert 
\begin{bmatrix}
\gamma _{1}J_{1}^{-1}(\indx)p_{i} & \gamma _{0}J_{0}^{-1}(\indx)p_{i}%
\end{bmatrix}%
\right\Vert ^{2}\right] \\
\gtrsim & \Er\left[ \norm{\gamma_1 J_1^{-1}(\indx) p_i}^{2}\vee %
\norm{\gamma_0 J_0^{-1}(\indx) p_i}^{2}\right] \\
\gtrsim & \Er\left[ \norm{\gamma_1 J_1^{-1}(\indx) p_i}^{2}\right] \vee \Er%
\left[ \norm{\gamma_0 J_0^{-1}(\indx) p_i}^{2}\right]
\end{align*}%
Repeated use of the inequality $\norm{xy}^{2}\geq \mineig(yy^{\prime })%
\norm{x}^{2}$ yields for $l=0,1$, 
\begin{align*}
\Er\left[ \norm{\gamma_\ell J_\ell^{-1}(\indx) p_i}^{2}\right] & \geq \mineig%
\left( \Er\left[ p_{i}p_{i}^{\prime }\right] \right) \mineig\left( J_{\ell
}^{-1}(\indx)\right) ^{2}\norm{\gamma_\ell}^{2} \\
& \gtrsim \mineig(b_{\ell }^{\prime }b_{\ell })\norm{q'\Sigma}^{2} \\
& \gtrsim b_{\ell }^{\prime }b_{\ell }
\end{align*}%
where the last line follows from the fact that $b_{\ell }^{\prime }b_{\ell }$
is a scalar. We now show that $b_{\ell }^{\prime }b_{\ell }>0$. Let $%
f_{1i}=z_{i}1\{q^{\prime }\Sigma z_{i}>0\}$ and $f_{0i}=z_{i}1\{q^{\prime
}\Sigma z_{i}<0\}$. Observe that $z_{i}=f_{1i}+f_{0i}$ and $\Er\left[
f_{1i}^{\prime }f_{0i}\right] =0$, so 
\begin{equation*}
\Er\left[ f_{1i}^{\prime }f_{1i}\right] \vee \Er\left[ f_{0i}^{\prime }f_{0i}%
\right] \geq \frac{1}{2}\Er\left[ z_{i}^{\prime }z_{i}\right] >0
\end{equation*}%
By the completeness of our series functions, we can represent $f_{1i}$ and $%
f_{0i}$ in terms of the series functions. Let 
\begin{equation*}
f_{1i}=\sum_{j=1}^{\infty }c_{1j}p_{ji}f_{0i}=\sum_{j=1}^{\infty
}c_{0j}p_{ji}
\end{equation*}%
Without loss of generality, assume the series functions are orthonormal.
Then 
\begin{equation*}
\Er\left[ f_{1i}^{\prime }f_{1i}\right] =\sum_{j=1}^{\infty }c_{1j}^{2}\Er%
\left[ f_{0i}^{\prime }f_{0i}\right] =\sum_{j=1}^{\infty }c_{0j}^{2}
\end{equation*}%
Also, 
\begin{equation*}
b_{\ell }^{\prime }b_{\ell }=\sum_{j=1}^{k}c_{\ell j}^{2}
\end{equation*}%
Thus, 
\begin{equation*}
\Er\left[ \text{var}\left( h(t)|x_{i},z_{i}\right) \right] \gtrsim \mineig%
(b_{1}^{\prime }b_{1})\vee \mineig(b_{0}^{\prime }b_{0})>0
\end{equation*}
\end{proof}

\subsection{Conservative Inference with Discrete Covariates}

Let $\Theta (x,\indx)=\left[ \theta _{0}(x,\indx),\theta _{1}(x,\indx)\right]
,$ and to simplify notation suppress the dependence of $\Theta $ and $\theta
_{\ell }$ on $(x,\indx)$ and let the instruments coincide with $x=\left[
x_{1}\text{ }x_{2}\right] ^{\prime },$ with $x_{1}=1$ and $x_{2}\in \mathbb{R%
}^{d-1}$. Let $\Sigma =$\textrm{E}$\left( xx^{\prime }\right) ^{-1},$ $%
z=x+\sigma \left[ 0\text{ }\eta \right] ^{\prime },$ with $\eta \sim N\left(
0,I\right) $ and independent of $x$ and $\theta _{\ell },$ $\ell =0,1,$
where $I$ denotes the identity matrix$.$ Note that \textrm{E}$\left(
xx^{\prime }\right) =$\textrm{E}$\left( zx^{\prime }\right) ,$ and define%
\begin{equation*}
B=\Sigma \mathbf{E}\left( x\Theta \right) ,\text{ \ \ \ \ \ }\tilde{B}%
=\Sigma \mathbf{E}\left( z\Theta \right) ,
\end{equation*}%
where $\mathbf{E}\left( \cdot \right) $ denotes the Aumann expectation of
the random set in parenthesis, see \citeasnoun[Chapter 2]{Molchanov05}.
Denote by $\widehat{\tilde{B}}$ the estimator of $\tilde{B}$ (the unique
convex set corresponding to the estimated support function) and by $\mathcal{%
B}_{\hat{c}_{n(1-\tau )}}$ a ball in $\mathbb{R}^{d}$ centered at zero and
with radius $\hat{c}_{n(1-\tau )},$ with $\hat{c}_{n(1-\tau )}$ the Bayesian
bootstrap estimate of the $1-\tau $ quantile of $f\left( \mathbb{G}\left[
h_{k}\left( t\right) \right] \right) ,$ with $f(s(t))=\sup_{t\in T}\left\{
-s\left( t\right) \right\} _{+}$, see Section \ref{subsec:theoret_results}.
Following arguments in \citeasnoun[Section 2.3]{molinari_beresteanu_2008},
one can construct a (convex) confidence set $CS_{n}$ such that $\sup_{\indx
\in \indxSet}\left( \sigma _{\widehat{\tilde{B}}}(q,\indx )-\sigma
_{CS_{n}}(q)\right) =\hat{c}_{n(1-\tau )}$ for all $q\in \mathcal{S}^{d-1},$
where $\sigma _{A}\left( \cdot \right) $ denotes the support function of the
set $A.$ It then follows that%
\begin{equation*}
\lim_{n\rightarrow \infty }\Pr \left( \sup_{q,\indx \in \mathcal{S}%
^{d-1}\times \indxSet}\left\vert \sigma _{\tilde{B}}(q,\indx )-\sigma
_{CS_{n}}(q)\right\vert _{+}=0\right) =1-\tau .
\end{equation*}

\begin{lemma}
\label{prop:approx}For a given $\delta >0,$ one can jitter $x$ via $%
z=x+\sigma _{\delta }\left[ 0\text{ }\eta \right] ^{\prime },$ so as to
obtain a set $\tilde{B}$ such that $\sup_{\indx \in \indxSet}\rho
_{H}\left( \tilde{B},B\right) \leq \delta $ and 
\begin{equation}
1-\tau -\gamma (\delta )\geq \lim_{n\rightarrow \infty }\Pr \left(
\sup_{q,\indx \in \mathcal{S}^{d-1}\times \indxSet}\left\vert \sigma
_{B}(q,\indx )-\left( \sigma _{CS_{n}}(q)+\delta \right) \right\vert
_{+}=0\right) \geq 1-\tau ,  \label{eq:cons_inference}
\end{equation}%
where $\gamma (\delta )=\Pr \left( \sup_{t\in T}\left\{ -\mathbb{G}\left[
h_{k}\left( t\right) \right] \right\} _{+}>2\delta \right) .$
\end{lemma}

\begin{proof}
Observe that $\rho _{H}\left( \tilde{B},B\right) =\rho _{H}\left( \Sigma 
\mathbf{E}\left( z\Theta \right) ,\Sigma \mathbf{E}\left( x\Theta \right)
\right) .$ By the properties of the Aumann expectation (see, e.g., %
\citeasnoun[Theorem 2.1.17]{Molchanov05}),%
\begin{equation*}
\rho _{H}\left( \Sigma \mathbf{E}\left( z\Theta \right) ,\Sigma \mathbf{E}%
\left( x\Theta \right) \right) \leq \text{\textrm{E}}\left[ \rho _{H}\left(
\Sigma \left( z\Theta \right) ,\Sigma \left( x\Theta \right) \right) \right]
.
\end{equation*}%
In turn,%
\begin{eqnarray*}
&&\sup_{\indx \in \indxSet}\text{\textrm{E}}\left[ \rho _{H}\left(
\Sigma \left( z\Theta \right) ,\Sigma \left( x\Theta \right) \right) \right] 
\\
&=&\sup_{\indx \in \indxSet}\text{\textrm{E}}\left[ \sup_{v=\Sigma
^{\prime }q:\left\Vert v\right\Vert =1}\left\vert \sup_{\theta \in \Theta
}\left( v_{1}+z_{2}v_{2}\right) \tilde{\theta}-\sup_{\theta \in \Theta
}\left( v_{1}+x_{2}v_{2}\right) \theta \right\vert \right]  \\
&=&\sup_{\indx \in \indxSet}\text{\textrm{E}}\left[ \sup_{v=\Sigma
^{\prime }q:\left\Vert v\right\Vert =1}\left\vert \underset{}{\left(
v_{1}+x_{2}v_{2}+\sigma \eta v_{2}\right) }\left( \theta _{0}1\left(
v_{1}+x_{2}v_{2}+\sigma \eta v_{2}<0\right) +\theta _{1}1\left(
v_{1}+x_{2}v_{2}+\sigma \eta v_{2}>0\right) \right) \right. \right.  \\
&&\left. \left. -\underset{}{\left( v_{1}+x_{2}v_{2}\right) }\left( \theta
_{0}1\left( v_{1}+x_{2}v_{2}<0\right) +\theta _{1}1\left(
v_{1}+x_{2}v_{2}>0\right) \right) \right\vert \right]  \\
&\leq &\sup_{\indx \in \indxSet}\text{\textrm{E}}\left[ \sup_{v=\Sigma
^{\prime }q:\left\Vert v\right\Vert =1}\left\vert \sigma \eta v_{2}\left(
\theta _{0}1\left( v_{1}+x_{2}v_{2}+\sigma \eta v_{2}<0\right) +\theta
_{1}1\left( v_{1}+x_{2}v_{2}+\sigma \eta v_{2}>0\right) \right) \right\vert %
\right]  \\
&&+\sup_{\indx \in \indxSet}\text{\textrm{E}}\left[ \sup_{v=\Sigma
^{\prime }q:\left\Vert v\right\Vert =1}\left\vert \left(
v_{1}+x_{2}v_{2}\right) \left( \theta _{1}-\theta _{0}\right) \left( 1\left(
0<-\left( v_{1}+x_{2}v_{2}\right) <\sigma \eta v_{2}\right) -1\left(
0<v_{1}+x_{2}v_{2}<-\sigma \eta v_{2}\right) \right) \right\vert \right]  \\
&\leq &\sigma \text{\textrm{E}}\left\vert \eta \right\vert \left(
\sup_{\indx \in \indxSet}\text{\textrm{E}}\left\vert \theta _{0}(x,\indx%
)\right\vert +\sup_{\indx \in \indxSet}\text{\textrm{E}}\left\vert
\theta _{1}(x,\indx)\right\vert +\sup_{\indx \in \indxSet}\text{\textrm{E%
}}\left\vert \theta _{1}(x,\indx)-\theta _{0}(x,\indx)\right\vert \right) .
\end{eqnarray*}

Hence, we can choose $\sigma _{\delta }=\frac{\delta }{\text{\textrm{E}}%
\left\vert \eta \right\vert \left( \sup_{\indx \in \indxSet}\text{%
\textrm{E}}\left\vert \theta _{0}(x,\indx)\right\vert +\sup_{\indx \in 
\indxSet}\text{\textrm{E}}\left\vert \theta _{1}(x,\indx)\right\vert
+\sup_{\indx \in \indxSet}\text{\textrm{E}}\left\vert \theta _{1}(x,\indx%
)-\theta _{0}(x,\indx)\right\vert \right) }.$

Now observe that because $\sup_{\indx \in \indxSet}\rho _{H}\left( 
\tilde{B},B\right) \leq \delta ,$ we have $B(\indx )\subseteq \tilde{B}%
(\indx )\oplus \mathcal{B}_{\delta }$ for all $\indx \in \indxSet,$
where "$\oplus $" denotes Minkowski set summation, and therefore%
\begin{eqnarray*}
\sup_{\indx \in \indxSet}\left( \sigma _{\tilde{B}}(q,\indx )-\sigma
_{CS_{n}}(q)\right)  &\leq &0\text{ }\forall \text{ }q\in \mathcal{S}^{d-1}
\\
&\Longrightarrow &\sup_{\indx \in \indxSet}\left( \sigma _{B}(q,\indx
)-\left( \sigma _{CS_{n}}(q)+\delta \right) \right) \leq 0\text{ }\forall 
\text{ }q\in \mathcal{S}^{d-1},
\end{eqnarray*}%
from which the second inequality in (\ref{eq:cons_inference}) follows.
Notice also that $\tilde{B}(\indx )\subseteq B(\indx )\oplus \mathcal{B}%
_{\delta }$ for all $\indx \in \indxSet,$ and therefore%
\begin{eqnarray*}
\sup_{\indx \in \indxSet}\left( \sigma _{B}(q,\indx )-\left( \sigma
_{CS_{n}}(q)+\delta \right) \right)  &\leq &0\text{ }\forall \text{ }q\in 
\mathcal{S}^{d-1} \\
&\Longrightarrow &\sup_{\indx \in \indxSet}\left( \sigma _{\tilde{B}%
}(q,\indx )-\left( \sigma _{CS_{n}}(q)+2\delta \right) \right) \leq 0\text{ 
}\forall \text{ }q\in \mathcal{S}^{d-1},
\end{eqnarray*}%
from which the first inequality in (\ref{eq:cons_inference}) follows.
Because $\delta >0$ is chosen by the researcher, inference is arbitrarily
slightly conservative. Note that a similar argument applies if one uses a
Kolmogorov statistic rather than a directed Kolmogorov statistic. Moreover,
the Hausdorff distance among convex compact sets is larger than the $L_{p}$
distance among them (see, e.g., \citeasnoun[Theorem 1]{Vitale85}), and
therefore a similar conclusion applies for Cramer-Von-Mises statistics.
\end{proof}

\subsection{Lemmas on Entropy Bounds}

We collect frequently used facts in the following lemma.

\begin{lemma}
\label{lemma: andrews} Let $Q$ be any probability measure whose support
concentrates on a finite set.

\begin{enumerate}
\item Let $\F$ be a measurable VC class with a finite VC index $k$ or any
other class whose entropy is bounded above by that of such a VC class, then
its entropy obeys 
\begin{equation*}
\log N(\epsilon \|F\|_{Q,2}, \F, L^2(Q)) \lesssim 1+ k \log (1/\epsilon) 
\newline
\end{equation*}
Examples include e.g., linear functions $\F=\{ \indx^{\prime }w_i, \indx
\in \mathbb{R}^{k}, \| \indx\| \leq C\}$ and their indicators $\F=\{ 1\{\indx%
^{\prime }w_i> 0\}, \indx \in \mathbb{R}^{k}, \| \indx \| \leq C\} $. 
\newline

\item Entropies obey the following rules for sets created by addition,
multiplication, and unions of measurable function sets $\F$ and $\F^{\prime
} $: 
\begin{eqnarray*}
&&\log N(\epsilon \Vert F+F^{\prime }\Vert _{Q,2},\F+\F^{\prime
},L^{2}(Q))\leq B \\
&&\log N(\epsilon \Vert F\cdot F^{\prime }\Vert _{Q,2},\F\cdot \F^{\prime
},L^{2}(Q))\leq B \\
&&\log N(\epsilon \Vert F\vee F^{\prime }\Vert _{Q,2},\F\cup \F^{\prime
},L^{2}(Q))\leq B \\
&&B=\log N\left( \frac{\epsilon }{2}\Vert F\Vert _{Q,2},\F,L^{2}(Q)\right)
+\log N\left( \frac{\epsilon }{2}\Vert F^{\prime }\Vert _{Q,2},\F^{\prime
},L^{2}(Q)\right) .
\end{eqnarray*}

\item Entropies are preserved by multiplying a measurable function class $%
\mathcal{F}$ with a random variable $g_{i}$: 
\begin{equation*}
\log N(\epsilon \Vert |g|F\Vert _{Q,2},g\F,L^{2}(Q))\lesssim \log N\left(
\epsilon /2\Vert F\Vert _{Q,2},\F,L^{2}(Q)\right)
\end{equation*}

\item Entropies are preserved by integration or taking expectation: for $%
f^{\ast }(x):=\int f(x,y)d\mu (y)$ where $\mu $ is some probability measure, 
\begin{equation*}
\log N(\epsilon \Vert F\Vert _{Q,2},\F^{\ast },L^{2}(Q))\leq \log N\left(
\epsilon \Vert F\Vert _{Q,2},\F,L^{2}(Q)\right)
\end{equation*}
\end{enumerate}
\end{lemma}

\textbf{Proof.} For the proof of (1)-(3) see e.g., \citeasnoun{andrews:emp}.
For the proof of (4), see e.g., \citeasnoun[Lemma A2]{Vaart00}. \qed

Next consider function classes and their envelops 
\begin{eqnarray}
\mathcal{H}{_{1}} &=&\{q^{\prime }\Sigma \Er[z_i p_i' 1\{q'\Sigma z_i <0\}]%
J_{0}^{-1}(\indx)p_{i}\varphi _{i0}(\indx),t\in T\},\text{ \ }H_{1}\lesssim
\Vert z_{i}\Vert \xi _{k}F_{1}  \notag \\
\mathcal{H}{_{2}} &=&\{q^{\prime }\Sigma \Er[z_i p_i' 1\{q '\Sigma z_i >0\}]%
J_{1}^{-1}(\indx)p_{i}\varphi _{i1}(\indx),t\in T\},\text{ \ }H_{2}\lesssim
\Vert z_{i}\Vert \xi _{k}F_{1}  \notag \\
\mathcal{H}{_{3}} &=&\{q^{\prime }\Sigma x_{i}z_{i}^{\prime }\Sigma \Ep\left[
z_{i}w_{i,q^{\prime }\Sigma }(\indx)\right] ,t\in T\},\text{ \ }%
H_{3}\lesssim \Vert x_{i}\Vert \Vert z_{i}\Vert  \notag \\
\mathcal{H}{_{4}} &=&\{q^{\prime }\Sigma z_{i}w_{i,q^{\prime }\Sigma }(\indx%
),t\in T\},\text{ \ }H_{4}\lesssim \Vert z_{i}\Vert F_{2}  \notag \\
\F_{3} &=&\{\bar{\mu}^{\prime }J^{-1}(\indx)p_{i}\varphi _{i}(\indx),\text{ }%
\indx\in \indxSet\},\ \ F_{3}\lesssim \xi _{k}F_{1},  \label{eq:F_3}
\end{eqnarray}%
where $\bar{\mu}^{\prime }$ is defined in equation (\ref{eq:mu_bar}).

\begin{lemma}
\label{lemma: complexities} 1. (a) The following bounds on the empirical
entropy apply 
\begin{eqnarray*}
&&\log N(\epsilon \Vert H_{1}\Vert _{\Pn,2},\mathcal{H}_{1},L^{2}(\Pn%
))\lesssim _{\Pr }\log n+\log (1/\epsilon ) \\
&&\log N(\epsilon \Vert H_{2}\Vert _{\Pn,2},\mathcal{H}_{2},L^{2}(\Pn%
))\lesssim _{\Pr }\log n+\log (1/\epsilon ) \\
&&\log N(\epsilon \Vert F_{3}\Vert _{\Pn,2},\F_{3},L^{2}(\Pn))\lesssim _{\Pr
}\log n+\log (1/\epsilon )
\end{eqnarray*}%
(b) Moreover similar bounds apply to function classes $g_{i}(\mathcal{H}%
_{l}^{o}-\mathcal{H}_{l}^{o})$ with the envelopes given by $|g_{i}|4H\ell $,
where $g_{i}$ is a random variable.

2. (a) The following bounds on the uniform entropy apply 
\begin{eqnarray*}
&&\sup_{Q}\log N(\epsilon \Vert H_{1}\Vert _{Q,2},\mathcal{H}%
_{1},L^{2}(Q))\lesssim k\log (1/\epsilon ) \\
&&\sup_{Q}\log N(\epsilon \Vert H_{2}\Vert _{Q,2},\mathcal{H}%
_{2},L^{2}(Q))\lesssim k\log (1/\epsilon ) \\
&&\sup_{Q}\log N(\epsilon \Vert F_{3}\Vert _{Q,2},\F_{3},L^{2}(Q))\lesssim
k\log (1/\epsilon ) \\
&&\sup_{Q}\log N(\epsilon \Vert H_{3}\Vert _{Q,2},\mathcal{H}%
_{3},L^{2}(Q))\lesssim \log (1/\epsilon ) \\
&&\sup_{Q}\log N(\epsilon \Vert H_{4}\Vert _{Q,2},\mathcal{H}%
_{4},L^{2}(Q))\lesssim \log (1/\epsilon ).
\end{eqnarray*}%
(b) Moreover similar bounds apply to function classes $g_{i}(\mathcal{H}%
_{l}^{o}-\mathcal{H}_{l}^{o})$ with the envelopes given by $|g_{i}|4H\ell $,
where $g_{i}$ is a random variable.
\end{lemma}

\textbf{Proof.} Part 1 (a). Case of $\mathcal{H}_{1}$ and $\mathcal{H}_{2}$.
We shall detail the proof for this case, while providing shorter arguments
for others, as they are simpler or similar.

Note that $\mathcal{H}{_{1}}\subseteq \M_{1}\cdot \M_{2}\cdot \F_{1}$, where 
$\M_{1}=\{q^{\prime }\Sigma z_{i},q\in \mathcal{S}^{d-1}\}$ with envelope $%
M_{1}=\Vert z_{i}\Vert $ is VC with index $\dim (z_{i})+\dim (x_{i})$, and $%
\mathcal{M}_{2}=\{\gamma (q)J_{0}^{-1}(\indx)p_{i},(q,\indx )\in \mathcal{S}%
^{d-1}\times \indxSet\}$ with envelope $M_{2}\lesssim \Vert \xi _{k}\Vert 
$, $\F_{1}=\{\varphi _{i0}(\indx),\indx \in \indxSet\}$ with envelope $%
F_{1}$, where $\gamma (q)$ is uniformly Holder in $q\in \mathcal{S}^{d-1}$
by Lemma \ref{lemma: derivative}. Elementary bounds yield 
\begin{eqnarray*}
&&\Vert m_{2}(t)-m_{2}(\tilde{t})\Vert _{\Pn,2}\leq L_{1n}\Vert \indx-\tilde{%
\indx}\Vert +L_{2n}\Vert q-\tilde{q}\Vert , \\
&&L_{1n}\lesssim \sup_{\indx\in \indxSet}\Vert J^{-1}(\indx)\Vert \Vert \xi
_{k}\Vert \ \ L_{2n}\lesssim \Vert \En[p_i p_i']\Vert , \\
&&\log L_{1n}\lesssim _{\Pr }\log n\ \text{ and }\ \log L_{2n}\lesssim _{\Pr
}1.
\end{eqnarray*}%
Note that $\log \xi _{k}\lesssim \log n$ by assumption, $\sup_{\indx\in %
\indxSet}\Vert J^{-1}(\indx)\Vert \lesssim 1$ by assumption, $\Vert 
\En[p_i
p_i']\Vert \lesssim _{\Pr }1$ by Lemma \ref{lemma: rudelson}. The sets $%
\mathcal{S}^{d-1}$ and $\indxSet$ are compact subsets of Euclidian space of
fixed dimension, and so can be covered by a constant times $1/\epsilon ^{c}$
balls of radius $\epsilon $ for some constant $c>0$. Therefore, we can
conclude 
\begin{equation*}
\log N(\epsilon \Vert M_{2}\Vert _{\Pn,2},\M_{2},L_{2}(\Pn))\lesssim _{\Pr
}\log n+\log (1/\epsilon ).
\end{equation*}%
Repeated application of Lemma \ref{lemma: andrews} yields the conclusion,
given the assumption on the function class $\F_{1}$. The case for $\mathcal{H%
}{_{2}}$ is very similar.

Case of $\F_{3}$. Note that $\F_{3}\subset \mathcal{M}_{2}\cdot \F_{1}$ and $%
\Vert \bar{\mu}\Vert =o_{\Pr }(1)$ by Step 4 in the proof of Lemma \ref{VC
lemma: linearization}. Repeated application of Lemma \ref{lemma: andrews}
yields the conclusion, given the assumption on the function class $\F_{1}$.

Part 1 (b). Note that $\mathcal{H}^o = \mathcal{H}- \Ep[\mathcal{H}^o]$, so
it is created by integration and summation. Hence repeated application of
Lemma \ref{lemma: andrews} yields the conclusion.

Part 2. (a) Case of $\mathcal{H}_{1},\mathcal{H}_{2}$, and $\F_{3}$. Note
that all of these classes are subsets of $\{\mu ^{\prime }p_{i},\Vert \mu
\Vert \leq C\}\cdot \mathcal{F}_{1}$ with envelope $\xi _{k}F_{1}$. The
claim follows from repeated application of Lemma \ref{lemma: andrews}.

Case of $\mathcal{H}_{3}$. Note that $\mathcal{H}_{3}\subset \{q^{\prime
}\Sigma x_{i}z_{i}^{\prime }\mu ,\Vert \mu \Vert \leq C\}$ with envelope $%
\Vert x_{i}\Vert \Vert z_{i}\Vert $. The claim follows from repeated
application of Lemma \ref{lemma: andrews}.

Case of $\mathcal{H}_{4}$. Note that $\mathcal{H}_{4}$ is a subset of a
function class created from taking the class $\F_{2}$ multiplying it with
indicator function class $1\{q^{\prime }\Sigma z_{i}>0,q\in \mathcal{S}%
^{d-1}\}$ and with function class $\{q^{\prime }\Sigma z_{i},q\in \mathcal{S}%
^{d-1}\}$ and then adding the resulting class to itself. The claim follows
from repeated application of Lemma \ref{lemma: andrews}.

Part 2 (b). Note that $\mathcal{H}^o = \mathcal{H}- \Ep[\mathcal{H}^o]$, so
it is created by integration and summation. Hence repeated application of
Lemma \ref{lemma: andrews} yields the conclusion.

\qed



\subsection{Auxiliary Maximal and Random Matrix Inequalities}

We repeatedly use the following matrix LLN.


\begin{lemma}[Matrix LLN]
\label{lemma: rudelson} Let $Q_{1},...,Q_{n}$ be i.i.d. symmetric
non-negative matrices such that $Q=\mathrm{E}Q_{i}$ and $\Vert Q_{i}\Vert
\leq M$, then for $\hat{Q}=\En Q_{i}$ 
\begin{equation*}
\mathrm{E}\Vert \hat{Q}-Q\Vert \lesssim \sqrt{\frac{M(1+\Vert Q\Vert )\log k
}{n}}.
\end{equation*}%
In particular, if $Q_{i}=p_{i}p_{i}^{\prime }$, with $\Vert p_{i}\Vert \leq
\xi _{k}$, then 
\begin{equation*}
\mathrm{E}\Vert \hat{Q}-Q\Vert \lesssim \sqrt{\frac{\xi _{k}^{2}(1+\Vert
Q\Vert )\log k}{n}}.
\end{equation*}
\end{lemma}

\noindent \textbf{Proof.} 
This is a variant of a result from \citeasnoun{Rudelson99}. By the
symmetrization lemma,
\begin{align*}
  \Delta := \Er\norm{\hat{Q} - Q} \leq 2 \Er \Er_{\epsilon} \norm{
    \En[ \epsilon_i Q_i ] } 
\end{align*}
where $\epsilon_i$ are Rademacher random variables. The Khintchine
inequality for matrices, which was shown by \citeasnoun{Rudelson99} to
follow from the results of \citeasnoun{LustPiquard91}, states that
\[ \Er_{\epsilon} \norm{\En[\epsilon_i Q_i]} \lesssim \sqrt{ \frac{\log
  k}{n}} \norm{ \left(\En[Q_i^2] \right)^{1/2} }. \]
Since (remember that $\norm{\cdot}$ is the operator norm)
\begin{align*}
  \Er \norm{ \left(\En[Q_i^2]\right)^{1/2} } = \Er \norm{
    \left(\En[Q_i^2]\right) }^{1/2} \leq \left[ M \Er\norm{\En Q_i}
  \right]^{1/2},
\end{align*}
and 
\[ \norm{\En Q_i} \leq \Delta + \norm{Q}, \]
one has
\begin{align*}
  \Delta \leq 2\sqrt{\frac{M \log k}{n}} \left[\Delta + \norm{Q}
  \right]^{1/2}.
\end{align*}
Solving for $\Delta$ gives
\begin{align*}
  \Delta \leq \sqrt{\frac{4 M\norm{Q} \log k}{n} + \left(\frac{M\log
        k}{n}\right)^2} + \frac{M \log k}{n},
\end{align*}
which implies the result stated in the lemma if $\frac{M \log k}{n} <
1$. 
\qed

We also use the following maximal inequality.

\begin{lemma}
\label{lemma: maximal} Consider a separable empirical process $\mathbb{G}%
_{n}(f)=n^{-1/2}\sum_{i=1}^{n}\{f(Z_{i})-\Er[f(Z_i)]\}$, where $Z_{1},\ldots
,Z_{n}$ is an underlying independent data sequence on the space $(\Omega ,%
\mathcal{G},\Pr )$, defined over the function class $\mathcal{F}$, with an
envelope function $F\geq 1$ such that $\log [\max_{i\leq n}\Vert F\Vert
]\lesssim _{\Pr }\log n$ and 
\begin{equation*}
\log N\left( \varepsilon \left\Vert F\right\Vert _{\Pn,2},\mathcal{F},L_{2}(%
\Pn)\right) \leq \upsilon m\log (\kappa /\epsilon ),\ 0<\epsilon <1,
\end{equation*}%
with some constants $0<\log \kappa \lesssim \log n$, $m$ potentially
depending on $n$, and $1<\upsilon \lesssim 1$. For any $\delta \in (0,1)$,
there is a large enough constant $K_{\delta }$, such that for $n$
sufficiently large, then 
\begin{equation*}
\Pr \left\{ \sup_{f\in \mathcal{F}}|\mathbb{G}_{n}(f)|\leq K_{\delta }\sqrt{%
m\log n}\max \left\{ \sup_{i\leq n,f\in \mathcal{F}}\Vert f(Z_{i})\Vert
_{\Pr ,2},\ \sup_{f\in \mathcal{F}}\Vert f\Vert _{\mathbb{P}_{n},2}\right\}
\right\} \geq 1-\delta .
\end{equation*}
\end{lemma}

\noindent \textbf{Proof.} TO\ BE\ ADDED. This is a restatement of Lemma 19 from
\citeasnoun{BelloniChernozhkov2009Lasso}.\qed

\clearpage

\nocite{Manski07}

\bibliographystyle{econometrica}
\bibliography{set}

\end{document}